\RequirePackage{fix-cm}
\documentclass[11pt]{article}

\usepackage{mypackages}
\usepackage{makeidx}

\usepackage[T1]{fontenc}
\usepackage{titlesec}

\usepackage{type1cm}

\title{Non-conservation of a generalized helicity in the Euler equations}

\author{Vikram Giri, Hyunju Kwon, and Matthew Novack}

\date{}
\makeindex

\allowdisplaybreaks

\usepackage{tkz-euclide}
\definecolor{ao(english)}{rgb}{0.0, 0.5, 0.0}
\definecolor{applegreen}{rgb}{0.55, 0.71, 0.0}
\definecolor{darkpastelgreen}{rgb}{0.01, 0.75, 0.24}
\definecolor{azure}{rgb}{0.0, 0.5, 1.0}
\definecolor{burgundy}{rgb}{0.5, 0.0, 0.13}
\definecolor{battleshipgrey}{rgb}{0.52, 0.52, 0.51}
\definecolor{bluey}{rgb}{.15, 0, 5}
\definecolor{alizarin}{rgb}{0.82, 0.1, 0.26}
\definecolor{brightpink}{rgb}{1.0, 0.0, 0.5}

\newcommand{\qho}{{q*}}
\newcommand{\bp}{\mathbb{P}}
\newcommand{\mcf}{\mathcal{F}}
\newcommand{\ddt}{\left[ \sfrac{d}{dt} \right]}

\titleformat{\subsection}[runin]
      {\normalfont\bfseries}
      {\thesubsection}
      {0.5em}
      {}
      [.]
\titleformat{\subsubsection}[runin]
      {\normalfont\bfseries}
      {\thesubsubsection}
      {0.5em}
      {}
      [.]

\begin{document}

\maketitle

\begin{abstract}
For a $C^1_{t,x}$ solution $u$ to the incompressible 3D Euler equations, the helicity $H(u(t))=\int_{\mathbb{T}^3} u \cdot \curl u$ is constant in time.  For general low-regularity weak solutions, it is not always clear how to define the helicity, or whether it must be constant in time in the case that there is a clear definition.  In this paper, we define a generalized helicity which extends the classical definitions and construct weak solutions of Euler of almost Onsager-critical regularity in $L^3$ with prescribed generalized helicity and kinetic energy.   
\end{abstract}

\setcounter{tocdepth}{1}
\tableofcontents

\section{Introduction}
Given a smooth vector field $v : \T^3 \to \R^3$, its \emph{helicity} $H(v)$ is defined by
\begin{equation}\label{def:hel}
    H(v) := \int_{\T^3} v(x) \cdot \curl v(x)\,dx\,.
\end{equation}
Note that the sign of the helicity flips upon reflection of the vector field in a plane; thus the helicity can be regarded as a measure of reflection non-invariance of a vector field.  Moreau~\cite{Moreau} and Moffatt~\cite{Moffatt} noticed that helicity is a conserved quantity for the three-dimensional incompressible Euler equations
\begin{equation}\label{eqn:Euler}
    \begin{cases}
    \pa_t u + \div (u \otimes u) + \na p = 0 \\
    \div \, u =0 
    \end{cases}
\end{equation}
posed on $\T^3= \R^3 / \Z^3$. 
That is, for a smooth solution $(u,p)$ of~\eqref{eqn:Euler}, we have for all $0\leq t \leq T$ that
\begin{equation}\label{eqn:hel_cons}
    H(u(t)) = H(u(0))\,.
\end{equation}

\subsection{Motivation and background} 
The helicity of a vector field has a topological interpretation. To see this, let the vorticity $\omega = \curl u$ be composed of two simple, linked vortex filaments of strengths $\kappa_1$ and $\kappa_2$. More precisely, there exist two smooth, simple, null-homologous, unknotted closed curves $\ga_1, \ga_2$ that are linked, $\omega$ is a measure supported on the $\ga_i$'s with mass $\kappa_i$ on $\gamma_i$, and $\omega(x)$ is the unit vector tangent to $\ga_i$ at $x$. Let $\Sigma_i$ be a surface such that $\pa \Sigma_i = \ga_i$. Following Moffat~\cite{Moffatt}, we can compute the helicity as
\begin{align*}
    H(u) = \int_{\T^3} u \cdot \curl u = \sum_{i=1}^2 \int_{\ga_i} u \cdot \dot \ga_i = \sum_{i=1}^2 \int_{\Sigma_i} \curl u = \pm \kappa_1 \kappa_2 \, \times  \, \textnormal{linking number}(\ga_1, \ga_2)\,.
\end{align*}
The $\pm$ arises due to the two possible relative orientations of the filaments. Thus the helicity of a vector field is related to a topological linking property of the vortex lines.  This point of view was clarified by Arnold~\cite{Arnold}, who showed that the helicity of a smooth vector field is equal to an averaged, asymtotic linking number of trajectories of its vorticity. 

Furthermore, the helicity obeys a stronger invariance unrelated to the Euler equations: namely, for any volume-preserving diffeomorphism $\Phi$ and exact vector field $\omega$, the push-forward vector field $\Phi_*\omega = (\na\Phi)^{-1} \omega\circ\Phi$ satisfies
\begin{equation}\label{eqn:hel_cons_gen}
    H(\curl^{-1} \Phi_*\omega) = H(\curl^{-1}\omega)\,.
\end{equation}
Note that the push-forward of an exact vector field is exact, and so the $\curl^{-1}$ in the above equation makes sense. This shows that the helicity is an integral invariant of exact vector fields $\omega$ under volume-preserving diffeomorphisms; in fact it is the only such invariant~\cite{AK98, Ku09, Enciso-et-al-16, Serre84, SerreSurvey}. The helicity conservation for smooth solutions of the Euler equations~\eqref{eqn:hel_cons} follows easily from this more general invariance~\eqref{eqn:hel_cons_gen}: indeed, the vorticity equation is precisely the statement that the vorticity evolves as the push-forward of the volume-preserving flow of $u$. An extension of helicity to $C^0$ Hamiltonian structures was recently obtained in~\cite{ES25} where the authors prove a version of the invariance~\eqref{eqn:hel_cons_gen} for volume-preserving \emph{homeomorphisms}, answering a question of Arnold, in a special case.

However, in turbulent flows exhibiting anomalous dissipation of kinetic energy, the regularity of the vector field $u$ is insufficient for the theory of the transport equation to give us a well-behaved, volume-preserving flow. Regularizing the equation, for example by considering a vanishing viscocity limit of Navier-Stokes solutions, might not help due to the problem of vortex reconnections. Therefore in the low regularity setting, the topological interpretation of helicity could break down, and one is left with the analytic definition~\eqref{def:hel}.  

Since there is no global existence theory for smooth solutions of~\eqref{eqn:Euler}, we must consider weak solutions, which may be square-integrable vector fields but no better.  For such weak solutions, however, it is not clear how to interpret~\eqref{def:hel}, much less check that it is conserved; indeed it would appear that~\eqref{def:hel} is only well-defined for vector fields in $\dot H^{\sfrac 12}(\mathbb{T}^3)$ via the duality pairing $H(v):=\langle v , \curl v \rangle_{\dot H^{\sfrac 12}, \dot H^{-\sfrac 12}}$.  While it is known that every weak solution $u$ of~\eqref{eqn:Euler} which belongs to $L^3_t B^{\sfrac 23}_{3, c(\mathbb{N})} \cap L^\infty_t H^{\sfrac 12}$ does indeed conserve the helicity~\cite{CCFS08} -- see also~\cite{Chae03, DeRosa, WWY23, IS25, BT25} -- the following question has remained open.
\smallskip

\noindent\textbf{Question: } Do there exist weak solutions $u$ of~\eqref{eqn:Euler} for which the helicity~\eqref{def:hel} admits a canonical definition, but for which $H(u(t))$ is non-constant in time?
\smallskip

\noindent We will answer this question positively using a notion of generalized helicity introduced below.  We remark that a stronger version of this question was posed by Buckmaster and Vicol~\cite[Conjecture~2.7]{BVShortReview}, in which the weak solution is required to belong to $C^0_t H^{\sfrac 12}$.  Our extension of the definition of helicity to rougher vector fields offers an alternative approach.

\subsection{Generalized helicity and main result}
To define the generalized helicity, we set some conventions regarding Fourier series. We identify a smooth function $f:\T^3\rightarrow \R^d$ with its smooth, periodic lift: the function which satisfies $f(x+k)=f(x)$ for all $x\in \R^3$ and $k\in \Z^3$. For $k\in \Z^3$, the Fourier coefficient $\mcf(f)(k)$ is defined by $\mcf(f)(k) = \int_{\T^3} f(x) e^{-2\pi i k\cdot x}$.  Note that for any $s \in \mathbb{R}$, this definition makes sense on $H^s(\mathbb{T}^3)$.  Now let $\varphi:\mathbb{Z}^3 \rightarrow \mathbb{R}$ be given and $f \in H^s(\mathbb{T}^3)$; we set
\begin{equation}\label{def:multiplier}
    \bp_{\varphi}( f ) = \sum_{k \in \Z^3} \varphi(k) \mcf(f)(k) e^{2\pi i k \cdot x} \, .
\end{equation}
A specific example of the above is the $m^{\rm th}$ Littlewood-Paley projection $\bp_{\varphi_{2^m}}(f)$ of $f$, where $m\in \mathbb{N}\cup \{0\}$ and $\varphi_{2^m}$ is defined as in Definition~\ref{def:hel}. Note that $\mathbb{P}_{\varphi_{2^m}} f$ is always a smooth function and depends on the choice of base kernel $\varphi$. 

Let $u_1, u_2 : \T^3 \rightarrow \R^3$ be smooth vector fields.  We define
\begin{align}
H(u_1, u_2) &:= \int_{\T^3} u_1(x) \cdot \curl u_2(x) \, dx = \sum_{\vec k \in \mathbb{Z}^3}  \mcf(u_1) (k) \cdot \left( 2\pi ik \times \mcf(u_2)(-k) \right)  \, , \label{e:fourier:series}
\end{align}
where the latter equality holds due to Parseval's identity and the assumption that $u_2$ is real-valued, so that $\overline{\mcf (u_2)}(k) = \mcf (u_2)(-k)$.  
Then for a smooth vector field $u:\T^3\rightarrow \R^3$, the helicity $H(u)$ as defined in~\eqref{def:hel} is precisely $H(u,u)$.  It is clear that $H(u_1, u_2)$ makes sense if $u_1, u_2 \in \dot H^{\sfrac 12}(\T^3)$, or if $u_1$ is smooth and $u_2$ is a distribution.   On the other hand, observe that the sum in~\eqref{e:fourier:series}, evaluated for a given non-smooth vector field $u=u_1=u_2$, may converge if $\mcf( u)(k)$ is orthogonal to $ k \times \mcf(u)(-k)$.  We therefore offer the following definition of helicity in terms of mollifications and Littlewood-Paley projections, which allows for cancellations in the sum in~\eqref{e:fourier:series}.  At a heuristic level, we require the value of the helicity to be independent of the choice of mollification kernel used in the approximation (c.f.~\eqref{e:gen.hel.2}), and the shell-by-shell helicity to be absolutely summable for every choice of kernel defining the Littlewood-Paley shells, (c.f.~\eqref{e:sum.hel}).

\begin{definition}[\textbf{Generalized helicity}]\label{def:gen.hel}
Let $u \in L^2(\T^3)$ be given. We say that the generalized helicity of $u$ is well-defined if the following two conditions hold.
\begin{enumerate}[(i)]
\item There exists a number $H(u) \in \mathbb{R}$ such that for all Schwartz kernels $\varphi$ with $\int_{\R^d} \varphi = 1$, 
\begin{align}\label{e:gen.hel.2}
    H(u) = \lim_{\epsilon \rightarrow 0} H(\varphi_\epsilon \ast u , u) \,  , \qquad \textnormal{where } \varphi_\epsilon(x) = \epsilon^{-3}\varphi(\epsilon^{-1}x) \, .
\end{align}
\item Let $\mathcal{A}$ be the class of functions $\td\varphi:\mathbb{R}\rightarrow [0,1]$ which are smooth, compactly supported, and satisfy the following: there exists $\delta(\td\varphi)>0$, depending on $\td\varphi$, such that $\td\varphi(x)= 1$ for all $|x|\leq \delta(\td\varphi)$. For a given such $\td\varphi$, define $\ph (k) = \td\ph (k) - \td\ph(2k)$ and $\ph_{2^m}(k)= \ph(2^{-m} k)$ for $m \geq 1$, and $\ph_{2^0}(k)=\td\vp(k)$.  For all such $\varphi$, 
\begin{align}\label{e:sum.hel}
 \lim_{M\rightarrow \infty} \sum_{m,m'=0}^M \left| H\left( \bp_{\varphi_{2^m}}(u) , \bp_{\varphi_{2^{m'}}} (u) \right) \right|  <\infty \, ,
\end{align}
\end{enumerate}
\end{definition}
\begin{remark}[\textbf{Computing helicity shell-by-shell}]
If $u \in L^2(\mathbb{T}^3)$ satisfies Definition~\ref{def:gen.hel}, then for all $\td\varphi \in \mathcal{A}$,
\begin{align}
    H(u) = \lim_{M\rightarrow \infty} H(\check{\tilde\varphi}_{2^{-M}}\ast u , u)= \lim_{M \rightarrow \infty} \sum_{m=0}^M H \left( \bp_{\varphi_{2^m}} (u) , u \right) =  \lim_{M\rightarrow \infty} \sum_{m,m'=0}^M H\left( \bp_{\varphi_{2^m}}(u) , \bp_{\varphi_{2^{m
    '}}}(u) \right) \, . \notag 
\end{align}
\end{remark}
\begin{remark}[\textbf{Comparison with the classical definitions of helicity}]
Note that for $\dot H^{\sfrac 12}$ vector fields,~\eqref{e:sum.hel} holds, and the limit in \eqref{e:gen.hel.2} is well-defined and agrees with~\eqref{e:fourier:series} and~\eqref{def:hel}. Therefore we are justified in using the same notation $H$ in all three cases.  However, Definition~\ref{def:gen.hel} does extend the classical definitions to a larger class of objects.  For example, shear flows with regularity worse that $\dot H^{\sfrac 12}$ will satisfy~\eqref{e:gen.hel.2}-\eqref{e:sum.hel} since $u$ and $\curl u$ are always orthogonal (as are any combination of Littlewood-Paley projections or mollifications of $u$ and $\curl u$).   Since helicity is a measure of the total amount of linking, twisting, and writhing of vortex lines, the helicity of a shear flow \emph{should} be zero; the trajectories are all straight lines. More generally, functions which have cancellations of helicity occurring within the overlap of the $k$ and $k'$ shells may fall under the purview of Definition~\ref{def:gen.hel}, such as functions which in each Littlewood-Paley shell resemble shear flows to leading order. Our construction of weak Euler solutions in this paper will be flows of this type; in each Littlewood-Paley shell, the solution is very nearly a sum of shears.
\end{remark}

Our main theorem can now be stated as follows.
\begin{theorem}\label{thm:main}
    Let $0<\beta<\sfrac 12$ and $T>0$. Suppose that smooth functions $e:[0,T]\rightarrow(0,\infty)$ and $h:[0,T]\rightarrow(-\infty,\infty)$ are given. Then there exist weak solutions $u:[0,T]\times\T^3\rightarrow\R^3$ to the 3D Euler equations, belonging to $C([0,T]; (H^\beta\cap L^{\frac{1}{1-2\beta}})(\T^3))$, such that the generalized helicity is well-defined in the sense of Definition~\ref{def:gen.hel} and the kinetic energy and generalized helicity satisfy
    \begin{equation}\label{e:profiles}
    \sfrac 12 \left\| u(t,\cdot) \right\|_{L^2(\T^3)}^2 = e(t) \, , \qquad H(u(t)) = h(t) \, .    
    \end{equation}
\end{theorem}
\noindent
The constructed solutions exhibit intermittency, a characteristic feature of turbulence as in \cite{NV22}.
By interpolation, in particular, the solutions belong to $C_tB^s_{3,\infty,x}$ with $s \to \sfrac13$ as $\beta \to \sfrac12$, and so such solutions are almost Onsager-critical in $L^3$.
Using techniques which are by now standard, we could augment the proof of this statement to instead prove non-uniqueness for the initial value problem, for example in a class of functions whose generalized helicities are constant in time.  In other words, imposing the conservation of generalized helicity has no bearing whatsoever on the uniqueness of the initial value problem, or on the behavior of the kinetic energy.

\subsection{Ideas of the proof}
Convex integration and Nash iteration techniques were first applied to the Euler equations by De Lellis and  Sz\'ekelyhidi~\cite{DLS09, DLS13}.  In the following years, Nash iterations were used to prove Onsager's conjecture~\cite{Isett2018, BDLSV17, GiriRadu} and demonstrate non-uniqueness of weak solutions to the Navier-Stokes equations~\cite{BV19}.  We also refer to the surveys~\cite{BV19, DSReview} for overviews and the papers~\cite{Luo, MS, BBV, BCV22, DK22, CheskidovLuo, Is22, FLS1, FLS2, BCK1, BCK2, MNY, M24} for recent developments.  On the other hand, our construction of solutions to the Euler equations is based on the style of intermittent Nash iterations pioneered in~\cite{BMNV21, NV22}, together with the simplifying ideas of the `continuous scheme' in~\cite{GKN23, GKN23strong}. These tools construct solutions with regularity $L^\infty_t(H^{1/2-}_x \cap L^{\infty-}_x)$, which is almost enough for the helicity to be well-defined. In order for the generalized helicity to be well-defined and prescribed, we require some new ideas and adaptations of the tools from~\cite{BMNV21, NV22, GKN23, GKN23strong}.
\subsubsection{Mikado flows have zero helicity} 
The high-frequency building blocks that we use in our Nash iteration are intermittent pipe bundles, or products of two sets of intermittent Mikado flows with different thicknesses and spacings. These are non-interacting, smooth shear flows and so have zero (generalized) helicity. Pretending for a moment that the bundles need not be transported by the low-frequency background flow, the leading order term in the helicity of the perturbation vanishes.  By using this cancellation, we are able to control the leading order terms required to satisfy Definition~\ref{def:gen.hel}.  However, a number of further ideas are required to control the helicity from the remaining terms; it is at this point that the methods from~\cite{BMNV21, NV22, GKN23, GKN23strong} cease to be sufficient.

\subsubsection{Shortening of the time-scale} In order for our perturbation to have interact well with the background flow $\Phi$, it is necessary for the high-frequency building blocks to be transported by $\Phi$. If $\Phi$ deforms our shear building blocks to order one by an arbitrary diffeomorphism, i.e. if $\|D\Phi - \Id \| = O(1)$, then it is no longer possible to say that the deformed shears have small helicity. 

To remedy this, we shorten the time-scale on which we run the background flow from the Lipschitz time-scale $\tau_q$ to a new time-scale $\mu_q \ll \tau_q$.  Crucially, since the Lipschitz norm of the background vector field is unchanged, we now have $\|D\Phi - \Id \| < \tau_q^{-1}\mu_q$; see Corollary~\ref{cor:deformation}. The helicity contribution from the deformed shears, computed in~\eqref{514:1}, is of order $\tau_q^{-1}\mu_q \la_q \de_\qbn$ 
which must be $\ll 1$; see subsection~\ref{sec:not.general} for the definitions of these parameters.  But we cannot choose $\mu_q$ arbitrarily small due to the transport error (see Section~\ref{ss:ER:TN} and~\eqref{par:trans}), which requires that
$$ \frac{\de_\qbn^{\sfrac12}\mu_q^{-1}}{\la_\qbn}\frac{\la_{q+\half}}{\la_\qbn} < \de_{q+2\bn}\,. $$
Putting the above two constraints together, we see that $\mu_q$ must satisfy
$$ \de_{q+2\bn}^{-1} \de_\qbn^{\sfrac12} \la_{q+\half} \la_{\qbn}^{-2} < \mu_q < \tau_q \la_q^{-1}\de_\qbn^{-1} \underset{\eqref{eq:defn:tau}}\approx {\de_\qbn^{-1}\de_q^{-\sfrac12}\la_q^{-2}}\,. $$
Using the definitions of $\de_q$ and $\la_q$, such a choice of $\mu_q$ only exists if
$$ \de_{q+2\bn}^{-1} \de_\qbn^{\sfrac32} \de_q^{\sfrac12} \la_{q+\half} \la_q^{2} \la_{\qbn}^{-2} < 1 \iff 2\beta \bar b^4 - 3\beta \bar b^2 - \beta + 2 + \bar b - 2 \bar b^2 < 0\,. $$
In other words, we must have $\beta > \frac{2 + \bar b - 2 \bar b^2}{1+3\bar b^2 - 2\bar b^4}$, which is essentially one of the constraints in~\eqref{con:ovb}. 

\subsubsection{Using regularity to deal with lower order terms} 
Certain lower order terms, such as the divergence correctors of the principal Nash perturbation, have no apparent cancellations that would imply that their helicity is small. For such terms, we simply estimate their helicity by their amplitude. For example, we estimate
\begin{align*}
    \int_{\T^3} w^{(c)} \cdot \curl w^{(c)}\, dx \leq \|w^{(c)}\|_{L^2}\|D w^{(c)}\|_{L^2}
\end{align*}
Such terms can only be small if the solutions are of a certain regularity $\beta > \frac{2-\bar b}{2\bar b}$; see again~\eqref{con:ovb}. 

The oscillation error in the Nash iteration puts an upper bound $\beta < (2\bar b^2)^{-1}$, implying that $(\la_q\de_{\qbn})_{q=1}^\infty$ is an increasing sequence. However, the shortening of the time-scale and the lower order terms force lower bounds on the regularity $\beta$ in terms of $\bar b$, a consequence of which we have that $(\la_q\de_{q+2\bn})_{q=1}^\infty$ is a decreasing sequence. For a choice of $\beta$, such bounds are satisfied for a range of $\bar b$, but we are unable to send $\bar b$ to $1$ for an arbitrary choice of $\beta < 1/2$. We deal with this by first choosing $\bar b$ and then choosing $\beta$ in our Nash iteration. The precise choices are given in Section~\ref{sec:not.general}.

\subsubsection{New `helical' building blocks in order to inject a desired helicity}
In order to prescribe the helicity profile in Theorem~\ref{thm:main}, at each step of our iteration we add in a new perturbation $w_h$ that makes sure that the new total helicity differs from our desired profile by an amount that goes to zero along the iteration; see Hypothesis~\ref{hyp:helicity.prescription} for the exact estimate. Our $w_h$ is small enough in $L^2$ so that the oscillation errors it produces in the Nash iteration can be directly absorbed into the new Reynolds error without needing to invert the divergence. However, $w_h$ is large enough in $\dot H^{\sfrac 12}$ in order to correct the helicity produced by the previous steps of the iteration as well as that produced by the Nash perturbation. 

This perturbation $w_h$ will have the overall structure of the usual Nash perturbation with the intermittent pipe bundles replaced by helical building blocks. These are divergence-free flows that are supported on a small neighborhood of a helix $t \mapsto (r\cos (2\pi\la_*t), r\sin (2\pi\la_*t), t)$ that itself is (up to fixed geometric constants) contained in the support of a single Mikado. The latter property allows for the required dodging Hypotheses~\ref{hyp:dodging1}--\ref{hyp:dodging2} to hold using the same proof that these hypotheses hold for the usual Nash perturbation. For this, one must have $r \lec \la_\qbn^{-1}$, the radius of a single Mikado; in fact we will choose $r\approx \la_\qbn$. The `wavelength' $\la_*^{-1}$ of the helix is chosen such that $\la_* \ll \la_\qbn$, so that the helical perturbation lives at frequency $\la_\qbn$. However, in order to inject enough helicity and close various estimates, we also require that $\la_{q+\bn-1}\ll \la_*$.

\subsection*{Acknowledgements}
 V.G. was supported by the Simons Collaboration on the Localization of Waves and by an SNF Ambizione grant. H.K. gratefully acknowledged the financial support of ETH Z\"{u}rich. M.N. was supported by NSF grant no. DMS-2307357.

\section{Inductive assumptions}

In this section, we present the inductive assumptions required for the iterative proposition, which is stated in subsection~\ref{ss:iterative}. The main theorems are then proved in the same subsection.

\subsection{General notations and parameters}\label{sec:not.general}

To formulate the inductive assumptions, we need a number of carefully chosen parameters.  We specify the primary parameters now while ignoring certain secondary parameters. A full delineation of the parameters is contained in Section~\ref{sec:pam}.  

\newcommand{\ovb}{\ov b}

\begin{enumerate}[(i)]
    \item\label{ip1} Choose a parameter $\ov b$ such that $0<\ov b-1 \ll 1$,  $\bar b^2 <2$, and
\begin{align}\label{con:ovb}
      \max\left( \frac{2+\ovb- 2\ovb^2}{1+3\ovb^2 -2\ovb^4}, \frac{2-\ovb}{2\ovb}\right) < \frac 1{2\ovb^2}. 
    \end{align}
    \item\label{ip2} We define the regularity index $0<\beta<\sfrac12$, so that $\beta \nearrow \sfrac 12$ as $\ovb \searrow 1$, by
    \begin{equation}\label{def:beta}
        \beta = \frac 12 \left( \max\left( \frac{2+\ovb- 2\ovb^2}{1+3\ovb^2 -2\ovb^4}, \frac{2-\ovb}{2\ovb}\right) +    \frac{1}{2\ovb^2}\right) 
    \end{equation}
\item\label{ip3} Choose $\bn \gg 1$ an even integer such that 
\begin{equation}\label{par:osc:pre}
    \beta < \frac 1{2b^{\bn}} - \frac8{\bn}\,,
\end{equation}
and then define
\begin{equation}\label{def:b}
b = \ovb^{\sfrac 2 \bn } \qquad \iff \qquad  \ovb = b^{\sfrac \bn 2} \, .
\end{equation}
Since $\ovb$ is fixed, $b\searrow 1$ as $\bn\nearrow \infty$. 
\end{enumerate}

The frequency parameter $\la_q$, the amplitude parameter $\de_q$, the intermittency parameter $r_q$, and the multi-purpose parameter $\Gamma_q$ are defined by 
\begin{subequations}
\begin{align}
    &\la_q = 2^{\left \lceil (b^q) \log_2 a \right \rceil}\approx a^{(b^q)} \, , \qquad  \de_q =\la_0^{2\be} \la_q^{-2\beta} \, ,
    \label{eq:def:la:de}\\
    r_q = \frac{\la_{q+\half}\Ga_q}{\la_\qbn}&\, , \qquad
    \Ga_q = 2^{\left \lceil \varepsilon_\Gamma \log_2 \left( \frac{\la_{q+1}}{\la_q} \right) \right \rceil}  
    \approx \left( \frac{\la_{q+1}}{\la_q} \right)^{\varepsilon_\Gamma} \approx \la_{q}^{(b-1)\varepsilon_\Gamma} \, .   \label{eq:deffy:of:gamma}
\end{align}
\end{subequations}
We shall require the additional parameters $\tau_q$, $\mu_q$,  $\la_\qho$, $\Lambda_q$, $\Tau_q$, $\la_\qho$, and $\badshaq$, defined by
\begin{subequations}\label{def:m:t}
\begin{align}
\tau_q = \left(\delta_q^{\sfrac 12} \lambda_q \Gamma_q^{35}\right)^{-1}, \qquad
\mu_q =  {\tau_q} \la_q^{-(1-2\be b^{\bn})}{\Ga_{q+1}^{-30}} \approx \tau_q (\lambda_q\delta_\qbn)^{-1}  \Gamma_{q+1}^{{-30}} \, , \label{eq:defn:tau}  \\
\Lambda_q = \lambda_q \Gamma_q^{10} \, , \qquad \Tau_q = {\Lambda_q^{10}} \, , \qquad 
\la_\qho = \lambda_\qbn \Gamma_{q-100}^{-\frac{1}{100}} \, ,
\qquad
\badshaq {\approx \log_{\Gamma_q} \lambda_q^{\frac{12}{\bn}}} \, ,  \label{def:thursday}
\end{align}
\end{subequations}
so that $\Gamma_q^{\badshaq}\leq \lambda_q^{\frac{12}{\bn}}$. Note that from~\eqref{eq:eg:timescale}, we have that $\mu_q \ll \tau_q$; the amount by which we shorten $\mu_q$ is essentially the ``amount of extra room in the transport error estimate.'' Finally, we set
\begin{align*}
    D_{t,q} = \pa_t  + \hat u_q\cdot \nabla, \quad  1\ll \NcutSmall\ll \Nindt \ll  \Nind \ll \Nfin
\end{align*}
and
\[ \MM{n,N_*,\tau^{-1},\Tau^{-1}} := \tau^{-\min\{n,N_*\}} \Tau^{-\max\{n-N_*,0\}}\,.\]
\begin{remark}[\bf Space-time norms]\label{rem:notation:space:time}
We shall always measure objects using uniform-in-time norms $\sup_{t\in[T_1,T_2]}\|\cdot (t)\|$, where $\| \cdot (t)\|$ is any of a variety of norms used to measure functions defined on $\T^3\times[T_1,T_2]$ but restricted to time $t$. We shall always abbreviate these norms by $\| \cdot \|$.
\end{remark}

\begin{remark}[\bf Space-time balls]\label{rem:notation:space:time:balls}
For any set $\Omega\subseteq \mathbb{T}^3\times \mathbb{R}$, we shall use the notations
\begin{subequations}\label{eq:space:time:balls}
\begin{align}
B(\Omega,\lambda^{-1}) &:= \left\{ (x,t) \, : \, \exists \, (x_0,t) \in \Omega \textnormal{ with } |x-x_0| \leq \lambda^{-1} \right\}\\
B(\Omega,\lambda^{-1},\tau) &:= \left\{ (x,t) \, : \, \exists \, (x_0,t_0) \in \Omega \textnormal{ with } |x-x_0| \leq \lambda^{-1} \, , |t-t_0| \leq \tau \right\}
\end{align}
\end{subequations}
for space and space-time neighborhoods of $\Omega$ of radius $\lambda^{-1}$ in space and $\tau$ in time, respectively. 
\end{remark}

\subsection{Relaxed equations}\label{ss:relaxed}

We assume that $(u_q, p_q, R_q)$ satisfy
\begin{align}\label{eqn:ER}
\begin{cases}
\partial_t u_q + \div (u_q \otimes u_q) + {\nabla p_q} = \div R_q \, , \qquad \qquad 
\div\, u_q = 0 \, .
\end{cases}
\end{align}
We will use the notations
\begin{equation}\label{eq:hat:no:hat}
u_q = \underbrace{\hat u_{q-1} + \hat w_q}_{=: \hat u_q} + \hat w_{q+1} + \dots + \hat w_{q+\bn-1} =: \hat u_{q+\bn-1} \, 
\end{equation}
for velocity fields.  We assume that $R_q$ has a decomposition into symmetric matrices
\begin{align}
R_q &= \sum_{k=q}^{q+\bn-1} R_q^k \, . \label{eq:ER:decomp:basic} 
\end{align}

\subsection{Inductive assumptions for velocity cutoff functions}
\label{sec:cutoff:inductive}
We note that all assumptions in subsection~\ref{sec:cutoff:inductive} are assumed to hold for all $q-1\leq q'\leq q+\bn-1$. First, we assume that the velocity cutoff functions\index{velocity cutoffs} form a partition of unity:\index{$i$ and $\imax$}
\begin{align}\label{eq:inductive:partition}
    \sum_{i\geq 0} \psi_{i,q'}^2 \equiv 1, \qquad \mbox{and} \qquad \psi_{i,q'}\psi_{i',q'}=0 \quad \textnormal{for}\quad|i-i'| \geq 2 \, .
\end{align}
Second, we assume that there exists an $\imax = \imax(q') \geq 0$, bounded uniformly in $q'$, such that
\begin{subequations}\label{ind:imax}
\begin{align}
\imax(q') &\leq \frac{\badshaq+12}{(b-1)\varepsilon_\Gamma} \, ,
\label{eq:imax:upper:lower} \\
\psi_{i,q'} \equiv 0 \quad \mbox{for all} \quad i > \imax(q')\,,
\qquad &\mbox{and} \qquad
\Gamma_{q'}^{\imax(q')} \leq 
\Ga_{q'-\bn}^{\sfrac{\badshaq}{2}+18}
 \delta_{q'}^{-\sfrac 12}r_{q'-\bn}^{-1}   \, .
\label{eq:imax:old}
\end{align}
\end{subequations}
For all $0 \leq i \leq \imax$, we assume the following pointwise derivative bounds for the cutoff functions $\psi_{i,q'}$. First, for multi-indices $\alpha,\beta \in {\mathbb N}^k$, $k \geq 0$, $0 \leq |\alpha| + |\beta| \leq \Nfin$, we assume that
\begin{align}
&\frac{{\bf 1}_{\supp \psi_{i,q'}}}{\psi_{i,q'}^{1- (K+M)/\Nfin}} \left|\left(\prod_{l=1}^k D^{\alpha_l} D_{t,q'-1}^{\beta_l}\right) \psi_{i,q'}\right| \leq \Gamma_{q'} (\Gamma_{q'}  \lambda_{q'})^{|\alpha|} 
\MM{|\beta|,\NindSmall - \NcutSmall,  \Gamma_{q'-1}^{i+3}   {\mu_{q'-1}^{-1}}, \Gamma_{q'-1} \Tau_{q'-1}^{-1}} \, .
\label{eq:sharp:Dt:psi:i:q:old}
\end{align}
Next, with $\alpha, \beta,k$ as above, $N\geq 0$ and $D_{q'}:=\hat w_{q'}\cdot\nabla$, we assume that
\begin{align}
&\frac{{\bf 1}_{\supp \psi_{i,q'}}}{\psi_{i,q'}^{1- (N+K+M)/\Nfin}} \left| D^N \left( \prod_{l=1}^k D_{q'}^{\alpha_l} D_{t,q'-1}^{\beta_l}\right)  \psi_{i,q'} \right| \notag\\
&\qquad \leq \Gamma_{q'} ( \Gamma_{q'}  \lambda_{q'})^N
(\Gamma_{q'}^{i-5}  {\mu_{q'-1}^{-1}})^{|\alpha|}
\MM{|\beta|,\Nindt-\NcutSmall,  \Gamma_{q'-1}^{i+3}   {\mu_{q'-1}^{-1}}, \Gamma_{q'-1}  \Tau_{q'-1}^{-1}}
\label{eq:sharp:Dt:psi:i:q:mixed:old}
\end{align}
for $0 \leq N+ |\alpha| + |\beta| \leq \Nfin$. Moreover, for $0\leq i \leq \imax(q')$, we assume that
\begin{align}
\norm{\psi_{i,q'}}_{1}  \leq \Gamma_{q'}^{-2i+\CLebesgue} \qquad \mbox{where} \qquad \CLebesgue = \frac{4+b}{b-1} \, .
\label{eq:psi:i:q:support:old}
\end{align}
Our last assumption is that for all $q' \leq q+\bn-1$ and all $q''\leq q'-1$, we assume 
\begin{equation}
    \psi_{i',q'}  \psi_{i'',q''} \not \equiv 0 \quad \implies \quad 
    \mu_{q'} \Gamma_{q'}^{-i'} \leq \mu_{q''} \Gamma_{q''}^{-i'' -25}   \, . \label{eq:inductive:timescales}
\end{equation}

\subsection{Inductive bounds on the Reynolds stress}\label{sec:R:inductive}

We assume that for $q\leq k \leq q+\bn-1$ and $N+M\leq 2\Nind$, $R_q^k$ satisfies 
\begin{subequations}\label{eq:pressure:inductive}
\begin{align}
\norm{ \psi_{i,k-1} D^N D_{t,k-1}^M  R_q^k }_{1}  
 &\leq \Ga_q \delta_{k+\bn} \Lambda_k^N \MM{M, \NindRt, \Gamma_{k-1}^{i {+20}}  {\mu_{k-1}^{-1}},   \Tau_{k-1}^{-1}  {\Ga_{k-1}^{10}}} \, .
\label{eq:R:inductive:dtq} \\
\norm{ \psi_{i,k-1} D^N D_{t,k-1}^M R_q^k }_{\infty}  
 &\leq \Ga_q {\Gamma_k^{\badshaq}} \Lambda_k^N \MM{M, \NindRt, \Gamma_{k-1}^{i {+20}}  {\mu_{k-1}^{-1}} ,  \Tau_{k-1}^{-1} {\Ga_{k-1}^{10}} } \, , \label{eq:R:dtq:uniform} 
\end{align}
\end{subequations}

\subsection{Dodging inductive hypotheses}\label{ss:dodging}

\begin{hypothesis}[\bf Effective dodging]\label{hyp:dodging1}
For $q',q''\leq q+\bn-1$ that satisfy $0<|q''-q'|\leq \bn-1$,
\begin{equation}\label{eq:ind:dodging}
    B\left(\supp \hat w_{q'} , \lambda_{q'}^{-1}\Gamma_{q'+1} \right) \cap  B\left( \supp \hat w_{q''} , \lambda_{q''}^{-1}\Gamma_{q''+1} \right)= \emptyset \, .
\end{equation}
\end{hypothesis}

\begin{hypothesis}[\textbf{Density and direction of old pipe bundles}]\label{hyp:dodging2}
There exists a $q$-independent constant $\const_D$ such that the following holds.  Let $\bar q', \bar q''$ satisfy $q \leq \bar q'' < \bar q'\leq q+\bn-1$, and set
\begin{equation} \label{eq:diam:def}
d(\bar q', \bar q'') :=  \min\left[ (\lambda_{\bar q''}\Gamma_{\bar q''}^7)^{-1} , (\lambda_{\bar q' - \half} \Gamma_{\bar q' - \bn})^{-1} \right] \, .
\end{equation}
Let $t_0\in\R$ be any time and $\Omega\subset\T^3$ be a convex set of diameter at most $d(\bar q', \bar q'')$.  Let $i$ be such that $\Omega \times \{t_0\}\cap \supp \psi_{i,\bar q''} \neq \emptyset$. 
Let $\Phi_{\bar q''}$ be the flow map such that
\begin{align*}
\begin{cases}
\pa_t \Phi_{\bar q''} + \left(\hat u_{\bar q''} \cdot \na \right) \Phi_{\bar q''} = 0\\
\Phi_{\bar q''}(t_0,x) = x \, .
\end{cases}
\end{align*}
We define $\Omega(t)=\Phi_{\bar q''}(t)^{-1}(\Omega)$.\footnote{For any set $\Omega'\subset \T^3$, $\Phi_{\bar q''}(t)^{-1}(\Omega')=\{x\in \T^3: \Phi_{\bar q''}(t,x) \in \Omega'\}$. We shall also sometimes write $\Omega\circ \Phi_{\bar q''}(t)$.} 
Then there exists a set $L=L(\bar q',\bar q'', \Omega, {t_0})\subseteq \T^3\times \R$ such that for all $t\in(t_0- {\mu_{\bar q''}}\Gamma_{\bar q''}^{-i+2},t_0+ {\mu_{\bar q''}}\Gamma_{\bar q''}^{-i+2})$,
\begin{align}\label{eq:ind:dodging2}
    (\partial_t + \hat u_{\bar q''}\cdot\nabla) \mathbf{1}_{L}(t,\cdot) \equiv 0 \qquad \textnormal{and} \qquad \textnormal{supp}_x \hat w_{\bar q'}(x,t) \cap \Omega(t) \subseteq L \cap \{t\} \, .
\end{align}
Furthermore, there exists a finite family of continuously differentiable curves $\{\ell_{j,L}\}_{j=1}^{\const_D}$ of length at most $2d(\bar q', \bar q'')$ which satisfy 
\begin{equation}\label{eq:concentrazion}
    L \cap \{t=t_0\} \subseteq \bigcup_{j=1}^{\const_D} B\left( \ell_{j,L} , 3\la_{\bar q'}^{-1} \right) \, .
\end{equation}
Finally, there exist $q$-independent sets  {$\Xi_0,\Xi_0', \dots, \Xi_{\bn-1}, \Xi_{\bn-1}' \subset \mathbb{Q}^3\cap \mathbb{S}^2$} such that for all curves $\ell_{j,L}$ used to control the support of $\hat w_{\overline q'}$, the tangent vector to the curve $\ell_{j,L}$ belongs to a $\Gamma_{ {0}}^{-1}$-neighborhood of a vector $\xi\in \Xi_{\overline{q}' \, \textnormal{mod} \, \bn} \cup \Xi'_{\overline{q}' \, \textnormal{mod} \, \bn}$.
\end{hypothesis}

\subsection{Inductive velocity bounds}
\label{sec:inductive:secondary:velocity}

In this subsection, we assume that $ {1}\leq q'\leq q+\bn-1$.  First, for $0 \leq i \leq \imax$, $k\geq 1$, $\alpha, \beta \in \N^k$, we assume that
\begin{align}
&\norm{\left( \prod_{l=1}^k D^{\alpha_l} D_{t,q'-1}^{\beta_l} \right) \hat w_{q'} }_{L^\infty(\supp \psi_{i,q'})} \leq \Gamma_{q'}^{i+2}\de_{q'}^{\sfrac12}(\la_{q'}\Ga_{q'})^{|\alpha|} \MM{|\beta|,\Nindt, \Gamma_{{q'}}^{i+3}   {\mu_{q'-1}^{-1}},  \Gamma_{{q'-1}} \Tau_{q'-1}^{-1}}
\label{eq:nasty:D:wq:old}
\end{align}
for $|\alpha|+|\beta|  \leq {\sfrac{3\Nfin}{2}+1}$. We also assume that for $N\geq 0$,
\begin{subequations}\label{eq:nasty}
\begin{align}
&\norm{ D^N \Big( \prod_{l=1}^k D_{q'}^{\alpha_l} D_{t,q'-1}^{\beta_l} \Big) \hat w_{q'}}_{L^\infty(\supp \psi_{i,q'})} \notag\\
&\qquad \leq
(\Gamma_{{q'}}^{i+2}\delta_{q'
}^{\sfrac 12})^{|\alpha|+1} (\la_{q'}\Ga_{q'})^{N+|\alpha|} \MM{|\beta|,\Nindt,  \Gamma_{q'}^{i+3}   {\mu_{q'-1}^{-1}},  \Gamma_{q'-1} \Tau_{q'-1}^{-1}}   \label{eq:nasty:Dt:wq:old} \\
&\qquad \leq
\Gamma_{q'}^{i+2} \delta_{q'}^{\sfrac 12}  (\lambda_{q'}\Ga_{q'})^N (\Gamma_{q'}^{i-5}   {\mu_{q'}^{-1}})^{|\alpha|}  \MM{|\beta|,\Nindt, \Gamma_{q'}^{i+3}   {\mu_{q'-1}^{-1}},  \Gamma_{q'-1} \Tau_{q'-1}^{-1}}
\label{eq:nasty:Dt:wq:WEAK:old}
\end{align}
\end{subequations}
whenever $N+|\alpha|+|\beta|\leq  {\sfrac{3\Nfin}{2}+1}$. Next, we assume
\begin{align}
&\norm{\left( \prod_{l=1}^k D^{\alpha_l} D_{t,q'}^{\beta_l} \right) D \hat u_{q'} }_{L^\infty(\supp \psi_{i,q'})}  \leq  {\delta_{q'}^{\sfrac 12} \lambda_{q'}\Ga_{q'}^{i+3}} (\lambda_{q'}\Ga_{q'})^{|\alpha|} \MM{|\beta|,\Nindt,\Gamma_{q'}^{i-5}  {\mu_{q'}^{-1}},   \Gamma_{q'-1} \Tau_{q'-1}^{-1}}
\label{eq:nasty:D:vq:old}
\end{align}
for $|\alpha|+|\beta| \leq {\sfrac{3\Nfin}{2}}$.  In addition, we assume the lossy bounds
\begin{subequations}\label{eq:bob:old}
\begin{align}
\norm{\left( \prod_{l=1}^k D^{\alpha_l} D_{t,q'}^{\beta_l} \right)  \hat u_{q'}}_{L^\infty(\supp \psi_{i,q'})} &\leq \Gamma_{q'}^{i+2} \delta_{q'}^{\sfrac 12} \lambda_{q'}^2 (\lambda_{q'}\Ga_{q'})^{|\alpha|} \MM{|\beta|,\Nindt,\Gamma_{q'}^{i-5}  {\mu_{q'}^{-1}},   \Gamma_{q'-1} \Tau_{q'-1}^{-1}}
\label{eq:bob:Dq':old} \\
\left\| D^{|\alpha|} \partial_t^{|\beta|} \hat{u}_{q'} \right\|_{L^\infty} & \leq \lambda_{q'}^{\sfrac 12} (\lambda_{q'}\Gamma_{q'})^{|\alpha|} \Tau_{q'}^{-|\beta|} \, , \label{eq:bobby:old}
\end{align}
\end{subequations}
hold, where the first bound holds for $|\alpha|+|\beta| \leq \sfrac{3\Nfin}{2}+1$, and the second bound for $|\alpha|+|\beta|\leq 2\Nfin$.

\begin{remark}[\bf Upgrading material derivatives for velocity and velocity cutoffs]
\label{rem:D:t:q':orangutan}
Similar to~\cite[Remark~2.9]{GKN23}, we have the bound
\begin{align}
&\norm{ D^N  D_{t,q'}^{M}  \hat w_{q'} }_{L^\infty(\supp \psi_{i,q'})} \les \Gamma_{q'}^{i+2} \delta_{q'}^{\sfrac 12}  (\la_{q'}\Ga_{q'})^N 
\MM{M,\Nindt, \Gamma_{q'}^{i-5}   {\mu_{q'}^{-1}},  \Gamma_{q'-1} \Tau_{q'-1}^{-1}}
\label{eq:nasty:Dt:uq:orangutan}
\end{align}
for all $N+M \leq {\sfrac{3\Nfin}{2}+1}$. We also have that for all $N+M \leq \Nfin$,
\begin{align}
\frac{{\bf 1}_{\supp \psi_{i,q'}}}{\psi_{i,q'}^{1- (N+M)/\Nfin}} \left| D^N  D_{t,q'}^{M}  \psi_{i,q'} \right| &\leq \Gamma_{q'} (\lambda_{q'}\Ga_{q'})^N
\MM{M,\Nindt-\NcutSmall, \Gamma_{q'}^{i-5}  {\mu_{q'}^{-1}}, \Gamma_{q'-1}  \Tau_{q'-1}^{-1}} \notag\\
&< \Gamma_{q'} (\lambda_{q'}\Ga_{q'})^N \MM{M,\Nindt,\Gamma_{q'}^{i-4} {\mu_{q'}^{-1}},\Gamma_{q'-1}^2\Tau_{q'-1}^{-1}} \, .
\label{eq:nasty:Dt:psi:i:q:orangutan}
\end{align}
\end{remark}

\subsection{Inductive energy and helicity bounds}\label{sec:he:bounds}
We assume that
\begin{align}\label{eq:resolved:energy}
    \frac{\delta_\qbn \Ga_q^{10}}{4} \leq e(t) - \frac 12 \| u_q(t,\cdot) \|_{L^2(\T^3)}^2 &\leq \delta_\qbn \Ga_q^{10}\,. 
\end{align}
Let $B=\overline{B_1(0)} \subset \mathbb{C}$ be the closed unit ball in the complex plane. For any $q'\leq q$, the helicity of $u_{q'}$ satisfies the following hypothesis. 

\begin{hypothesis}[\bf Helicity prescription]\label{hyp:helicity.prescription}
    For any $q'\leq q$ and any multiplier $\varphi :\mathbb{Z}^3 \rightarrow B$  satisfying $\varphi(k) = 1$ for $|k|\leq \lambda_{q'+\bn-1} \Gamma_{q'+\bn-1}$, we have that
\begin{align}\label{eq:resolved:helicity1} 
\Ga_{q'}^{-1} \leq  h(t) - \int_{\T^3}  \bp_{\varphi} (u_{q'}) \cdot \curl u_{q'}  \leq \Ga_{q'-1}^{-1}  \, .   
\end{align}
\end{hypothesis}
\noindent Note we may take $\varphi \equiv 1$ above, so that~\eqref{eq:resolved:helicity1} implies that for all $q' \leq q$,
\begin{align}\label{eq:resolved:helicity1a} 
   \Ga_{q'}^{-1} \leq  h(t) - H(u_{q'}) \leq \Ga_{q'-1}^{-1} \, . 
\end{align}
\begin{hypothesis}[\bf Helicity-in-a-shell]\label{hyp:helinashell}
    For $q' \leq q$ and any Schwartz kernel $\varphi:\R^3 \rightarrow B$  satisfying
\begin{align}\label{schwartz:decay}
\| \check \varphi \|_{L^1(\R^3)}
    \leq 1\,,\quad
\left\|  \mathbf{1}_{F_q} \Check{\varphi} \right\|_{L^1(\R^3)} \leq \lambda_{q'+\bn}^{-100} \, , \quad \textnormal{where } F_q = \left( B_{(\lambda_{q'+\bn-1}\Gamma^{-2}_{q'+\bn-1})^{-1}}\right)^c \, ,
\end{align}
we have that\footnote{We have in mind $\bp_\varphi=\bp_{\varphi_{2^m}} \bp_{\varphi_{2^{m'}}}
=\bp_{\varphi_{2^m}*\varphi_{2^{m'}} } $
; this condition will allow us to prove~\eqref{e:sum.hel}.}
\begin{align}
\left| \int_{\T^3} \bp_\varphi (\hat w_{q'+\bn-1}) \cdot \curl  \hat w_{q'+\bn-1} \right| < \Gamma_{q'-5}^{-1/2} \, . \label{summy}
\end{align}
\end{hypothesis}
\begin{hypothesis}[\bf Frequency localization]\label{hyp:freq:loc}
    For any $q' \leq q$, any natural numbers $k_1 \leq \lambda_{q'+\bn-1} \Gamma_{q'+\bn-1}^{-1}$ and $k_2 \geq \lambda_{q'+\bn-1}\Gamma_{q'+\bn-1}$, and any multipliers $\varphi_1, \varphi_2:\mathbb{Z}^3\rightarrow B$ such that $\varphi_1(k)=0$ for all $|k|>k_1$ and  $\varphi_2(j)=0$ for all $ |j|<k_2$, 
    we have 
\begin{align}\label{eq:resolved:helicity3}
 \left\| \bp_{\varphi_1}\left( \hat w_{q'+\bn-1} \right) \right\|_\infty \leq {k_1}{\lambda_{q'+\bn-1}^{-20}} \, , \qquad  \left\| \bp_{\varphi_2}\left( \hat w_{q'+\bn-1} \right) \right\|_\infty \leq {\lambda_{q'+\bn-1}}{k_2^{-20}}  \, .
\end{align}
\end{hypothesis}
\begin{hypothesis}[\bf Temporal regularity]\label{hyp:temp}
For $0 \leq M \leq \sfrac{\Nfin}{3}$, $H(u_q):  {[0,T]} \rightarrow \R$ satisfies
\begin{equation}\label{hel:lip}
    \left\| \left[ \sfrac{d}{dt} \right]^M H(u_q) \right\|_{L^{\infty}([0,T])} \les \MM{M, \Nindt, \mu_{q-1}^{-1} \Gamma_{q-1}^{20}, \Tau_{q-1}^{-1}\Ga_{q-1}^8} \, .
\end{equation}
\end{hypothesis}

\subsection{Iterative proposition and proof of main theorem}\label{ss:iterative}

 {Before stating the iterative proposition, we smoothly extend $e$ and $h$ to the time interval $[-1,T+1]$ and denote (in a slight abuse of notation) the extensions by $e$ and $h$.}

\begin{proposition}[\bf Iterative proposition]\label{prop:main} 
Let $e,h,T$ be given as in Theorem~\ref{thm:main}, $\ovb$ chosen as in~\eqref{con:ovb}, $\beta$ as in~\eqref{def:beta} and Theorem~\ref{thm:main}, $\bn$ as in~\eqref{par:osc:pre}, and $b$ as in~\eqref{def:b}. Then there exist parameters $\varepsilon_\Gamma$, $\badshaq$, $\NcutSmall$, $\Nindt$, $\Nind$, $\Nfin$ such that we can find sufficiently large $a_*$ such that for $a\geq a_*$, the following statements hold for any $q\geq 0$. Suppose that $(u_q, p_q, R_q)$ solves~\eqref{eqn:ER} on the time interval $[-\tau_{q-1},T+\tau_{q-1}]$ is given, and suppose that there exists a partition of unity $\{\psi_{i,q'}^6\}_{i\geq 0}$ of $[-\tau_{q-1},T+\tau_{q-1}] \times \T^3$ for $q-1\leq q'\leq q+\bn-1$ such that 
    \begin{itemize}
        \item $\psi_{i,q'}$ satisfies \eqref{eq:inductive:partition}--\eqref{eq:inductive:timescales}, and
        \item $u_q$ and $R_q$ satisfy \eqref{eq:hat:no:hat}, \eqref{eq:ER:decomp:basic}, \eqref{eq:pressure:inductive}, \eqref{eq:resolved:energy}, and Hypotheses~\ref{hyp:dodging1}, \ref{hyp:dodging2}, \ref{hyp:helicity.prescription}--\ref{hyp:temp},
         and
\eqref{eq:nasty:D:wq:old}--\eqref{eq:bob:old}
         hold. 
    \end{itemize}
     Then there exist  $\{\psi_{i, q+\bn}^6\}_{i\geq 0}$ of $[-\tau_q,T+\tau_q] \times \T^3$ satisfying \eqref{eq:inductive:partition}--\eqref{eq:inductive:timescales} for $q'=q+\bn$, and $(u_{q+1}, p_{q+1}, R_{q+1})$ solving \eqref{eqn:ER} on $[-\tau_q,T+\tau_q]$ satisfying the following conditions: \eqref{eq:hat:no:hat}, \eqref{eq:ER:decomp:basic}, \eqref{eq:pressure:inductive}, \eqref{eq:resolved:energy}, and Hypotheses~\ref{hyp:dodging1}, \ref{hyp:dodging2}, \ref{hyp:helicity.prescription}--\ref{hyp:temp}, and \eqref{eq:nasty:D:wq:old}--\eqref{eq:bob:old} hold for $q\mapsto q+1$.    
\end{proposition}
\begin{proof}[Proof of Proposition~\ref{prop:main}]
\eqref{eq:inductive:partition} is verified in Lemma \ref{lem:partition:of:unity:psi},~\eqref{ind:imax} in~\eqref{eq:imax:upper:bound:uniform},~\eqref{eq:sharp:Dt:psi:i:q:old}--\eqref{eq:sharp:Dt:psi:i:q:mixed:old} in Lemma \ref{lem:sharp:D:psi:i:q}, \eqref{eq:psi:i:q:support:old} in Lemma~\ref{lem:psi:i:q:support}, and~\eqref{eq:inductive:timescales} in Lemma~\ref{lem:overlap:timescales}.  Next,~\eqref{eqn:ER} is verified in~\eqref{ER:new:equation},~\eqref{eq:hat:no:hat} in~\eqref{def.w.mollified}, and \eqref{eq:ER:decomp:basic} holds by induction and~\eqref{ho},~\eqref{eq:Onpnp:est},~\eqref{eq:trans:est},~\eqref{eq:div:corrector}, and~\eqref{est.stress.moll}. Hypotheses~\ref{hyp:dodging1}--\ref{hyp:dodging2} are verified in Lemma~\ref{lem:dodging} and Remark~\ref{rem:checking:hyp:dodging:1}, 
 {\eqref{eq:pressure:inductive} is proved in Section~\ref{sec:err},}
and~\eqref{eq:nasty:D:wq:old}--\eqref{eq:bob:old} holds by induction and Lemma~\ref{lem:Dt:Dt:wq:psi:i:q:multi}. 
 {Lastly, \eqref{eq:resolved:energy} is proved in subsection~\ref{sec:energy}, Hypothesis \ref{hyp:helicity.prescription} is verified in Proposition~\ref{prop:hel,prescription}, Hypothesis~\ref{hyp:helinashell} in Proposition~\ref{prop:hel:shell}, Hypothesis~\ref{hyp:freq:loc} in Proposition~\ref{prop:hyp:freq:loc}, and Hypothesis~\ref{hyp:temp} in Lemma~\ref{lem:hel:derivs}.}
\end{proof}
 
\begin{proof}[Proof of Theorem~\ref{thm:main}]
In \texttt{Step 0}, we construct initial approximate solution $(u_1, p_1, R_1)$. 
Then, applying Proposition~\ref{prop:main} iteratively, we obtain a sequence of approximate solutions and construct a solution with desired regularity and total kinetic energy in \texttt{Step 1}. 
In \texttt{Step 2}, we check that~\eqref{e:sum.hel} is satisfied; that is, the helicity is absolutely summable over any choice of Littlewood-Paley decomposition.  In \texttt{Step 3}, we check \eqref{e:gen.hel.2} and the prescription of the helicity profile $h(t)$.
\bigskip

\noindent\texttt{Step 0: Construct $(u_1,p_1, R_1)$. } We first choose parameters. Let
\begin{align*}
    \bar{\la} = \la_0, 
    \quad
    \bar\la_{*}= \la_0^{\frac23 + \frac13(b^{\bn+1}-1)}, \quad
    \bar\tau= \la_0^{-1}\, , \quad \hat u_{0,R} = \bar{e}(t)  \bar{\varrho}_{R}(x)\vec e \, . 
\end{align*}
Here, we choose $a$ in the definition of $\la_0$ (see \eqref{eq:def:la:de}) sufficiently large so that $\bar{e}(t):=\sqrt{2e(t)-\de_{\bn+1}}\geq \left(\min_{t\in [0,T]} e(t)\right)^\frac12$,
and $\vec e := \frac35 e_2 + \frac45 e_3$. 
Also, $\bar{\varrho}_{R} = \bar{\la}^{-\dpot}\div^{\dpot} \bar{\vartheta}_{R}$, where $\bar\vartheta_{R}$ is smooth, $(\sfrac{\mathbb{T}}{\bar{\la}})^3$-periodic, has support contained in pipes of thickness $\bar{\la}^{-1}$, and satisfies
\begin{align*}
\vec e \cdot \na \vartheta_R =0, \quad
\norm{\na^N\bar{\vartheta}_{R}}_\infty
    \lec \bar{\la}^N, \quad \norm{\bar{\varrho}_{R}}_2^2 =1
\end{align*}
for all $N\leq 2\Nfin+\dpot$.  This implies that $\hat u_{0,R}$ solves stationary Euler with zero pressure and satisfies
\begin{align*}
    \norm{\hat u_{0,R}}_{L^2(\mathbb{T}^3)}^2 
    = 2e(t)-\de_\bn, \quad H( \hat u_{0,R}) =0\, . 
\end{align*}
Define $\hat u_{0, h,\pm}$ by
\begin{align*}
    \hat u_{0, h,\pm}
    = \frac{\bar{\la}}{\bar{\la}_*^{\sfrac32}}
\bar{h}_{\pm}(t)\bar{\varrho}_{h}(\Psi_{\pm}) (\na \Psi_{\pm})^{-1}e_2\, . 
\end{align*}
Here,   
$\bar{h}_{\pm}(t) = \mathcal{P}_{\bar\tau}\sqrt{(h(t)-2^{-1}\Ga_0^{-1})_{\pm}}$ 
and $\mathcal{P}_{\bar \tau}$ is a standard mollifier in time whose kernel is compactly supported on the interval $(-\bar\tau, \bar\tau)$.  The deformation $\Psi_{\pm}$ is defined by
\begin{align*}
    \Psi_{\pm}
    &=(x_1 \pm \bar{\la}^{-1}\sin(\bar{\la}_* x_2), x_2, x_3 + \bar{\la}^{-1}\cos(\bar{\la}_* x_2))^{\top}
    \end{align*}
Lastly, the density functions $\bar{\varrho}_{h}$ satisfy the same assumptions as $\bar{\varrho}_{R}$, except that $\na\bar{\varrho}_{h}$ is orthogonal to $e_2$ instead of $\vec e$. Moreover, 
$\bar{\varrho}_{R}$ and
$\bar{\varrho}_{h}$ are chosen to have disjoint supports. Then, one can compute
\begin{align*}
    \norm{\hat u_{0,h,\pm}}_2^2 \lec \frac{\bar{\la}^2}{\bar{\la}_*^3}, \quad
    H(\hat u_{0,h,\pm})
= \pm 
\bar h_{\pm}^2\, . 
\end{align*}

Set $\hat w_{q'} \equiv 0$ for $1\leq q'\leq \bn$, $p_0=0$, $u_1 = \hat u_{0}
    = \hat u_{0,R}
    +\hat u_{0,h, +}+ \hat u_{0,h, -} =:\hat u_{0,R}
    +\hat u_{0,h}$,
$R_1^k \equiv 0$ for $1< k \leq \bn $, and
\begin{align*}
    R_1^1 &= \hat u_{0,h} \otimes \hat u_{0,h}
    + \frac{\bar{\la}}{\bar{\la}_*^{\sfrac32}}
   \bar{h}_{ +}'\divH(\bar\varrho_{h}(\Psi_+)(\na\Psi_+)^{-1}e_2)
    + \frac{\bar{\la}}{\bar{\la}_*^{\sfrac32}}
    \bar{h}_{ -}'\divH(\bar\varrho_{h}(\Psi_-)(\na\Psi_-)^{-1}e_2)
    \\
     &\quad+ \bar{e}' 
     \begin{pmatrix}
         0 & (\bar{\la}^{-\dpot}\div^{\dpot-1}\bar\vartheta_{R})_2 &(\bar{\la}^{-\dpot}\div^{\dpot-1}\bar\vartheta_{R})_3 \\
         (\bar{\la}^{-\dpot}\div^{\dpot-1}\bar\vartheta_{R})_2  & 0 & 0\\
         (\bar{\la}^{-\dpot}\div^{\dpot-1}\bar\vartheta_{R})_3 &0 &0\, , 
     \end{pmatrix}
\end{align*}
where $\divH$ is an operator inverting the divergence as in~\cite[Proposition A.13]{GKN23}. Furthermore, set $\psi_{i,q'}=1$ if $i=0$ and $0$ otherwise for $0\leq q'\leq \bn$.

By construction, one can easily check that $(u_1,p_1, R_1)$ solves the Euler-Reynolds system and $R_1^k$ are symmetric matrices. Also, $\psi_{i,q'}$ for $0\leq q'\leq \bn$ satisfies all inductive hypotheses in Section \ref{sec:cutoff:inductive} for $q=1$. By the choices of $\bar{\la}$, $\bar{\la}_*$, and $\bar \tau$ so that $\frac{\bar \la^2}{\bar\la_*^3}\ll \de_{\bn+1}$ and $\bar\tau^{-1}\leq \mu_0$, $R_1^k$ satisfies the inductive hypotheses in Section \ref{sec:R:inductive} for $q=0$. Since $\hat w_{\bar q'} \equiv 0$ for $1\leq \bar q' \leq \bn$, $\hat u_{q'}=\hat u_0$ for $0\leq q'\leq \bn$, and $\bar\la =\la_0\de_0^{\frac12}$, the inductive hypotheses in Section \ref{ss:dodging} and \ref{sec:inductive:secondary:velocity} hold for $q=1$.  Finally, we consider the inductive energy and helicity bounds. Since $\hat u_{0,R}$ and $\hat u_{0, h}$ have disjoint supports, 
\begin{align*}
    \norm{u_1}_2^2 
    = \norm{\hat u_{0,R}}_2^2
    + \norm{\hat u_{0, h}}_2^2
    = 2e(t)-\de_{\bn+1} + O\left(\frac{\bar{\la}^2}{\bar{\la}_*^3}\right)
    \, .
\end{align*} 
Using $\frac{\bar\la^2}{\bar\la_*^3}\ll \de_{\bn+1}$, \eqref{eq:resolved:energy} holds true. As for the helicity, 
\begin{align}
    H(u_1) 
    &=H( \hat u_{0,R})
    + H(\hat u_{0, h,+})
    + H(\hat u_{0, h,-})
    + 2H(\hat u_{0, h,+}, \hat u_{0, h,-})
    =
    \bar{h}^2_{+} - 
    \bar{h}^2_{-}+ 2H(\hat u_{0, h,+}, \hat u_{0, h,-})\notag\\
    &= g_+^2 - g_-^2
    + 
(\mathcal{P}_{\bar \tau}g_+)^2 -
g_+^2 
   -
\left((\mathcal{P}_{\bar \tau}g_-)^2 -
g_-^2\right) + 2H(\hat u_{0, h,+}, \hat u_{0, h,-})\notag\\
&= h(t) - 2^{-1}\Ga_0^{-1}
    + O(\Ga_1^{-1})
    \label{hel.u_0}
\end{align}
where $g_{\pm}:= \sqrt{(h(t)-2^{-1}\Ga_0^{-1})_{\pm}}$.
To prove the last line, we write
\begin{align}\label{eqn:mol.g0}
   |(\mathcal{P}_{\bar \tau}-1) g_\pm(t)|
   \leq \int_{\mathbb{R}} |K_{\bar \tau}(s)| |g_{\pm}(t-s)
   - g_{\pm}(t)| ds
\end{align}
on $[0,T]$, and split into three cases. First if $\norm{h'}_{L^\infty([-1,T+1])}=0$, then $h$ and $g_{\pm}$ are constant functions. Therefore 
\begin{align}
\label{eqn:mol.g1}
|g_\pm(t-s)-g_\pm(t)|\equiv 0   
\end{align}
for any $t\in [0,T]$ and $|s|\leq \bar\tau$. Second, we now assume that $\norm{h'}_{L^\infty([-1,T+1])}>0$ and consider $t\in [0,T]$ such that $|h(t)-2^{-1}\Ga_0^{-1}|\leq 2\norm{h'}_{L^\infty([-1,T+1])}\bar \tau$. Then for any $|s|\leq \bar\tau$,
\begin{align}
   |h(t-s)-2^{-1}\Ga_0^{-1}| 
   &\leq 
|h(t-s)-h(t)|
+ |h(t)-2^{-1}\Ga_0^{-1}|
\leq 3\norm{h'}_{L^\infty([-1,T+1])}\bar \tau \, , \notag \\
\label{eqn:mol.g2}
    |g_\pm(t-s) - g_\pm(t)|
    &\leq
|g_\pm(t-s)|+ |g_\pm(t)|
\lec \bar \tau^\frac12\, . 
\end{align}
Lastly, if $\norm{h'}_{L^\infty([-1,T+1])}>0$ and  and $|h(t)-2^{-1}\Ga_0^{-1}|>2\norm{h'}_{L^\infty([-1,T+1])}\bar \tau$, then we have
\begin{align*}
|h(t-\tau)-2^{-1}\Ga_0^{-1}|
\geq 
|h(t)-2^{-1}\Ga_0^{-1}| - |h(t-\tau)-h(t)|
\geq
\norm{h'}_{L^\infty([-1,T+1])}\bar \tau
\end{align*}
for any $|\tau|\leq \bar\tau$, and 
\begin{align}\label{eqn:mol.g3}
    |g_\pm(t-s) - g_\pm(t)|
    \leq  \int_0^s 
    |g'(t-\tau)| d\tau 
    = \int_0^s
    \frac{|h'(t-\tau)|}
    {2\sqrt{(h(t-\tau)-2^{-1}\Ga_0^{-1}})_{\pm}} d\tau
    \lec 
\frac{\norm{h'}_{\infty} \bar\tau}{2 \sqrt{\norm{h'}_{\infty}\bar \tau}}
\lec \bar\tau^\frac12
\end{align}
for any $|s|\leq \bar \tau$. 
Combining the estimates \eqref{eqn:mol.g0}--\eqref{eqn:mol.g3}, we get
\begin{align*}
    |(\mathcal{P}_{\bar\tau} g_{\pm})^2- g_{\pm}^2)|
    \leq |(\mathcal{P}_{\bar\tau}-1)g_{\pm})| |(\mathcal{P}_{\bar\tau}+1)g_{\pm})|
    \lec \bar\tau^\frac12 \ll \Ga_1^{-1} . 
\end{align*}
On the other hand, to handle $H(\hat u_{0, h,+}, \hat u_{0, h,-})$, we note that the supports of
$\bar h_+$ and $\bar h_-$ overlap only on the set $\Omega_0:=\bigcup_{\{t_0: h(t_0)=2^{-1}\Ga_0^{-1}\}} [t_0-\bar \tau, t_0+ \bar \tau]$. When $t\in B(\Omega_0, \bar \tau) := \{s: |s-s_0|<\bar\tau, s_0\in \Omega_0\}$, we have $t\in [t_0-2\bar \tau, t_0+ 2\bar \tau]$ for some $t_0$ satisfying $h(t_0)=2^{-1}\Ga_0^{-1}$, and
\begin{align*}
   | h(t)- 2^{-1}\Ga_0^{-1}|
   \leq 
   |h(t)-h(t_0)|+
   | h(t_0)- 2^{-1}\Ga_0^{-1}|
   \leq 2\norm{h'}_{\infty}\bar \tau\, . 
\end{align*}
This implies that 
\begin{align*}
    \bar \la^{-N}\norm{\na^N \hat u_{0,h,\pm}}_{L^\infty(\Omega_0 \times \mathbb{T}^3)}
    \lec \frac{\bar \la}{\bar \la_*^{\frac32}}
 \bar\tau^{\frac12}, \quad 
 H(\hat u_{0,h,+}, \hat u_{0,h,-})
 \lec \frac{\bar \la^3}{\bar \la_*^3} \bar \tau  \ll \Ga_1^{-1}\, . 
\end{align*}
Therefore, we verified \eqref{hel.u_0}.

Now, we check the inductive assumptions for the helicity bounds in section \ref{sec:he:bounds}. 
Let $\ph:\mathbb{Z}^3\to [0,1]$ be any multiplier which satisfies $\ph(k)=1$ for $|k|\leq \la_{\bn}\Ga_{\bn}$. Then,
\begin{align*}
    H(u_1)- \int \bp_{\ph}(u_1)\cdot\curl u_1 
    &= \sum_{k} (1-\ph(k)) \mathcal{F}(u_1)(k) \cdot 2\pi i k \times \mathcal{F}(u_1)(-k)\\
    &= \sum_{k} (1-\ph(k)) |k|^{-2\dpot}  \mathcal{F}(u_1)(k) \cdot |k|^{2\dpot} 2\pi i k \times \mathcal{F}(u_1)(-k)\,  
\end{align*}
and 
\begin{align*}
\left|H(u_1)- \int \bp_{\ph}(u_1)\cdot\curl u_1\right|
&\leq \left(\sum_k |(1-\ph)|^2 |k|^{-4\dpot} |\mathcal{F}(u_1)(k)|^2\right)^{\frac12} \norm{\Delta^\dpot \curl u_1}_2\\
&\lec (\la_{\bn}\Ga_{\bn})^{-2\dpot} \bar{\la}^{2\dpot+1}
\leq \frac 1{10}\Ga_0^{-1}\, .
\end{align*}
This together with \eqref{hel.u_0}
implies \eqref{eq:resolved:helicity1} and \eqref{eq:resolved:helicity1a} when $q'=1$. The inequalities \eqref{summy} and \eqref{eq:resolved:helicity3} with $q'=1$ hold true because $ \hat w_{\bn}\equiv 0$.  Lastly, \eqref{hel:lip} holds for $q=1$ because $\bar\tau^{-1}\leq \mu_0^{-1}$. 

\bigskip

\noindent\texttt{Step 1: Constructing a weak Euler solution and prescribing the energy profile $e(t)$.}
With the initial approximate solution $(u_1, p_1, R_1)$ from \texttt{Step 0} at hand, we apply Proposition~\ref{prop:main} iteratively to obtain a sequence $\{u_q\}_{q=1}^{\infty}$ of the approximate solutions. Since $\be=\be(\bar b)\nearrow \sfrac 12$ as $\bar b \searrow 1$ (see \eqref{def:beta}), for any given $\bar \beta\in (0,\sfrac12)$, $\bar \beta< \beta <\sfrac12$ can be achieved by adjusting $\bar b$. Then, the usual interpolation argument implies that $\{u_q\}$ is Cauchy in $C_t(H^{\bar\beta}\cap L^{\frac1{1-2\bar\beta)}})$, leading to the desired regularity of $u:=\lim_{q\rightarrow \infty} u_q$. Furthermore, $u$ is a weak solution to the Euler equations by
the convergence of $u_q$ to $u$ in $L^2_{t,x}$ and that of $R_q$ to $0$ in $L^1_{t,x}$. Also, the energy prescription, (i.e., the first part of \eqref{e:profiles}) follows from \eqref{eq:resolved:energy} and the convergence of $u_q$ to $u$ in $C^0_t L^2$.

\bigskip

\noindent\texttt{Step 2: Checking~\eqref{e:sum.hel}. } 
Let $\td\varphi\in \mathcal{A}$.  We first note that since $\td\varphi$ is compactly supported, there exists $M=M(\varphi) \in \mathbb{N}$ such that the Littlewood-Paley kernels $\varphi_{2^m}$ and $\varphi_{2^{m'}}$ defined using $\td \varphi$ have disjoint supports if $|m-m'| > M$. We will check that
\begin{equation}\label{thu:to:sho}
\sum_{m,m'} \left| H\left( \bp_{\varphi_{2^m}}(u), \bp_{\varphi_{2^{m'}}}(u) \right) \right| \leq \sum_{m,m'} \sum_{q,q'} \left| H\left( \bp_{\varphi_{2^m}}(\hat w_{q}), \bp_{\varphi_{2^{m'}}}(\hat{w}_{q'}) \right) \right| < \infty \, .
\end{equation}
The first inequality above follows immediately from the triangle inequality and the fact that $u = \sum_{q} \hat w_q$, so it only remains to check the second inequality in~\eqref{thu:to:sho}.  We split
\begin{align}
\sum_{m,m'} \sum_{q,q'} \left| H\left( \bp_{\varphi_{2^m}}(\hat w_{q}), \bp_{\varphi_{2^{m'}}}(\hat{w}_{q'}) \right) \right| &= \sum_{m,m'} \sum_{q \neq q'} \left| H\left( \bp_{\varphi_{2^m}}(\hat w_{q}), \bp_{\varphi_{2^{m'}}}(\hat{w}_{q'}) \right) \right| \notag \\
&\qquad + \sum_{m,m'} \sum_q \left| H\left( \bp_{\varphi_{2^m}}(\hat w_{q}), \bp_{\varphi_{2^{m'}}}(\hat{w}_{q}) \right) \right| \notag\\
&= S_{\neq} + S_{=} \, .  \label{split1}
\end{align}

We first show that $S_{\neq} < \infty$. Morally speaking, this term should be finite because $\hat w_q$ and $\hat w_{q'}$ have frequency supports very far away from one another due to superexponential frequency growth, which combined with the restriction $|m-m'|\leq M$ and~\eqref{eq:resolved:helicity3} allows us to bound this term. To make this precise, we first note that due to the superexponential growth of $\lambda_{q}$ as specified in~\eqref{eq:def:la:de},
$$ \lim_{q \rightarrow \infty} \frac{\lambda_q \Gamma_q}{\lambda_{q+1}\Gamma_{q+1}^{-1}} = 0 \, . $$
As a consequence, there exists $q_0$, dependent on the support of $\td\varphi$, such that for all $\td q \geq q_0$ and all $q'' \neq \td q$, the following holds.  If 
$\varphi_{2^m}$ takes non-zero values on frequencies whose magnitude belongs to $[\lambda_{q''+\bn-1}\Gamma_{q''+\bn-1}^{-1}, \lambda_{q''+\bn-1}\Gamma_{q''+\bn-1}]$,
then $\varphi_{2^m}(k) = 0$ for all frequencies $k$ such that $|k| \in [\lambda_{\td q+\bn-1}\Gamma_{\td q+\bn-1}^{-1},\lambda_{\td q+\bn-1}\Gamma_{\td q+\bn-1}]$.  Applying this fact to all terms in $S_{\neq}$ where either $q$ or $q'$ is larger than $q_0$, we find that either $\hat w_q$ or $\hat w_{q'}$ satisfies one of the estimates from~\eqref{eq:resolved:helicity3}.  The same fact applies also to all terms in $S_{\neq}$ where $q,q' < q_0$ but either $m>\log_2 \left( \lambda_{q_0+\bn}\Gamma_{q_0+\bn} \right)$ or $m'> \log_2 \left( \lambda_{q_0+\bn}\Gamma_{q_0+\bn} \right)$.  Thus we can write
\begin{align}
    &S_{\neq} \leq \sum_{m,m'} \sum_{\substack{q \neq q' \\ q, q' < q_0}} \left| H\left( \bp_{\varphi_{2^m}}(\hat w_{q}), \bp_{\varphi_{2^{m'}}}(\hat{w}_{q'}) \right) \right| + 2 \sum_{m,m'} \sum_{\substack{q \neq q' \\ q \geq q_0}} \left| H\left( \bp_{\varphi_{2^m}}(\hat w_{q}), \bp_{\varphi_{2^{m'}}}(\hat{w}_{q'}) \right) \right| \notag  \\
    &\leq \sum_{\substack{m,m'\\ m,m' \leq \log_2 \left( \lambda_{q_0+\bn}\Gamma_{q_0+\bn} \right) }} \sum_{\substack{q \neq q' \\ q, q' < q_0}} \left| H\left( \bp_{\varphi_{2^m}}(\hat w_{q}), \bp_{\varphi_{2^{m'}}}(\hat{w}_{q'}) \right) \right| \notag \\
    & + 2\sum_{\substack{m,m'\\ m > \log_2 \left( \lambda_{q_0+\bn}\Gamma_{q_0+\bn} \right) }} \sum_{\substack{q \neq q' \\ q, q' < q_0}} \left| H\left( \bp_{\varphi_{2^m}}(\hat w_{q}), \bp_{\varphi_{2^{m'}}}(\hat{w}_{q'}) \right) \right|  + 2\sum_{m,m'} \sum_{\substack{q \neq q' \\ q \geq q_0}} \left| H\left( \bp_{\varphi_{2^m}}(\hat w_{q}), \bp_{\varphi_{2^{m'}}}(\hat{w}_{q'}) \right) \right| \notag \\
    &\leq \sum_{\substack{m,m'\\ m,m' \leq \log_2 \left( \lambda_{q_0+\bn}\Gamma_{q_0+\bn} \right) }} \sum_{\substack{q \neq q' \\ q, q' < q_0}} \left| H\left( \bp_{\varphi_{2^m}}(\hat w_{q}), \bp_{\varphi_{2^{m'}}}(\hat{w}_{q'}) \right) \right| \notag \\
    &\qquad + 2 \sum_{\substack{m,m'\\ m > \log_2 \left( \lambda_{q_0+\bn}\Gamma_{q_0+\bn} \right) }} \sum_{\substack{q \neq q' \\ q, q' < q_0}}  {\delta_{q'}^{\sfrac 12} 
    \Gamma_{q'}^{100} \frac{\lambda_q^2}{(2^m)^2} }\notag\\
    &\qquad + 2\sum_{\substack{m,m' \\ m \leq \log_2 \left( \lambda_{q+\bn-1}\Gamma_{q+\bn-1}^{-1} \right)}} \sum_{\substack{q \neq q' \\ q \geq q_0}} \frac{(2^m)^2}{\lambda_{q+\bn-1}^2} \delta_{q'}^{\sfrac 12} \Gamma_{q'}^{100} + 2\sum_{\substack{m,m' \\ m \geq \log_2 \left( \lambda_{q+\bn-1}\Gamma_{q+\bn-1} \right)}} \sum_{\substack{q \neq q' \\ q \geq q_0}} \frac{\lambda_{q+\bn-1}^2}{(2^m)^2} \delta_{q'}^{\sfrac 12} \Gamma_{q'}^{100} \, . \label{mess}
\end{align}
The first sum on the right-hand side in~\eqref{mess} contains finitely many terms and is therefore finite.  The second contains finitely many different values of $q$ and $q'$, and then since $|m-m'| \leq M$ and we have $2^m$ in the denominator, this term is finite.  Next, the third term is summable in $q'$, while the summability in $m$ and $q$ follows from the fact that $\frac{2^m}{\lambda_{q+\bn-1}} \rightarrow 0$ superexponentially as $\rightarrow \infty$ due to the restriction that $m \leq \log_2( \lambda_{q+\bn-1} \Gamma_{q+\bn-1}^{-1})$ and the superexponential growth of frequencies (summability in $m'$ follows from $|m-m'|\leq M$).  Similarly, the fourth term is summable in $q'$, and in $q$ and $m$ since $\frac{\lambda_{q+\bn-1}}{2^m} \rightarrow 0$ superexponentially
as $q \rightarrow \infty$ due to the condition $m \geq \log_2\left( \lambda_{q+\bn-1} \Gamma_{q+\bn-1} \right)$ and the superexponential growth of frequencies.  Thus $S_{\neq} < \infty$. 

Moving to $S_=$, we first consider the cases when 
$\varphi_{2^m}$ and $\varphi_{2^{m'}}$ are supported  
\emph{outside} the active frequency support of $\hat w_q$; that is, when either $\varphi_{2^m}$ or $\varphi_{2^{m'}}$ satisfies~\eqref{eq:resolved:helicity3}.  By arguments quite similar to those from $S_{\neq}$, these terms are negligible and therefore absolutely summable, and we omit further details.  Thus we need only consider the cases when $\varphi_{2^m}$ and $\varphi_{2^{m'}}$ have supports which intersect the set of frequencies whose modulus takes values in $[\lambda_{q}\Gamma_q^{-1}, \lambda_q\Gamma_q]$.  In such case, we claim that there exists $q_1$ such that for all $q \geq q_1$,~\eqref{schwartz:decay} is satisfied by any kernel  {$
\|\check{\varphi}\|_1^{-2}
\check{\varphi}_{2^m}\ast \check{\varphi}_{2^{m'}}$. }
Indeed the first condition in~\eqref{schwartz:decay} follows obviously from the Young's convolution inequality. 
 {For the second condition, we choose a ball $B_R$ which contains the support of $\td\varphi$. This implies that $\supp(\varphi_{2^m})\subset B_{2^mR}$. Since $\supp(\varphi_{2^m})$ intersects the set of frequencies whose modulus takes values in $[\lambda_{q}\Gamma_q^{-1}, \lambda_q\Gamma_q]$, we have $2^mR>\lambda_{q}\Gamma_q^{-1}$. Similarly, $2^{m'}R>\lambda_{q}\Gamma_q^{-1}$. Then, by the Schwartz decay of the base kernel $\varphi = \td \varphi(k) - \td \varphi(2k)$, 
for any $p \in \mathbb{N}$, there exists $C(p, \td\varphi)$ such that 
$$ |\check\varphi(x)| \leq \frac{C(p, \td \varphi)}{1+|x|^p} \qquad \implies \left| 2^{3m} \check \varphi(2^mx) \right| \leq \frac{2^{3m}C(p, \td \varphi)}{1+|2^mx|^p} \, . $$
Using that $2^m, 2^{m'} >\lambda_q \Gamma_q^{-1}R^{-1}$, we have that for $p$ sufficiently large and $q\geq q_0$, where $q_0$ depends on $p$ and thus $\td\varphi$,
\begin{align*}
\|\check{\varphi}_{2^m}\ast \check{\varphi}_{2^{m'}}\|_{L^1(|x|\geq \lambda_q^{-1} \Gamma_q^2)}
&\leq     
\int_{|x| \geq \lambda_q^{-1} \Gamma_q^2}\left[\int_{|y|\geq \frac12 \la_q^{-1}\Ga_q^2}
+\int_{|y|\leq \frac12 \la_q^{-1}\Ga_q^2}
\right]
2^{3(m+m')}
|\check{\varphi}(2^my)||\check{\varphi}(2^{m'}(x-y))| dy dx\\
&\leq
\|\check{\varphi}\|_1
\int_{|y|\geq \frac12 \la_q^{-1}\Ga_q^2}
2^{3m}|\check{\varphi}(2^my)| dy\\
&\qquad 
+
\int_{|y|\leq \frac12 \la_q^{-1}\Ga_q^2}2^{3(m+m')}|\check{\varphi}(2^my)|
\int_{|x-y| \geq\frac12 \lambda_q^{-1} \Gamma_q^2}
|\check{\varphi}(2^{m'}(x-y))|dx dy \\
&\lec_{C(p,\td\varphi), R} 
\Ga_q^{-p+3}\leq \|\check{\varphi}\|_1^2\la_q^{-1000}\, .
\end{align*}
Therefore, 
$
\|\check{\varphi}\|_1^{-2}
\check{\varphi}_{2^m}\ast \check{\varphi}_{2^{m'}}$ satisfies
\eqref{schwartz:decay}, and \eqref{summy} holds for $\hat w_q$.}
We have reduced the finiteness of $S_=$ to the finiteness of 
\begin{align*}
&\sum_{\left\{\substack{m,m',q \, : \\ |m-m'| \leq M, \\ \supp(\varphi_{2^m}, \varphi_{2^{m'}})\cap [\lambda_q\Gamma_q^{-1}, \lambda_q \Gamma_q]\neq \emptyset , \\ \eqref{summy}\textnormal{ holds}}\right\}} 
\left| H\left( \bp_{\varphi_{2^{m'}}}\bp_{\varphi_{2^m}}(\hat w_{q}), \hat{w}_{q} \right) \right| \\
&\les \norm{\check{\varphi}}_1^2\sum_q \sum_{\left\{ \substack{ m,m' \, : \\ \supp(\varphi_{2^m}, \varphi_{2^{m'}})\cap [\lambda_q\Gamma_q^{-1}, \lambda_q \Gamma_q]\neq \emptyset } \right\} } \Gamma_{q-\bn-4}^{-1} \\
&\les \norm{\check{\varphi}}_1^2\sum_q \Gamma_{q-\bn-4}^{-1} \log_2(\lambda_q \Gamma_q) 
\les \norm{\check{\varphi}}_1^2\sum_q \Gamma_{q-\bn-4}^{-1} b^q \log_2(a) < \infty \, ,
\end{align*}
where we have used~\eqref{eq:def:la:de}.  This concludes the proof of~\eqref{e:sum.hel}.

\bigskip

\noindent\texttt{Step 3: Checking~\eqref{e:gen.hel.2}.} 
Let $\varphi$ be a Schwarz function with unit mean. We start with the case that $\phi:=\hat\varphi\in \mathcal{A}$, and aim to check that for any given $\bar\varepsilon>0$, there exists $\epsilon_0=\epsilon_0(\bar\varepsilon)$ such that
\begin{equation}
\label{helicitiy.small}
\left| h(t) - H\left( \varphi_\epsilon\ast u, u \right) \right|
=\left| h(t) - H\left( \bp_{\phi_{\epsilon^{-1}}} u, u \right) \right|
< \bar\varepsilon \, .  
\end{equation}
holds true for any $\epsilon\in(0, \epsilon_0)$, where $\phi_{\epsilon^{-1}}:= \phi (\epsilon \cdot)$.  
Let $q_{\epsilon}$ be the smallest integer such that $\lambda_{q_{\epsilon} + \bn} \Gamma_{q_{\epsilon} + \bn}^{-1} > |k|$ for all $k$ in the support of $\phi_{\epsilon^{-1}}$. We decompose the helicity above into
\begin{align*}
      H\left( \bp_{\phi_{\epsilon^{-1}}}u, u \right) 
    &=
      H\left( \bp_{\phi_{\epsilon^{-1}}}u_{q_{\bar m}+1}, u_{q_{\bar m}+1} \right) +H\left( \bp_{\phi_{\epsilon^{-1}}}(u-u_{q_{\bar m}+1}), u_{q_{\bar m}+1} \right)
    +  H\left( \bp_{\phi_{\epsilon^{-1}}}u, u- u_{q_{\bar m}+1} \right) \\
    &=: H_1(t) + H_2(t)+ H_3(t)\, . 
\end{align*}
Since none of the frequencies in the support of $\phi_{\epsilon^{-1}}$ are in the support of the active frequencies of $u - u_{q_{\bar m}+1}$, by repeating arguments similar to \texttt{Step 2}, we can choose $\epsilon_0$ small enough so that $|H_2(t)| + |H_3(t)| < \sfrac {\bar\varepsilon}5$. Moving to $H_1(t)$, we split
\begin{align*}
&H\left( \bp_{\phi_{\epsilon^{-1}}}u_{q_{\bar m}+1}, u_{q_{\bar m}+1} \right)\\ &= 
H\left( \bp_{\phi_{\epsilon^{-1}}}u_{q_{\bar m}-1}, u_{q_{\bar m}-1} \right) + 2  H\left( \bp_{\phi_{\epsilon^{-1}}}(u_{q_{\bar m}-1}), \hat w_{q_{\bar m}+\bn} + \hat w_{q_{\bar m} + \bn -1} \right) \\
&\qquad 
+   H\left( \bp_{\phi_{\epsilon^{-1}}}(\hat w_{q_{\bar m}+\bn} + \hat w_{q_{\bar m} + \bn -1}), \hat w_{q_{\bar m}+\bn} + \hat w_{q_{\bar m} + \bn -1} \right) =: H_{11}(t) + H_{12}(t) + H_{13}(t)\, .
\end{align*}
Using arguments similar to those in \texttt{Step 2}, which are predicated on the disjoint active frequency support of separate velocity increments, we can choose $ \epsilon_0$  sufficiently small so that $|H_{12}(t)|<\sfrac {\bar\varepsilon} 5$.  In addition, exploiting Hypothesis ~\ref{hyp:helinashell} as in \texttt{Step 2},
we can choose $ \epsilon_0$ small enough so that $|H_{13}|<\sfrac {\bar\varepsilon} 5$.  
Next, by choosing $\epsilon_0$ sufficiently small based on $\varphi$, we can ensure that $\phi_{\epsilon^{-1}}$ takes value $1$ when
$|k|\leq \lambda_{q_{\epsilon} + \bn-2} \Gamma_{q_{\epsilon} + \bn-2}$, which is possible due to the superexponential growth of frequencies and $\phi\in\mathcal{A}$.
Therefore $\phi_{\epsilon^{-1}}$ satisfies the assumption for the multiplier required in~\eqref{eq:resolved:helicity1}, and so
\begin{align*}
\left|
    h(t)-  H\left( \bp_{\phi_{\epsilon^{-1}}}(u_{q_{\epsilon}-1}), u_{q_{\epsilon}-1} \right) \right| \leq \Ga_{q_{\epsilon}-2}^{-1} < \frac {\bar\varepsilon} 5 
\end{align*}
if $\epsilon$ is sufficiently large.   
Combining all the estimates, we have \eqref{helicitiy.small}.

We now outline the steps needed to prove~\eqref{e:gen.hel.2} for general Schwarz functions $\varphi$ with unit mean. Noting that $\phi:=\hat \varphi$ satisfies $\phi(0)=1$,
let $\{\phi^{\delta}\}$ for $\delta>0$ be a family of multipliers such that the following hold. First, $\phi^{\delta}(z) \equiv \phi(z)$ for all $2 \delta \leq |z| \leq \delta^{-1}$, $\phi^{\delta}(z) \equiv 1$ for all $|z| \leq \delta$, and $\phi^{\delta}(z)=0$ for all $|z| \geq 2\delta^{-1}$. Next, we construct $\phi^\delta$ so that there exists $C$ such that for all $\delta\in (0,1)$,\footnote{$\phi^\delta$ can be constructed by adjusting $\phi$ so that it is constant in a ball of radius $\approx \delta$ around the origin, then mollifying only near the origin with a mollifier at scale $\approx\delta$.}
\begin{equation}\label{eq:mollifier:bounds}
    \| \nabla^M ( \phi - \phi^\delta ) \|_{L^\infty(B_{2\delta}(0))} \leq C \delta^{1-M} \, , \qquad M=0,1,2,3,4 \, . 
\end{equation}
Then $\phi^{\delta}\in \mathcal{A}$ for $\delta\in(0,1)$, and it is enough to prove that there exists a small $\delta=\delta(\bar\varepsilon, \varphi)>0$ such that for any small $\epsilon>0$,
\begin{align}
\left|  H(\varphi_\eps \ast u, u) - H(\bp_{\phi^{\delta}_{\epsilon^{-1}}}u, u)\right|
=
\left|   H(\bp_{\phi_{\epsilon^{-1}}-\phi^{\delta}_{\epsilon^{-1}}}u, u)\right|
<\bar\varepsilon.
\label{t1}
\end{align}
To this end, observe that $\|\phi_{\epsilon^{-1}}-\phi^{\delta}_{\epsilon^{-1}}\|_{\infty}$ and $\|\check\phi_{\epsilon}-\check{\phi}^{\delta}_{\epsilon}\|_{1}$ are of order $\delta$, since
\begin{align*}
\|\phi_{\epsilon^{-1}}-\phi^{\delta}_{\epsilon^{-1}}\|_{\infty} 
&\lec 
\|\check\phi_{\epsilon}-\check{\phi}^{\delta}_{\epsilon}\|_{1}
=\|\check\phi- \check{\phi^{\delta}}\|_1
\leq 
\|\check\phi- \check{\phi^{\delta}}\|_{L^1(B_{\de^{-1}})}
+\|\check\phi- \check{\phi^{\delta}}\|_{L^1(B_{\de^{-1}}^c)}\\
&\lec \de^{-3}\|\check\phi- \check{\phi^{\delta}}\|_{\infty}
+\de\||x|^4(\check\phi- \check{\phi^{\delta}})\|_{L^\infty}
\lec \de^{-3}\| \phi- \phi^\delta\|_1
+ \de\|\Delta^2(\phi- \phi^{\delta})\|_1\lec_{\|\hat\varphi\|_{\mathcal{S}}} \de,
\end{align*}
where the last line follows from the Schwartz decay of $\phi$ and~\eqref{eq:mollifier:bounds}. Using this and the linearity of \eqref{summy}--\eqref{eq:resolved:helicity3} in the kernel $\varphi$, repeating earlier arguments gives the bound $|   H(\bp_{\phi_{\epsilon^{-1}}-\phi^{\delta}_{\epsilon^{-1}}}u, u)|\lec_{\|\hat\varphi\|_{\mathcal{S}}}\de$, leading to \eqref{t1}. \end{proof}

\section{Helical building blocks}\label{sec:helicalbb}

\subsection{Mikado bundles and other preliminaries}
We start with a series of preliminary facts from~\cite{GKN23, BMNV21}, with some slight adjustments related to the helicity.
\begin{customprop}{4.1, \cite{GKN23}}[\textbf{Geometric lemma}]\label{p:split}
Let $\Xi= {\Xi_0}=\left\{\sfrac 35 e_i \pm \sfrac 45 e_j\right\}_{1\leq i < j\leq 3}$. {Let $\bn\in \mathbb{N}$ be a given natural number.} Then there exists $\epsilon>0$ and disjoint sets  {$\Xi_0, \Xi_1, \dots, \Xi_{\bn-1} \subset \mathbb{Q}^3 \cap \mathbb{S}^2$} such that every symmetric 2-tensor in $B(\Id, \epsilon)$ can be written as a unique, positive linear combination of {$\xi_j \otimes \xi_j$ for $\xi_j \in \Xi_j$, $1\leq j \leq \bn$}.\index{$\Xi_j$} 
i.e., there exists $\ga_{\xi_j}\in C^\infty(B(\Id, \epsilon);(0,\infty))$ such that
\begin{align*}
     R = \sum_{{\xi_j\in\Xi_j}}\left(\gamma_{{\xi_j}}(R)\right)^2  {\xi_j\otimes \xi_j}, \quad\forall R\in B(\Id, \epsilon)\, . 
\end{align*}
\end{customprop}

\begin{definition}
[\textbf{Vector directions}]\label{def:vec}
The sets  {$\Xi_0, \dots, \Xi_{\bn-1}$} are obtained by applying rational rotations to the fixed set $\Xi$ and will be used to produce simple tensors which cancel Reynolds stresses; we choose the rational rotation matrices such that none of  {$\Xi_0, \dots, \Xi_{\bn-1}$} contain $\vec e_2$.  We also require $\bn$-many vector directions  {$\vec h_0, \dots, \vec h_{\bn-1}$} used to correct the helicity profile. We set $ {\vec h_0} = \vec e_2$ and then apply rational rotation matrices to obtain $\vec h_2, \dots \vec h_n$, ensuring that for all $1 \leq j_1 \neq  j_2 \leq \bn$,
\begin{equation}
    \left(\Xi_{j_1} \cup \{\vec h_{j_1}\} \right) \cap 
    \left(\Xi_{j_2} \cup \{\vec h_{j_2}\} \right) = \emptyset \, , 
\end{equation}
so that vector directions used at a given step of the iteration are always distinct from the vector directions used in the previous $\bn-1$ steps of the iteration.
\end{definition}

The next results are nearly verbatim from~\cite{GKN23}, the only adjustment being the incorporation of the vector directions needed to correct the helicity.

\begin{customprop}{4.4,~\cite{GKN23}}[\textbf{Rotating, Shifting, Periodizing}]\label{prop:pipe:shifted}
Fix $ \xi \in  {\bigcup_{j=0}^{\bn-1}} \left(  \Xi_j \cup \{ \vec h_j \} \right) $ from Proposition \ref{p:split} and Definition~\ref{def:vec}.  For any such $\xi$, we choose $\xi',\xi'' \in \mathbb{Q}^3\cap\mathbb{S}^2$ such that $\{\xi,\xi',\xi''\}$ is an orthonormal basis of $\R^3$. We then denote by $n_\ast$ the least positive integer such that $n_\ast \xi, n_\ast \xi' n_\ast \xi'' \in \mathbb{Z}^3$ for all $\xi \in  {\bigcup_{j=0}^{\bn-1}} \left(  \Xi_j \cup \{ \vec h_j \} \right)$. Let ${r^{-1},\lambda\in\mathbb{N}}$ be given such that $\lambda r\in\mathbb{N}$. Let $\varkappa:\mathbb{R}^2\rightarrow\mathbb{R}$ be a smooth function with support contained inside a ball of radius $\sfrac{1}{4}$. Then for $k\in\{0,...,r^{-1}-1\}^2$, there exist functions $\varkappa^k_{\lambda,r,\xi}:\mathbb{R}^3\rightarrow\mathbb{R}$ defined in terms of $\varkappa$, satisfying the properties listed in~\cite[Proposition~4.4]{GKN23}.
\end{customprop}

We now recall the basic intermittent pipe flows used in~\cite{BMNV21, GKN23}.  

\begin{customprop}{4.5,~\cite{GKN23}}[\bf Intermittent pipe flows for Reynolds stress]
\label{prop:pipeconstruction}
Fix a vector direction \\ $\xi \in  {\bigcup_{j=0}^{\bn-1}} \left(  \Xi_j \cup \{ \vec h_j \} \right) $ from Proposition~\ref{p:split} and Definition~\ref{def:vec}, $r^{-1},\lambda \in \mathbb{N}$ with $\lambda r\in \mathbb{N}$, and large integers $\Nfin$ and $\Dpot$. There exist vector fields $\mathcal{W}^k_{\xi,\lambda,r}:\mathbb{T}^3\rightarrow\mathbb{R}^3$ for $k\in\{0,...,r^{-1}-1\}^2$ and implicit constants depending on $\Nfin$ and $\Dpot$ but not on $\lambda$ or $r$ such that:
\begin{enumerate}[(1)]
    \item\label{item:pipe:1} There exists $\varrho:\mathbb{R}^2\rightarrow\mathbb{R}$ given by the iterated divergence $\div^\Dpot  \vartheta =: \varrho$ of a pairwise symmetric tensor potential $\vartheta:\mathbb{R}^2\rightarrow\mathbb{R}$ with compact support in a ball of radius $\frac{1}{4}$ such that the following holds.  Let $\varrho_{\xi,\lambda,r}^k$ and $\vartheta_{\xi,\lambda,r}^k$ be defined as in Proposition~\ref{prop:pipe:shifted}, in terms of $\varrho$ and $\vartheta$ (instead of $\varkappa$).  Then there exists $\mathcal{U}^k_{\xi,\lambda,r}:\mathbb{T}^3\rightarrow\mathbb{R}^3$ such that
    if $\{\xi,\xi',\xi''\} \subset \mathbb{Q}^3 \cap \mathbb{S}^2$ form an orthonormal basis of $\R^3$ with $\xi\times\xi'=\xi''$, then we have
    \begin{equation}
    \mathcal{U}_{\xi,\lambda,r}^k
     =  - \frac 13  \xi' \underbrace{\lambda^{-\Dpot} \xi''\cdot \nabla \left(\div^{\Dpot-2} \left(\vartheta_{\xi,\lambda,r}^k \right)\right)^{ii}}_{=:\varphi_{\xi,\lambda,r}^{\prime \prime k}}
     +  \frac 13
     \xi'' \underbrace{\lambda^{-\Dpot} \xi'\cdot \nabla \left(\div^{\Dpot-2} \left(\vartheta_{\xi,\lambda,r}^k \right)\right)^{ii}
     }_{=:\varphi_{\xi,\lambda,r}^{\prime k}}
     \label{eq:UU:explicit}
        \,, 
    \end{equation}
and thus
\begin{align}
\label{eq:WW:explicit}
\curl \mathcal{U}^k_{\xi,\lambda,r} = \xi \lambda^{-\Dpot }\div^\Dpot  \left(\vartheta^k_{\xi,\lambda,r}\right) &= \xi \varrho^k_{\xi,\lambda,r} =: \mathcal{W}^k_{\xi,\lambda,r}
\,, \qquad 
    \xi \cdot \nabla \vartheta_{\xi,\lambda,r} = 0
    \,.
\end{align}
    \item\label{item:pipe:2} $\{\mathcal{U}_{\xi,\lambda,r}^k\}_{k}$, $\{\varrho_{\xi,\lambda,r}^k\}_{k}$, $\{\vartheta_{\xi,\lambda,r}^k\}_{k}$, and $\{\mathcal{W}_{\xi,\lambda,r}^k\}_{k}$ satisfy the conclusions from~\cite[Proposition~4.4]{GKN23}.
    \item\label{item:pipe:3} $\mathcal{W}^k_{\xi,\lambda,r}$ is a stationary, pressureless solution to the Euler equations.
    \item\label{item:pipe:4} $\displaystyle{\dashint_{\mathbb{T}^3} \mathcal{W}^k_{\xi,\lambda,r} \otimes \mathcal{W}^k_{\xi,\lambda,r} = \xi \otimes \xi }$.
    \item\label{item:pipe:5} For all $n\leq 3 \Nfin$, 
    \begin{align}\label{e:pipe:estimates:1}
    \left\| \nabla^n\vartheta^k_{\xi,\lambda,r} \right\|_{L^p(\mathbb{T}^3)} &\lesssim \lambda^{n}r^{\left(\frac{2}{p}-1\right)} , \qquad {\left\| \nabla^n\varrho^k_{\xi,\lambda,r} \right\|_{L^p(\mathbb{T}^3)} \lesssim \lambda^{n}r^{\left(\frac{2}{p}-1\right)} } \\
\label{e:pipe:estimates:2}
    \left\| \nabla^n\mathcal{U}^k_{\xi,\lambda,r} \right\|_{L^p(\mathbb{T}^3)} &\lesssim \lambda^{n-1}r^{\left(\frac{2}{p}-1\right)} , \qquad {\left\| \nabla^n\mathcal{W}^k_{\xi,\lambda,r} \right\|_{L^p(\mathbb{T}^3)} \lesssim \lambda^{n}r^{\left(\frac{2}{p}-1\right)} }.
    \end{align}
\item\label{item:pipe:3.5} We have that $\supp \vartheta_{\xi,\lambda,r}^k \subseteq B\left( \supp\varrho_{\xi,\lambda,r} ,2\lambda^{-1}\right)$.
    \item\label{item:pipe:6} Let $\Phi:\mathbb{T}^3\times[0,T]\rightarrow \mathbb{T}^3$ be the periodic solution to the transport equation $\partial_t \Phi + v\cdot\nabla \Phi =0$ with initial datum $
\Phi|_{t=t_0} = x$, with a smooth, divergence-free, periodic velocity field $v$. Then $\nabla \Phi^{-1} \cdot \left( \mathcal{W}^k_{\xi,\lambda,r} \circ \Phi \right) = \curl \left( \nabla\Phi^T \cdot \left( \mathcal{U}^k_{\xi,\lambda,r} \circ \Phi \right) \right)$.
\item\label{item:pipe:7} For any convolution kernel $K$, $\Phi$ as above, $A=(\nabla\Phi)^{-1}$, and for $i=1,2,3$,
\begin{align}
\bigg{[} \nabla \cdot \bigg{(} A \, K\ast \left(  \mathcal{W}^k_{\xi,\lambda,r} \otimes  \mathcal{W}^k_{\xi,\lambda,r} \right)(\Phi) A^T \bigg{)} \bigg{]}_i 
& = A_m^j \xi^m \xi^l \partial_jA_{l}^i \, K\ast \left( \left( \varrho^k_{\xi,\lambda,r} \right)^2(\Phi) \right) \, .
\label{eq:pipes:flowed:2}
\end{align}
\end{enumerate}
\end{customprop}

We now recall the bundling pipe flows. Since only quadratic cancellations are used in this paper, some simplifications have been made relative to~\cite[Proposition~4.9]{GKN23}.

\begin{customprop}{4.9,~\cite{GKN23}}[\textbf{``Bundling'' pipe flows $\rhob_{\xi}^k$}]\label{prop:bundling}
Fix a vector direction $\xi \in  {\bigcup_{j=0}^{\bn-1}} \left(  \Xi_j \cup \{ \vec h_j \} \right) $. Then for $k\in\{1,\dots,\Gamma_q^6\}$, there exist bundling pipe flows $\rhob_{\xi}^k$ satisfying the following.
\begin{enumerate}[(1)]
    \item\label{i:bundling:1} $\rhob^k_{\xi}$ is $\left( \sfrac{\T}{\lambda_{q+1}\Gamma_q^{-4}}\right)^3$-periodic and satisfies $\xi \cdot \nabla \rhob^k_{\xi}\equiv 0$.
    \item\label{i:bundling:2} $\{\rhob^k_{\xi}\}_{k}$ satisfies the conclusions of Proposition~\ref{prop:pipe:shifted} with $r^{-1}=\Gamma_q^3$, $\lambda=\lambda_{q+1}\Gamma_q^{-1}$.  In particular, $\supp \rhob^k_{\xi} \cap \supp \rhob^{k'}_{\xi} = \emptyset$ for $k\neq k'$, and there are $\Gamma_q^6$ disjoint choices of placement.
    \item\label{i:bundling:3} $\displaystyle \int_{\T^3} \rhob_{\xi,k}^2=1$.
    \item\label{i:bundling:4} For all $n\leq 3\Nfin$ and $p\in[1,\infty]$,
    $\left\| \nabla^n \rhob^k_{\xi} \right\|_{L^p(\mathbb{T}^3)} \lesssim \left(\Gamma_q^{-1}\lambda_{q+1}\right)^n \Gamma_q^{-3\left(\frac 2p -1\right)}$.
\end{enumerate}
\end{customprop}

Now we recall the anistropic cutoffs, on which we will chose placements for different bundles.
\begin{customdef}{4.10,~\cite{GKN23}}[\textbf{Strongly anisotropic cutoffs}]\label{def:etab}
To each $\xi \in  {\bigcup_{j=0}^{\bn-1}} \left(  \Xi_j \cup \{ \vec h_j \} \right)$, we associate a partition of the orthogonal space $\xi^\perp \in \T^3$ into a grid\footnote{We refer to the grid used in Proposition~\ref{prop:pipe:shifted}, as any periodicity issues have been avoided there.} of squares of sidelength $\approx\lambda_{q+\half}^{-1}$. We index the squares $\mathcal{S}$ in this partition by $I$. To this grid, we associate a partition of unity $ \etab_{\xi}^I$, i.e.,
    \begin{align}\label{eq:sa:summability}
         \etab_\xi^{I} =\begin{cases}
        1 &\textnormal{  on  } \frac34 \mathcal{S}_{I}\\
        0 &\textnormal{  outside  } \frac54 \mathcal{S}_{I}
        \end{cases}, \qquad
        \sum_{I} (\etab_\xi^{I})^{2} = 1 \, ,
    \end{align}
    which in addition satisfies $(\xi\cdot\na) \etab_\xi =0$ and $\left\| \nabla^N \etab_\xi^I \right\|_\infty \lesssim \lambda_{q+\half}^N$ for all $N\leq 3\Nfin$ and all $I$. We can bound the cardinality of the index set $I$ by  $|\{I \in \mathcal{S} \}| \lec \la_{q+\half}^2$.
\end{customdef}

We now recall the \emph{intermittent pipe bundles} from~\cite[Definition~4.12]{GKN23}.  We do not define these for the vectors $\vec h_j$, as those vectors will be used to define distinct ``helical pipe bundles'' later.

\begin{definition}[\textbf{Intermittent pipe bundles}]\label{defn:pipe.bundle}
For each $\xi \in  {\bigcup_{j=0}^{\bn-1}}\Xi_j \subset\mathbb{Q}^{3} \cap \mathbb{S}^2$, we define intermittent pipe bundles by $\BB_{\xi} =\chib_{\xi} \sum_I  \etab_{\xi}^I \WW_{\xi}^I =\chib_{\xi} \sum_I  \etab_{\xi}^I \xi \varrho_{\xi}^I$, where $\chib_{\xi} = \chib_{\xi}^m$ is defined as in Proposition \ref{prop:bundling} for some $m=m_{\xi}$, $\WW_{\xi}^I:=\mathcal{W}^{k}_{\xi,\lambda_{q+\bn},\sfrac{\lambda_{q+\half}\Gamma_q}{\lambda_{q+\bn}}}$, and $\varrho_{\xi}^I:= \varrho^{k}_{\xi,\lambda_{q+\bn},\sfrac{\lambda_{q+\half}\Gamma_q}{\lambda_{q+\bn}}}$ is defined as in Propositions~\ref{prop:pipeconstruction} for some $k$.
\end{definition}

\subsection{Helical pipe bundles}
We now define the helical pipe bundles. 
\begin{proposition}[\textbf{Helical pipe bundles}]\label{prop:pipe:helical}
Let $\Nfin, \Dpot \in \mathbb{N}$ be large integers.  Define 
\begin{align}\notag
    \Psi_q^{\pm} (x_1,x_2,x_3) &= (x_1 \mp \la_{q+\bn}^{-1}\sin(\la_\qho x_2), x_2, x_3 + \la_{q+\bn}^{-1}\cos(\la_\qho x_2))^{\top}  \notag \\
\label{defn:mathcal.H}
    \mathbb{H}_{\pm \vec e_2,\lambda_\qbn, r_q}^k
    &:= \tilde e_2^{\pm} \,  \varrho_{\vec e_2, \lambda_\qbn, r_q}^k (\Psi_q^\pm)  \,  , \qquad \qquad   \tilde e_2^\pm = (\nabla \Psi_q^\pm)^{-1}  \vec e_2
    \end{align}
where $\varrho^k_{\vec e_2,\lambda_\qbn, r_q}$ is defined in~Proposition~\ref{prop:pipeconstruction}.  We define helical pipe bundles by
\begin{equation}
            \BB_{ \vec e_2,\pm} = \rhob_{ \vec e_2} \sum_I \zetab_{\vec e_2}^I \mathbb{H}^I_{{\pm}  \vec e_2,\lambda_\qbn, r_q} \, . \label{eq:hel:bundle}
        \end{equation}
The $I$ indicates that the placement (i.e. the value of $k$) depends on the cutoff $\zetab_{ \vec e_2}^I$. Then there exist implicit constants depending on $\Nfin$ and $\Dpot$ but not on $q$ such that the following hold.
\begin{enumerate}[(1)]
    \item\label{item:hpipe:1} $\varrho=\varrho_{\vec e_2,\lambda_\qbn,r_q}:\mathbb{R}^2\rightarrow\mathbb{R}$ is given by the iterated divergence $\div^\Dpot  \vartheta =: \varrho$ of a pairwise symmetric tensor potential $\vartheta:\mathbb{R}^2\rightarrow\mathbb{R}$ with compact support in a ball of radius $\approx \lambda_{\qbn}^{-1}$ such that the following holds.  Let $\varrho^k$ and $\vartheta^k$ be defined as in Proposition~\ref{prop:pipe:shifted}, in terms of $\varrho$ and $\vartheta$ (instead of $\varkappa$).  Then there exists $\mathbb{U}_{\pm  \vec e_2, \lambda_\qbn, r_q}^k:\mathbb{T}^3\rightarrow\mathbb{R}^3$ such that
    if $\{ \vec e_2,\xi',\xi''\} \subset \mathbb{Q}^3 \cap \mathbb{S}^2$ form an orthonormal basis of $\R^3$ with $ \vec e_2\times\xi'=\xi''$, then we have
    \begin{equation}
    \mathbb{U}_{\pm  \vec e_2, \lambda_\qbn, r_q}^k
     =  - \frac 13  \xi' \underbrace{\lambda_{q+\bn}^{-\Dpot} \xi''\cdot \nabla \left(\div^{\Dpot-2} \left(\vartheta^k \right)\right)^{ii}}_{=:\varphi^{\prime \prime k}}
     +  \frac 13
     \xi'' \underbrace{\lambda_{q+\bn}^{-\Dpot} \xi'\cdot \nabla \left(\div^{\Dpot-2} \left(\vartheta^k \right)\right)^{ii}
     }_{=:\varphi^{\prime k}}
     \label{eq:hUU:explicit}
        \,, 
    \end{equation}
and thus
\begin{equation}
\label{eq:hWW:explicit}
\curl \mathbb{U}_{\pm  \vec e_2, \lambda_\qbn, r_q}^k = \xi \lambda_{q+\bn}^{-\Dpot }\div^\Dpot  \left(\vartheta^k\right) =  \vec e_2 \varrho^k
\,,
\end{equation}
and 
\begin{equation}
     \vec e_2 \cdot \nabla \vartheta^k = ( \vec e_2 \cdot \nabla )\mathbb{U}_{\pm \vec e_2, \lambda_\qbn, r_q}^k
    = 0
    \,.
    \label{eq:hderivative:along:pipe}
    \end{equation}
\item\label{item:pipeh:2} The sets $\{\mathbb{U}_{\pm  \vec e_2, \lambda_\qbn, r_q}^k\}_{k}$, $\{\varrho^k\}_{k}$, and $\{\vartheta^k\}_{k}$ satisfy the conclusions from~\cite[Proposition~4.4]{GKN23}.
    \item\label{item:hpipe:2} For all $k \in \mathbb{N} \cup \{0\}$,
    \begin{align}\label{psi:q:bounds}
   \| \Psi_q^\pm \|_\infty \les 1 \, , \qquad  \|\nabla \Psi_q^\pm - \Id \|_\infty \les \frac{\lambda_\qho}{\lambda_\qbn} \, , \qquad   \| \nabla^{1+k} \Psi_q^\pm \|_\infty \les \lambda_{q+\bn}^{-1} \lambda_\qho^{1+k} \, .
    \end{align} 
    \item\label{item:hpipe:5} For all $n\leq 3 \Nfin$, 
    \begin{equation}\label{e:hpipe:estimates:1}
    {\| \nabla^n\vartheta^k \|_{L^p(\mathbb{T}^3)} \lesssim \lambda_\qbn^{n}r_q^{\left(\sfrac{2}{p}-1\right)} }, \qquad {\| \nabla^n\varrho^k \|_{L^p(\mathbb{T}^3)} \lesssim \lambda_\qbn^{n}r_q^{\left(\sfrac{2}{p}-1\right)} }
    \end{equation}
    and abbreviating $\mathcal H^k_{\pm  \vec e_2, \lambda_\qbn, r_q}$ by $\mathcal H^k$, we have the bounds ${\| \nabla^n\mathbb{U}_{\pm  \vec e_2, \lambda_\qbn, r_q}^k \|_{L^p(\mathbb{T}^3)} \lesssim \lambda_\qbn^{n-1}r_q^{\left(\sfrac{2}{p}-1\right)} }$ and ${\left\| \nabla^n\mathbb{H}^k \right\|_{L^p(\mathbb{T}^3)} \lesssim \lambda_\qbn^{n}r_q^{\left(\sfrac{2}{p}-1\right)} }$.
\item\label{item:hpipe:3.5} We have that $\supp \vartheta^k \subseteq B\left( \supp\varrho ,2\lambda_\qbn^{-1}\right)$.
    \item\label{item:hpipe:6} Let $\Phi:\mathbb{T}^3\times[0,T]\rightarrow \mathbb{T}^3$ be the periodic solution to the transport equation
\begin{align}
\label{e:hphi:transport}
\partial_t \Phi + v\cdot\nabla \Phi =0\,, 
\qquad 
\Phi|_{t=t_0} &= x\, ,
\end{align}
with a smooth, divergence-free, periodic velocity field $v$. Then
\begin{equation}\label{eq:hpipes:flowed:1}
\nabla \Phi^{-1} \cdot ( \mathbb{H}^k \circ \Phi ) = \curl ( \nabla(\Psi_q^\pm \circ\Phi)^T \cdot ( \mathbb{U}_{\pm  \vec e_2, \lambda_\qbn, r_q}^k \circ \Psi_q^\pm \circ \Phi ) ).
\end{equation}
\item\label{item:hpipe:7}  The helical pipe bundles have the property that for each fixed $I$, 
\begin{align}\label{helicity.mathcalH}
 {\pm} (\rhob_{\vec e_2} \zetab_{\vec e_2}^I)^2 &= \left( \frac{\lambda_\qbn}{\la_\qho^{\sfrac 32}}  \right)^2 (\rhob_{\vec e_2} \zetab_{\vec e_2}^I)^2 \dashint_{\left(\frac{\T}{\lambda_{q+\bn} r_q}\right)^3} \mathbb{H}^k_{{\pm}{\vec e_2},\lambda_\qbn, r_q} \cdot \curl \mathbb{H}^k_{{\pm} {\vec e_2},\lambda_\qbn, r_q} \notag\\
 &= \left( \frac{\lambda_\qbn}{\la_\qho^{\sfrac 32}} \right)^2 (\rhob_{ \vec e_2} \zetab_{ \vec e_2}^I)^2 \dashint_{\T^3} \mathbb{H}^k_{{\pm} \vec e_2,\lambda_\qbn, r_q} \cdot \curl \mathbb{H}^k_{{\pm} \vec e_2,\lambda_\qbn, r_q} \, .
\end{align}
\item\label{item:hpipe:8} There exists a $q$-independent geometric constant $C_\varrho >0$ such that
    \begin{align}\label{helical.in.straight}
            \supp \mathbb{H}^k_{{\vec e_2}, \la_{q+\bn}, r_q}(\cdot) \subset B(\supp \varrho^k_{{\vec e_2}, \la_{q+\bn},r_q}(\cdot), C_\varrho \la_{q+\bn}^{-1})\,.
        \end{align}
    \end{enumerate}
\end{proposition}

\begin{proof}
{Since $\la_\qho$ and $\la_{q+\bn}$ are powers of two, $\Psi^\pm_q$ is a diffeomorphism of $\T^3$.}
After computing
    \begin{align*}
    \nabla \Psi_q^\pm &= \begin{pmatrix}
        1 & \mp\frac{\la_\qho}{\la_{q+\bn}}\cos(\la_\qho x_2) & 0 \\
        0 & 1 & 0 \\
        0 & -\frac{\la_\qho}{\la_{q+\bn}}\sin(\la_\qho x_2) & 1
    \end{pmatrix} \, , \quad (\nabla \Psi_q^\pm)^{-1} = \begin{pmatrix}
        1 & \pm\frac{\la_\qho}{\la_{q+\bn}}\cos(\la_\qho x_2) & 0 \\
        0 & 1 & 0 \\
        0 & \frac{\la_\qho}{\la_{q+\bn}}\sin(\la_\qho x_2) & 1
    \end{pmatrix} \, ,
\end{align*}
item~\ref{item:hpipe:2} follows immediately.  Note importantly that $\Psi_q^\pm$ is a volume-preserving diffeomorphism, since the determinant of $\nabla\Psi_q^\pm$ is one.  Then the proofs of item~\ref{item:hpipe:1},~\ref{item:pipeh:2},~\ref{item:hpipe:5},~\ref{item:hpipe:3.5}, and~\ref{item:hpipe:6} are identical to the corresponding proofs in Proposition~\ref{prop:pipeconstruction}, and we omit the details. 

We focus now on the proof of item~\ref{item:hpipe:7}. By~\eqref{defn:mathcal.H},
\begin{align}
\curl \mathbb{H}_\pm &= \left(\na (\varrho(\Psi_q^\pm)) \times \tilde e_2^\pm + \varrho(\Psi_q^\pm)\, \curl \tilde e_2^\pm\right) \, , \label{eq:hel.building1}\\
    \tilde e_2^\pm &= \left( \pm    \frac{\la_\qho}{\la_{q+\bn}}\cos(\la_\qho x_2), 1,  \frac{\la_\qho}{\la_{q+\bn}}\sin(\la_\qho x_2) \right)^\top \, ,  \label{eq:hel.building2}\\
\curl \tilde e_2^\pm &=\left( \frac{\la_\qho^2}{\la_{q+\bn}}\cos(\la_\qho x_2), 0, \pm\frac{\la_\qho^2}{\la_{q+\bn}} \sin(\la_\qho x_2) 
    \right)^\top \, , \label{eq:hel.building3}
\end{align}
and so we apply the Pythagorean theorem and deduce that $\curl \tilde e_2^\pm \cdot \tilde e_2^\pm = \pm \frac{\la_\qho^3}{\la_{q+\bn}^2}$.  Thus
\begin{align*}
    \left( \frac{\lambda_\qbn}{\la_\qho^{\sfrac 32}} \right)^2 \int_{\T^3} \mathbb{H}_\pm \cdot \curl \mathbb{H}_\pm \, dx 
    = \left( \frac{\lambda_\qbn}{\la_\qho^{\sfrac 32}} \right)^2 \int_{\T^3} \rho(\Psi^\pm)^2 \curl \tilde e_2^\pm \cdot \tilde e_2^\pm \, dx
    = \pm 
    \int_{\T^3} \rho(\Psi_q^\pm)^2 \, dx = \pm 1 \, .
\end{align*}

Finally, we must prove item~\ref{item:hpipe:8}.  Since $\| \Psi_q^\pm - \Id \|_\infty \leq \lambda_\qbn^{-1}$, the support of $\rho(\Psi_{\bn})$ is contained in a $\lambda_\qbn^{-1}$ ball around the support of $\rho$, which gives the result. 
\end{proof}

\section{Mollification and upgrading material derivatives}\label{ss:mollification}

The details in this section are largely similar to~\cite[Section~3]{GKN23}, and so we omit them.

\begin{lemma}[\bf Mollification and upgrading material derivative estimates]\label{lem:upgrading}
Assume that \emph{all} inductive assumptions listed in subsections~\ref{ss:relaxed}-\ref{sec:he:bounds} hold. Let $\Pqxt$ be a space-time mollifier with kernel a product of $\mathcal{P}_{q,x}(x)$, which is compactly supported in space at scale $\Lambda_q^{-1}\Gamma_{q-1}^{-\sfrac 12}$, and $\mathcal{P}_{q,t}(t)$, which is compactly supported in time at scale $\Tau_{q-1}\Gamma_{q-1}^{\sfrac 12}$; we assume that both kernels have vanishing moments up to $10\Nfin$ and are $C^{10\Nfin}$-differentiable. Define\index{$\Pqxt$}\index{$R_\ell$}\index{$\pi_\ell$}\index{$\varphi_\ell$} the trace-free, symmetric $2$-tensor
\begin{align}
    &R_\ell := \Pqxt R_q^q - \frac 13 \tr \left( \Pqxt R_q^q \right) \Id     \label{def:mollified:stuff}
\end{align}
on the space-time domain $[-\sfrac{\tau_{q-1}}2,T+\sfrac{\tau_{q-1}}2]\times \T^3$. 
Then the following hold. 
\begin{enumerate}[(i)]
\item\label{item:moll:one} The following relaxed equation (replacing \eqref{eqn:ER}
) is satisfied:
\begin{align}\notag
&\partial_t u_q + \div (u_q \otimes u_q) + \nabla \left( p_q - \frac 13 \tr R_q^q \right) = \div\left(R_\ell + \left( R_q^q - \frac 13 (\tr R_q^q)\Id - R_\ell \right) + \sum\limits_{k=q+1}^{q+\bn-1} R_q^k  \right) \, . 
\end{align}
\item\label{item:moll:two}
The inductive assumptions in subsection~\ref{sec:cutoff:inductive}, \ref{ss:dodging},~\ref{sec:inductive:secondary:velocity}, and~\ref{sec:he:bounds} remain unchanged.
\item\label{item:moll:two:2}  The inductive bounds for $R_q^q$ in~\eqref{eq:pressure:inductive} are replaced by the following: for all $N+M \leq \Nfin$,
\begin{subequations}\label{eq:pressure:upgraded}
\begin{align}
\norm{ \psi_{i,q} D^N D_{t,q}^M  R_\ell }_{1}  
&\les \Gamma_q^2 \delta_{q+\bn}\left(\Lambda_q\Gamma_q\right)^N \MM{M, \NindRt, \Gamma_{q}^{i} \mu_q^{-1}, \Tau_{q}^{-1} } \, ,
\label{eq:pressure:inductive:dtq:upgraded} \\
\norm{ \psi_{i,q} D^N D_{t,q}^M R_\ell }_{\infty} 
&\les \Gamma_q^{2+\badshaq} \left(\Lambda_q\Gamma_q\right)^N \MM{M, \NindRt, \Gamma_{q}^{i} \mu_q^{-1}, \Tau_{q}^{-1} } \, . \label{eq:pressure:inductive:dtq:uniform:upgraded}   
\end{align}
\end{subequations}
\item\label{item:moll:four} The symmetric tensor $R_{\ell} + \sfrac 13 \tr R_q^q \Id -R_q^q$ satisfies the following: for all $N+M\leq 2\Nind$,
\begin{align}\notag    \left\| D^N \Dtq^M \left( R_{\ell} + \sfrac 13 (\tr R_q^q)\Id -R_q^q \right) \right\|_{\infty} \les \Gamma_{q+1} \Tau_{q+1}^{ 4\Nindt} \delta_{q+3\bn}^2 \lambda_{q+1}^N \MM{M,\Nindt, \mu_q^{-1},\Gamma_{q}^{-1}\Tau_{q}^{-1}} \, .
\end{align}
\end{enumerate}
\end{lemma}

\section{New velocity increment}

\subsection{Preliminaries}\label{subsec:cutoff} 
We recall basic facts about cutoff functions and flow maps.

\subsubsection{Time cutoffs}
\label{sec:cutoff:temporal}
We introduce two collections of basic temporal cutoffs similar to~\cite[subsection~5.1]{GKN23}; the only differences are that sixth powers have been substituted for second powers, and that $\mu_q$ has been substituted for $\tau_q$.  The first set of temporal cutoff functions $\chi_{i,k,q}(t)$ satisfies
\begin{subequations}\label{eq:chi}
\begin{align}
t\in \supp \chi_{i,k,q} \iff t &\in \left[ (k-1)\sfrac 12 \mu_q \Gamma_{q}^{-i-2}, (k+1) \sfrac 12 \mu_q \Gamma_{q}^{-i-2} \right]  \, , \label{eq:chi:support}\\
|\partial_t^m \chi_{i,k,q}| &\les (\Gamma_{q}^{i+2} \mu_{q}^{-1})^m \qquad \textnormal{for }m\geq 0 \, ,
\label{eq:chi:cut:dt} \\
\chi_{i,k_1,q}(t)\chi_{i,k_2,q}(t) &\equiv 0 \qquad \qquad\textnormal{unless }|k_1-k_2|\leq 1 \,  \label{e:chi:overlap}\\
 \sum_{k \in \Z} \chi_{i,k,q}^2 &\equiv 1 \, .
 \label{eq:chi:cut:partition:unity}
\end{align}
\end{subequations}
The second set of temporal cutoff functions $\tilde\chi_{i,k,q}$ is only needed for technical reasons (see Lemma~\ref{lem:dodging}) and is supported in the set of times $t$ satisfying
\begin{equation}\label{eq:chi:tilde:support}
\left| t-\mu_q \Gamma_{q}^{-i-2} k \right| \leq \mu_q\Gamma_{q}^{-i} \,.
\end{equation}
In addition, if $(i,k)$ and $(\istar,\kstar)$ are such that $\supp \chi_{i,k,q} \cap \supp \chi_{\istar,\kstar,q}\neq\emptyset$ and $\istar\in\{i-1,i,i+1\}$, then with a sufficiently large choice of $\lambda_0$, we have that $\supp \chi_{i,k,q} \subset \supp \tilde\chi_{\istar,\kstar,q}$.

\subsubsection{Velocity cutoffs}\index{$\psi_{i,q}$}\index{velocity cutoffs}  Recall from subsection~\ref{sec:inductive:secondary:velocity} that velocity cutoffs have already been constructed up to $\qbn-1$. The new velocity cutoffs $\psi_{i,\qbn}$ are constructed in Definition~\ref{def:psi:i:q:def}.

\subsubsection{Flow maps}\label{s:deformation}
\label{sec:cutoff:flow:maps}

We first recall two results from~\cite[Section~5.2]{GKN23}.  The only difference is cosmetic-- the timescale $\tau_q$ has been replaced by $\mu_q$.

\begin{lemma}[\bf Lagrangian paths don't jump many supports]
\label{lem:dornfelder}
Let $q'\leq q+\bn-1$ and $(x_0,t_0)$ be given. Assume that the index $i$ is such that $\psi_{i,q'}^2(x_0,t_0) \geq \kappa^2$, where $\kappa\in\left[\frac{1}{16},1\right]$. Then the forward flow $(X(t),t) := (X(x_0,t_0;t),t)$ of the velocity field $\hat u_{q'}$ originating at $(x_0,t_0)$ has the property that $\psi_{i,q'}^2(X(t),t) \geq\sfrac{\kappa^2}{2}$ for all $t$ such that $|t - t_0|\leq \mu_{q'} \Gamma_{q'}^{-i+4}$.
\end{lemma}

\begin{corollary}[\bf Backwards Lagrangian paths don't jump many supports]
\label{cor:dornfelder}
Suppose $(x_0,t_0)$ is such that $\psi^2_{i,q'}(x_0,t_0)\geq \kappa^2$, where $\kappa\in\left[\sfrac{1}{16},1\right]$. For $\abs{t-t_0}\leq\mu_{q'}\Gamma_{q'}^{-i+3}$, define $x$ to satisfy $x_0=X(x,t;t_0)$. Then we have that $\psi_{i,q'}(x,t)\neq 0$.
\end{corollary}

We can now define and estimate the flow\index{flow maps}\index{$\Phi_{i,k,q}$, $\Phi_{(i,k)}$} of the vector field $\hat u_{q'}$ for $q'\leq q+\bn-1$ on the support of a velocity and time cutoff function. These estimates are completely analogous to \cite[subsection~5.2]{GKN23}.  We do note, however, that we have sharpened the estimates in~\eqref{eq:Lagrangian:Jacobian:1}--\eqref{eq:Lagrangian:Jacobian:2}, which is important for our estimations of the helicity of the velocity increment.

\begin{definition}[\bf Flow maps]\label{def:transport:maps} We define $\Phi_{i,k,q'}(x,t)=\Phi_{(i,k)}(x,t)$ as the solution to
\begin{equation}\label{e:Phi}
(\partial_t + \hat u_{q'} \cdot\nabla) \Phi_{i,k,q'} = 0 \, \qquad \qquad 
\Phi_{i,k,q'}(x,k{\mu_{q'}}\Gamma_{q'}^{-i-2})=x\, .
\end{equation}
\end{definition}
\noindent We denote the inverse of $D\Phi_{(i,k)}$ by $(D\Phi_{(i,k)})^{-1}$, in contrast to $D\Phi_{(i,k)}^{-1}$, the gradient of $\Phi_{(i,k)}^{-1}$.

\begin{corollary}[\bf Deformation bounds]
\label{cor:deformation}
For $k \in \Z$, $0 \leq i \leq  i_{\rm max}$, $q'\leq q+\bn-1$, and $2 \leq N \leq \sfrac{3\Nfin}{2}+1$, we have the following bounds on the set $\Omega_{i,q',k}:=\supp \psi_{i,q'}(x,t){\tilde\chi_{i,k,q'}(t)}$:
\begin{subequations}
\begin{align}
&\norm{D\Phi_{(i,k)} - {\rm Id}}_{L^\infty(\Omega_{i,q',k})} + \norm{(D\Phi_{(i,k)})^{-1} - {\rm Id}}_{L^\infty(\Omega_{i,q',k})} \leq \tau_{q'}^{-1} \Gamma_{q'}^{-5} \mu_{q'} \, ,
\label{eq:Lagrangian:Jacobian:1}\\
&\norm{D^N\Phi_{(i,k)} }_{L^\infty(\Omega_{i,q',k})} + \norm{D^N\Phi^{-1}_{(i,k)} }_{L^\infty(\Omega_{i,q',k})} + \norm{D^{N-1}\left((D\Phi_{(i,k)})^{-1}\right) }_{L^\infty(\Omega_{i,q',k})}  \leq  \frac{\mu_{q'} (\lambda_{q'}\Ga_{q'})^{N-1}}{\tau_{q'} \Gamma_{q'}^{5}}  \, .\label{eq:Lagrangian:Jacobian:2}
\end{align}
\end{subequations}
Furthermore, we have the following bounds for $1\leq N+M\leq \sfrac{3\Nfin}{2}$ and $0\leq N'\leq N$:
\begin{subequations}
\begin{align}
\left\| D^{N-N'} D_{t,q'}^M D^{N'+1} \Phi_{(i,k)} \right\|_{L^\infty(\Omega_{i,q',k})} &\leq \frac{\mu_{q'} (\lambda_{q'}\Ga_{q'})^{N}}{\tau_{q'} \Gamma_{q'}^{5} } \MM{M,\NindSmall,\Gamma_{q'}^{i} \mu_{q'}^{-1},\Tau_{q'-1}^{-1}\Gamma_{q'-1}}\label{eq:Lagrangian:Jacobian:5}\\
\left\| D^{N-N'} D_{t,q'}^M D^{N'} (D \Phi_{(i,k)})^{-1} \right\|_{L^\infty(\Omega_{i,q',k})} &\leq \frac{\mu_{q'} (\lambda_{q'}\Ga_{q'})^{N}}{\tau_{q'} \Gamma_{q'}^{5} } \MM{M,\NindSmall,\Gamma_{q'}^{i} \mu_{q'}^{-1},{\Tau}_{q'-1}^{-1}\Gamma_{q'-1}} \label{eq:Lagrangian:Jacobian:6}
\end{align}
\end{subequations}
\end{corollary}
\begin{proof}
The proofs of~\eqref{eq:Lagrangian:Jacobian:1}--\eqref{eq:Lagrangian:Jacobian:2} are nearly identical to those in~\cite[Corollary~6.27]{BMNV21}. The estimates come from multiplying the size of $\nabla \hat u_{q'}$ given in~\eqref{eq:nasty:D:vq:old} on the support of $\Omega_{i,q',k}$ by the length of the timescale from~\eqref{eq:chi:tilde:support}, and paying the derivative cost of $\nabla \hat u_{q'}$ for each additional spatial derivative past the first.  Using Lemma~\ref{lem:dornfelder},~Corollary~\ref{cor:dornfelder},~\eqref{eq:nasty:D:vq:old} at level $q'$, and~\eqref{eq:defn:tau}, we obtain~\eqref{eq:Lagrangian:Jacobian:1}--\eqref{eq:Lagrangian:Jacobian:2}. Higher order derivatives are similar, and we omit further details.
\end{proof}

\subsubsection{Stress cutoffs}\label{sec:cutoff:stress}

In this subsection, we introduce the cutoff functions for the stress error $R_\ell$ obtained from mollifying $R_q^q$ in Lemma~\ref{lem:upgrading}. The definition is essentially a level-set chopping of $R_\ell$. Since $R_\ell$ satisfies estimates which are totally analogous to $\pi_\ell$ in \cite[Section~5.3]{GKN23} or $R_\ell$ in \cite[Section~5.3]{NV22}, we omit the details in the proofs of the following estimates.

\begin{lemma}[\bf Stress cutoffs]\label{lem:D:Dt:Rn:sharp}
There exist cutoff functions\index{$\omega_{j,q}$} $\{\omega_{j,q} \}_{j\geq 0}$ satisfying the following.
\begin{enumerate}[(i)]
    \item We have that
    \begin{align}
\sum_{j\geq 0} \omega_{j,q}^2(t,x) \equiv 1 \, , \qquad \omega_{j,q} \omega_{j',q} \equiv 0 \quad \textnormal{if} \quad |j-j'| >1  \, .
\label{eq:omega:cut:partition:unity}
\end{align}
\item We have that for $N+M\leq \Nfin$ and any $j\geq 0$,
\begin{subequations}
\begin{align}
{\bf 1}_{\supp (\omega_{j,q}\psi_{i,q})} | D^k D_{t,q}^m R_\ell (x,t) | 
&\leq \Ga_q^{2j+6}\de_{q+\bn}
(\Ga_q\La_q)^k \MM{m, \NindRt, \Gamma_{q}^{i} \mu_{q}^{-1} , \Tau_{q}^{-1} } \, .
\label{pt.est.pi.lem}\\
\sfrac 14 \delta_{q+\bn} \Ga_q^{2j} &\leq {\bf 1}_{\supp (\omega_{j,q})} |R_\ell| \label{pt.est.pi.lem.lower}
\end{align}
\end{subequations}
\item  There exists $\jmax(q) \leq \inf \{ j \, : \, \frac 14 \Ga_q^{2j} \delta_{q+\bn} \geq \Ga_q^{3+\badshaq} \}$, bounded independently of $q$, such that $\omega_{j,q} \equiv 0$ for all $\qquad j > \jmax = \jmax(q)$. Moreover, we have the bound $\Gamma_{q}^{2j_{\rm max}} \leq \de_{q+\bn}^{-1} \Gamma_q^{\badshaq+6}$. 
\item For $q\geq 0$, $0 \leq i \leq \imax$, $0 \leq j \leq \jmax$, and $N + M \leq \Nfin$, we have
\begin{align}
\frac{{\bf 1}_{\supp \psi_{i,q}} |D^N D_{t,q}^M \omega_{j,q}|}{\omega_{j,q}^{1-(N+M)/\Nfin}} 
\les (\Gamma_{q}^5 \Lambda_q)^N \MM{M, \Nindt,\Gamma_{q}^{i+4} \mu_{q}^{-1}, {\Tau}_{q}^{-1}} \, .
\label{eq:D:Dt:omega:sharp}
\end{align}
\item For any $r \geq 1$ and $0\leq j \leq \jmax$, we have that 
\begin{align}
\norm{\omega_{j,q}}_{L^r}  \lesssim  \Gamma_{q}^{\frac{2(1-j)}{r}} \, .
\label{eq:omega:support}
\end{align}
\end{enumerate}
\end{lemma}

\subsubsection{Anisotropic checkerboard cutoffs}
\label{sec:cutoff:checkerboard:definitions}

We recall from~\cite[subsection~5.4]{GKN23} the anisotropic cutoffs.  Aside from changing $\tau$ to $\mu$, these cutoffs are identical to those in~\cite{GKN23}. 

\begin{lemma}[\bf Mildly anistropic checkerboard cutoffs]\label{lem:checkerboard:estimates}
Given $q$, $\xi\in\Xi_0, \cdots, \Xi_{\bn-1}$ or $\xi\in \{h_0, \cdots, h_{\bn-1}\}$, $i\leq \imax$, and $k\in\mathbb{Z}$, there exist cutoff functions $\zeta_{q,i,k,{\xi},\vec{l}}\,(x,t) = \mathcal{X}_{q,\xi,\vec{l}}\left(\Phi_{i,k,q}(x,t)\right)$ which satisfy the following properties.
\begin{enumerate}[(i)]
\item\label{item:checkeeee} The support of $\mathcal{X}_{q,\xi,\vec{l}}$ is contained in a rectangular prism of dimensions no larger than $\sfrac 34 \lambda_q^{-1}\Gamma_q^{-8}$ in the direction of $\xi_z$, and $\sfrac 34 \Gamma_{q}^5(\lambda_{q+1})^{-1}$ in the directions perpendicular to $\xi_z$.
\item\label{item:check:1} The material derivative $\Dtq (\zeta_{q,i,k,{\xi},\vec{l}})$ vanishes.  
\item\label{item:check:2} We have the summability property for all $(x,t)\in \T^3\times \R$ 
\begin{subequations}\label{eq:checkerboard:partition}
    \begin{align}
    \sum_{\vec{l}} \bigl(\zeta_{q,i,k,{\xi},\vec{l}}\,(x,t)\bigr)^{2} &\equiv 1 \, . \label{eq:summy:summ:1} 
    \end{align}
    \end{subequations}
\item\label{item:check:3}  Let $A=(\nabla\Phi_{(i,k)})^{-1}$. Then for all $N_1+N_2+M\leq \sfrac{3\Nfin}{2}+1$,
    \begin{align}\label{eq:checkerboard:derivatives}
        \bigl\| D^{N_1} \Dtq^M  ({\xi^\ell A_\ell^j \partial_j} )^{N_2} \zeta_{q,i,k,\xi,\vec{l}} \bigr\|_{L^\infty\left(\supp \psi_{i,q}\tilde\chi_{i,k,q} \right)} &\lesssim \left(\Gamma_{q}^{-5} \lambda_{q+1} \right)^{N_1} \left(\Gamma_q^8\lambda_{q}\right)^{N_2} \notag\\
        &\quad \times \MM{M,\Nindt,\Gamma_{q}^{i}\mu_q^{-1},\Tau_q^{-1}\Gamma_{q}^{-1}} \, .
    \end{align}
\end{enumerate}
\end{lemma}

\begin{lemma}[\bf Strongly anisotropic checkerboard cutoff function]\label{lem:finer:checkerboard:estimates}
The cutoff functions $\etab_\xi^I\circ\Phiik$ defined using $\etab$ from Definition~\ref{def:etab} satisfy the following properties:
\begin{enumerate}[(1)]
\item\label{item:check:check:1} The material derivative $\Dtq (\etab_\xi^I\circ \Phiik)$ vanishes.
\item\label{item:check:check:2} For all $q,i,k,\xi$, $t\in\mathbb{R}$, and all $x\in\mathbb{T}^3$, $\sum_{I} (\etab_\xi^I \circ \Phiik)^2(x,t) = 1$.
\item\label{item:check:check:3} 
    Let $A=(\nabla\Phi_{(i,k)})^{-1}$.  Then for all $N_1+N_2+M\leq \sfrac{3\Nfin}{2}+1$, we have that
\begin{align}\notag
    \bigl\| D^{N_1} \Dtq^M  ({\xi^\ell A_\ell^j \partial_j} )^{N_2} \etab_\xi^I \circ \Phiik \bigr\|_{L^\infty\left(\supp \psi_{i,q}\tilde\chi_{i,k,q} \right)} &\lesssim \lambda_{q+ \sfrac \bn 2 }^{N_1}  \MM{M,\Nindt,\Gamma_{q}^{i}\mu_q^{-1},\Tau_q^{-1}\Gamma_{q}^{-1}} \, .
    \end{align}
    \item For fixed $q,i,k,\xi$, we have that $\# \left\{(\vecl,I) : \supp \left(\zeta_{q,i,k,\xi,\vecl} \, \etab^I_{\xi} \circ \Phi_{(i,k)}\right) \neq \emptyset \right\} \lec \Ga_q^8 \la_{q} \la_{q+\half}^2$.
\end{enumerate}
\end{lemma}

\subsubsection{Cumulative cutoff function}
\label{sec:cutoff:total:definitions}
We introduce the cumulative cutoff functions
\begin{subequations}
\begin{align}
\eta_{i,j,k,\xi,\vec{l} }\,(x,t) &:= \psi_{i,q}(x,t) \omega_{j,q}(x,t) \chi_{i,k,q}(t)\zeta_{q ,i,k,\xi,\vec{l}}\,(x,t)  \, , \qquad \xi \in \cup_{j=0}^{\bn-1}\Xi_j \, , \label{def:cumulative:current}\\
\eta_{i,k,\xi,\vecl}&:= \psi_{i,q}  \chi_{i,k,q}(t)\zeta_{q,i,k,\vec \xi,\vec{l}}\,(x,t) \, , \qquad \xi \in \{\vec h_0, \dots \vec h_{\bn-1}\}  \, . \label{def:cumulative:hel}
\end{align}
\end{subequations}
We have the following $L^p$ estimates, which may be proven directly using \eqref{eq:omega:support} and \eqref{eq:psi:i:q:support:old}.

\begin{lemma}[\bf Cumulative support bounds for cutoff functions]\label{lemma:cumulative:cutoff:Lp}
For $r_1,r_2\in [1,\infty]$ with $\frac{1}{r_1}+\frac{1}{r_2}=1$ and any $0\leq i \leq \imax$, $0\leq j,\leq \jmax$, we have that for each $t$,
\begin{align}
    \sum_{\vecl} \left| \supp_x \left( \eta_{i,k,\vec e_2,\vecl }(t,x) \right) \right|  + \sum_{\vecl} \left| \supp_x \left( \eta_{i,j,k,\xi,\vecl }(t,x) \right) \right| &\lesssim \Gamma_{q}^{\frac{-2i + \CLebesgue}{r_1} + \frac{-2j}{r_2}+3} \, . \label{eq:supp:cumul:varphi}
\end{align}
We furthermore have that
\begin{align}
\sum_{i,j,k,\xi,\vecl,I } \left( \mathbf{1}_{\supp \eta_{i,j,k,\xi,\vecl } \rhob_\pxi \zetab_\xi^I} + \mathbf{1}_{\supp \eta_{i,k,\xi,\vecl } \rhob_\pxi \zetab_\xi^I} \right) \approx \sum_{i,j,k,\xi,\vecl } \left( \mathbf{1}_{\supp \eta_{i,j,k,\xi,\vecl } \rhob_\pxi} + \mathbf{1}_{\supp \eta_{i,k,\xi,\vecl } \rhob_\pxi} \right) \lesssim 1 \, . \label{eq:desert:cowboy:sum}
\end{align}
\end{lemma}

\subsubsection{Cutoff aggregation lemmas}\label{sss:aggregation}

We now recall ``aggregation lemmas'' from~\cite[Section~5.6]{GKN23}, which allow us to sum over the indices in the cumulative cutoffs to produce useful global bounds in Lebesgue spaces.  Slight changes been made since $L^{\sfrac 32}$ bounds on stresses and $L^3$ bounds on velocities have been replaced with $L^1$ bounds on stresses and $L^2$ bounds on velocities, c.f. $\theta \in (0,2]$ rather than $(0,3]$.  For concision we state the lemmas for the cutoffs from~\ref{def:cumulative:current}; for the cutoffs from~\ref{def:cumulative:hel} identical statements hold with $\theta_2=0$ since these cutoffs do not depend on $j$.
\begin{corollary}[\bf Aggregated $L^p$ estimates]\label{rem:summing:partition}
Let $\theta\in (0,2]$, and $\theta_1,\theta_2\geq 0$ with $\theta_1+\theta_2=\theta = \sfrac 2p$. Let $H=H_{i,j,k,\xi,\vecl }$ or $H=H_{i,j,k,\xi,\vecl,I }$ be a function with 
\begin{align}
    \supp H_{i,j,k,\xi,\vecl } \subseteq \supp \eta_{i,j,k,\xi,\vecl } \qquad \textnormal{or} \qquad \supp H_{i,j,k,\xi,\vecl,I } \subseteq \supp  \eta_{i,j,k,\xi,\vecl } \, \zetab_{\xi}^{I }\circ\Phiik \, . \label{eq:agg:assump:1}
\end{align}
Assume that there exists $\const_H,N_*,M_*,N_x,M_t$ and $\lambda,\Lambda,\tau,\Tau$ such that
\begin{subequations}
\begin{align}
  \left\| D^N \Dtq^M H_{i,j,k,\xi,\vecl } \right\|_{L^p} &\lesssim \sup_{t\in\R} \left( \left| \supp_x \left( \eta_{i,j,k,\xi,\vecl } (t,x) \right) \right|^{\sfrac 1p} \right) \notag\\
  &\qquad \qquad \times \const_H \Gamma_q^{\theta_1 i + \theta_2 j}  \MM{N,N_x,\lambda,\Lambda} \MM{M,M_t,\tau^{-1}\Gamma_q^i,\Tau^{-1}} \label{eq:agg:assump:2} \\
  \left\| D^N \Dtq^M H_{i,j,k,\xi,\vecl,I } \right\|_{L^p} &\lesssim \sup_{t\in\R} \left( \left| \supp_x \left( \eta_{i,j,k,\xi,\vecl } \zetab_\xi^{I }\circ\Phiik  (t,x)\right) \right|^{\sfrac 1p} \right) \notag\\
  &\qquad \qquad \times \const_H \Gamma_q^{\theta_1 i + \theta_2 j} \MM{N,N_x,\lambda,\Lambda} \MM{M,M_t,\tau^{-1}\Gamma_q^i,\Tau^{-1}} \,  \label{eq:agg:assump:3}
\end{align}
\end{subequations}
for $N\leq N_*,M\leq M_*$. Then in the same range of $N$ and $M$,
\begin{subequations}
\begin{align*}
 \left\| \psi_{i,q} \sum_{i',j,k,\xi,\vecl } D^N \Dtq^M H_{i',j,k,\xi,\vecl } \right\|_{L^p} &\lesssim \Gamma_q^{{3+\theta_1 \CLebesgue}} \const_H \MM{N,N_x,\lambda,\Lambda} \MM{M,M_t,\tau^{-1}\Gamma_q^{i+1},\Tau^{-1}} \\
 \left\| \psi_{i,q} \sum_{i',j,k,\xi,\vecl,I } D^N \Dtq^M H_{i',j,k,\xi,\vecl,I } \right\|_{L^p} &\lesssim \Gamma_q^{{3+\theta_1 \CLebesgue}} \const_H \MM{N,N_x,\lambda,\Lambda} \MM{M,M_t,\tau^{-1}\Gamma_q^{i+1},\Tau^{-1}} \, .
\end{align*}
\end{subequations}
\end{corollary}

\subsection{Construction of the Euler-Reynolds perturbation}\label{ss:corr-ec-tor}

In this subsection, we define the Euler-Reynolds velocity increment, except for the choice of placement (see subsection \ref{sec:dodging}). None of the discussion in this subsection depend on the choice of placement. For fixed of $i$, $k$, define
\begin{align}
&R_{q,i,j,k} = \nabla\Phi_{(i,k)}
\left(e_{q+1}\Gamma_q^{2j}\Id -R_\ell  
\right)\nabla\Phi_{(i,k)}^T \, , \label{eq:rqnpj}
\end{align}
For all $\xi\in\Xi_{q \,  \rm{mod} \,  \bn}$, we define the coefficient function\index{$w_{q+1,R}$} $a_{\xi,i,j,k,\vecl}$ by\index{$a_{\xi,i,j,k,\vecl,R}$}
\begin{align}
a_{\xi,i,j,k,\vecl}
&=a_{(\xi)}=e_{q+1}
^{\sfrac12}
\Gamma^{j}_{q} 
\eta_{i,j,k,\xi,\vecl}\, \gamma_{\xi}\left(\frac{R_{q,i,j,k}}{e_{q+1}\Gamma_q^{2j}}\right)
\label{eq:a:xi:def}
\end{align}
where $\gamma_{\xi}$ is defined in Proposition~\ref{p:split} and
\begin{align*}
    e_{q+1}:=e_{q+1}(t)
    = \left(\frac{2e(t) - \|u_q\|_2^2 -\de_{q+\bn+1}}{3 \sum_j \int_{\T^3} \om_{j,q}^2 \Ga_q^{2j}}\right)\, . 
\end{align*}
In order to show that \eqref{eq:a:xi:def} is well-defined, we recall 
\eqref{pt.est.pi.lem.lower}, 
\eqref{eq:omega:cut:partition:unity}, and
\eqref{eq:pressure:inductive:dtq:upgraded} to get
\begin{align*}
    \sum_{j'\geq 0} \int \omega_{j',q}^2 \Ga_q^{2j'}
    \leq  4\de_{q+\bn}^{-1}\int  \sum_{j'\geq 0}\omega_{j',q}^2 |R_\ell| 
    \lec \Ga_q^2
\end{align*}
and further use \eqref{eq:Lagrangian:Jacobian:1}, \eqref{pt.est.pi.lem},
 and \eqref{eq:resolved:energy} to see that for all $j$,
\begin{equation}\label{pi:top:bottom}
\left|\left( {e_{q+1}\Gamma_q^{2j}}  \right)^{-1} {R_{q,i,j,k}|_{\supp \omega_{j,q}}}- \Id
    \right|\leq \Ga_q^{-1}\, .
\end{equation}
Also, by the inductive assumption \eqref{eq:resolved:energy}, the function $(2e(t) - \|u_q\|_2^2 -\de_{q+\bn+1})^{\sfrac12}$ is positive and smooth, and is of order $\de_{q+\bn}^{\sfrac12}$. 
The coefficient function $a_{(\xi)}$ is then multiplied by an intermittent pipe bundle $\nabla \Phi_{(i,k)}^{-1}  \BB_{(\xi)} \circ \Phi_{(i,k)}$, where we have used Proposition~\ref{prop:pipeconstruction} (with $\lambda=\lambda_{q+\bn}$ and $r=r_{q}$), Definition~\ref{defn:pipe.bundle}, and the shorthand notation
\begin{align} 
\mathbb{B}_{(\xi)} = \rhob_{(\xi)} \sum_I \zetab_{\xi}^{I} \WW_{(\xi)}^{I}
\label{eq:W:xi:q+1:nn:def}
\end{align}
to refer to the pipe bundle associated with the region $\Omega_0= \supp\zeta_{q,i,k,\xi,\vecl}\cap \{t=k\mu_q\Gamma_q^{-i-2}\}$ and the index $j$. We will use $\UU_{(\xi)}^I$ to denote the potential satisfying $\curl \UU_{(\xi)}^I=\WW_{(\xi)}^I$. Applying Proposition~\ref{prop:pipeconstruction}, we define the principal part of the Reynolds corrector by
\begin{equation}\label{wqplusoneonep}
    w_{q+1,R}^{(p)} = \sum_{i,j,k,\xi,\vecl,I} \underbrace{a_{(\xi)} \left(\chib_{(\xi)} \etab_{\xi}^{I} \right) \circ \Phiik  \curl \left(\nabla\Phi_{(i,k)}^T \mathbb{U}_{(\xi)}^I \circ \Phi_{(i,k)} \right)}_{=: w_{(\xi)}^{(p),I}}
    \, .
\end{equation}
The notation $w_{(\xi)}^{(p),I}$ refers to \emph{fixed} values of $i,j,k,\xi,\vecl,I$.  We add the divergence corrector
\begin{equation}\label{wqplusoneonec}
    w_{q+1,R}^{(c)} =\sum_{i,j,k,\xi,\vecl,I} \underbrace{\nabla  \left( a_{(\xi)} \left(\chib_{(\xi)} \etab_{\xi}^{I} \right) \circ \Phiik \right) \times \left(\nabla\Phi_{(i,k)}^T \mathbb{U}_{(\xi)}^I \circ \Phi_{(i,k)} \right)}_{=: w_{(\xi)}^{(c),I}}
    \, ,
\end{equation}
so that the mean-zero, divergence-free total Reynolds corrector is given by
\begin{equation}\label{wqplusoneone}
    w_{q+1,R} = \sum_{i,j,k,\xi,\vecl,I} \underbrace{\curl \left( a_{(\xi)} \left(\chib_{(\xi)} \etab_{\xi}^{I} \right) \circ \Phiik  \nabla\Phi_{(i,k)}^T \mathbb{U}_{(\xi)}^I \circ \Phi_{(i,k)} \right)}_{=: w_{(\xi)}^I} \, .
\end{equation}

\subsection{Construction of helical perturbation}
In this subsection, we define the helical velocity increment, except for the choice of placement.   Define $h_{q+1}^2$ by
\begin{equation}\label{defn:H}
    h_{q+1}^2 = |\T^3|^{-1} 
    (h(t)-H(u_q))\left(1-\Ga_{q-1} (2\Ga_{q})^{-1}\right)\, , 
\end{equation}
Note that the right hand side of \eqref{defn:H} is strictly positive due to the inductive assumption \eqref{eq:resolved:helicity1a}, so that $h_{q+1}$ is well-defined. For $\xi = \vec h_{q \, \rm{mod} \, \bn}$, we define the coefficient function $a_{\xi,\pm,i,k,\vecl}$ by
\begin{align}
\label{eq:a:xi:h:def}    a_{\xi, i,k,\vecl} = a_{\pxi} = \lambda_\qbn \la_\qho^{-\sfrac 32} \eta_{i,k,\xi,\vecl} \, h_{q+1} \, .
    \end{align}
The coefficient function $a_\pxi$ is then multiplied by a helical pipe bundle $\nabla \Phiik^{-1} \BB_\pxi \circ \Phiik$, where we have used Proposition~\ref{prop:pipe:helical} and the shorthand notation
\begin{equation}
    \BB_\pxi = \rhob_\pxi \sum_I \zetab_\xi^I \mathbb{H}^I_\pxi
\end{equation}
to refer to the helical pipe bundle associated with the region $\Omega_0= \supp \zeta_{q,i,k,\xi,\vecl} \cap \{t=k\mu_q\Gamma_q^{-i-2}\}$.  We will use $\td\UU_\pxi^I$ to denote the potential satisfying $\curl \td\UU_\pxi^I = \HH_\pxi^I$. So, from~\eqref{eq:hWW:explicit}, we have
\begin{equation}
    \td\UU_\pxi^I = (\na\Psi_{q,\xi}^+)^T \UU_\pxi^I \circ \Psi_{q,\xi}^+ \,.
\end{equation}
Applying~\eqref{eq:hpipes:flowed:1}, we define the principal part of the helicity corrector by
\begin{align}
    w_{q+1,h}^{(p)} 
    & = \sum_{i,k,\vecl,I}   \underbrace{a_\pxi \left( \rhob_\pxi \zetab_\xi^I \right) \circ \Phiik \curl \left( \nabla \Phiik^T \td\UU_\pxi^I \circ \Phiik \right) }_{w_\pxi^{(p),I}}
    \,. \label{w:q+1:H:def}
\end{align}
The notation $w_\pxi^{(p),I}$ refers to fixed values of $i,k,\vecl,I$; note that $\xi$ is always equal to $\vec h_{(q+1) \text{ mod } \bn}$.  We add the divergence corrector $w_{q+1,h}^{(c)}$ by
 \begin{align}
   w_{q+1,h}^{(c)}   =\sum_{i,k,\vecl,I}   \underbrace{\nabla \left( a_\pxi \left( \rhob_\pxi \zetab_\xi^I \right) \circ \Phiik  \right) \times \left( \nabla \Phiik^T \td\UU_\pxi^I \circ \Phiik \right)}_{w_{(\xi)}^{(hc),I}} \, , \label{w:q+1:h:c:def}
 \end{align}
so that the mean-zero, divergence-free total helicity corrector is given by
 \begin{align}
     w_{q+1,h}
     &= w_{q+1, h}^{(p)}+ w_{q+1, h}^{(c)} = \sum_{i,k,\vecl,I}   \underbrace{\curl \left( a_\pxi \left( \rhob_\pxi \zetab_\xi^I \right) \circ \Phiik \nabla \Phiik^T \td\UU_\pxi^I \circ \Phiik \right) }_{w_{(\xi)}^{(h),I}}\, .  \label{wqplusoneoneh}
 \end{align}

\subsection{Definition of the complete corrector}
We set
\begin{align}\label{defn:w}
    w_{q+1} = w_{q+1,R} + w_{q+1,h} \, , \qquad  w_{q+1}^{(p)} := w_{q+1,R}^{(p)} + w_{q+1,h}^{(p)} \, , \qquad w_{q+1}^{(c)} := w_{q+1,R}^{(c)} + w_{q+1,h}^{(c)} \, .
\end{align}
With $\diamond=R,h$ to distinguish between Reynolds and helicity correctors, we shall use the notations $w_{q+1,\diamond}$, $w_{q+1,\diamond}^{(p)}$,  and $w_{q+1,\diamond}^{(c)}$.

\subsection{Dodging for the velocity increment}\label{sec:dodging}

In this section, we define a mollified velocity increment $\hat w_{q+\bn}$. We then introduce Lemma \ref{lem:dodging}, which is slighter stronger than Hypothesis \ref{hyp:dodging1}.

\begin{definition}[\bf Definition of $\hat w_\qbn$ and $u_{q+1}$]\label{def:wqbn}
Let $\mathcal{\tilde P}_{q+\bn,x,t}$\index{$\mathcal{\tilde P}_{q+\bn,x,t}$} denote a space-time mollifier which is a product of compactly supported kernels at spatial scale $\lambda_{q+\bn}^{-1}\Gamma_{q+\bn-1}^{-\sfrac 12}$ and temporal scale $\Tau_{q+1}^{-1}$, where both kernels have vanishing moments up to $10\Nfin$ and are $C^{10\Nfin}$ differentiable.  Define
\begin{equation}\label{def.w.mollified}
    \hat w_{q+\bn} := \mathcal{\tilde P}_{q+\bn,x,t} w_{q+1} \, , \qquad u_{q+1} = u_q + \hat w_{\qbn} \, .
\end{equation}
\end{definition}
\noindent 
Recalling from~\eqref{eq:space:time:balls} the notations $B(\Omega,\lambda^{-1})$ and $B(\Omega,\lambda^{-1},\tau)$, we have that
\begin{equation}
    \supp \hat w_\qbn \subseteq B\left( \supp w_{q+1}, \sfrac 12 \lambda_\qbn^{-1} , \sfrac 12 \Tau_q \right) \, . \label{eq:dodging:useful:support}
\end{equation}
Then using~\eqref{eq:WW:explicit},~\eqref{eq:W:xi:q+1:nn:def},~\eqref{eq:hel:bundle},~\eqref{helical.in.straight},~\eqref{eq:W:xi:q+1:nn:def}, and~\eqref{w:q+1:H:def}, we set 
\begin{align}\label{eq:pipez:thursday}
    \varrho_{(\xi),R}^{I} := \xi \cdot \WW_{(\xi)}^I \, , \qquad \varrho_{(\xi), h}^I(\cdot) := \left[ \left( \nabla\Psi_q^\pm \right)^{-1} \xi \right] \cdot  \HH_{\pxi}^I  \, .
\end{align}
Next, in slight conflict with \eqref{eq:space:time:balls}, we shall also use the notation
\begin{align}\label{eq:ballz:useful}
    B\left(\supp \varrho_{\pxi,\diamond}^I,\lambda^{-1}\right) := \left\{ x\in\T^3 \, : \, \exists y \in \supp \varrho_{\pxi,\diamond}^I \, , |x-y| \leq \lambda^{-1} \right\}
\end{align}
throughout this section, despite the fact that $\supp\varrho_{\pxi,\diamond}^I$ is not a set in space-time, but merely a set in space. We shall also use the same notation but with $\varrho_{\pxi,\diamond}^I$ replaced by $\rhob_\xi^\diamond$. Finally, for any smooth set $\Omega\subseteq \mathbb{T}^3$ and any flow map $\Phi$ defined in Definition~\ref{def:transport:maps}, we use the notation 
\begin{equation}\label{eq:flowing:sets}
\Omega \circ \Phi := \left\{(y,t): t\in \R, \Phi(y,t)\in \Omega\right\} = \supp \left(\mathbf{1}_{\Omega}\circ \Phi\right) \, .
\end{equation}
In other words, $\Omega\circ\Phi$ is a space-time set whose characteristic function is annihilated by $\Dtq$.

\begin{lemma}[\bf Dodging for $w_{q+1}$ and $\hat w_{\qbn}$]\label{lem:dodging} 
We construct $w_{q+1}$ so that the following hold.
\begin{enumerate}[(i)]    \item\label{item:dodging:more:oldies} Let {$q+1\leq q' \leq q+ \sfrac \bn 2$} and fix indices $\diamond,i,j,k,\xi,\vecl$, which we abbreviate by $(\pxi,\diamond)$, for a coefficient function $a_{\pxi,\diamond}$ (c.f.~\eqref{eq:a:xi:def}, \eqref{eq:a:xi:h:def}).  Then
    \begin{equation}
        B\left( \supp \hat w_{q'}, \frac 12 {\lambda_{q+1}^{-1}\Ga_q^2}, {2 \Tau_q} \right) \cap \supp \left( \tilde \chi_{i,k,q} \zeta_{q,\diamond,i,k,\xi,\vecl} \, \rhob_{\pxi}^{\diamond}\circ \Phi_{(i,k)} \right) = \emptyset \, . \label{eq:oooooldies}
    \end{equation}
    \item\label{item:dodging:1} Let $q'$ satisfy $q+1\leq q' \leq q+\bn-1$, fix indices $(\pxi,\diamond,I)$, and assume that $\Phiik$ is the identity at time $t_{\pxi}$, cf. Definition~\ref{def:transport:maps}. Then we have that
\begin{align}
    B \left(\supp \hat w_{q'}, \frac 14 \lambda_{q'}^{-1} \Gamma_{q'}^2, 2\Tau_{q} \right)  \cap  \supp &\left( \tilde \chi_{i,k,q} \zeta_{q,\diamond,i,k,\xi,\vecl} \left(\rhob_{(\xi)}^\diamond \zetab_{\xi}^{I,\diamond} \right)\circ \Phiik \right) \notag \\
    &\cap
    B\left( \supp \varrho^I_{(\xi),\diamond} , \frac 12 {\lambda_{q'}^{-1} \Gamma_{q'}^2}\right)\circ \Phiik
    = \emptyset \, . \label{eq:dodging:oldies:prep}
    \end{align}
As a consequence we have
\begin{equation}\label{eq:dodging:oldies}
        B\left( \supp \hat w_{q'}, \frac 14 {\lambda_{q'}^{-1} \Gamma_{q'}^2}, 2\Tau_q \right) \cap  \supp w_{q+1} = \emptyset \, .
    \end{equation}
    \item\label{item:dodging:2} Consider the set of indices $\{(\pxi,\diamond,I)\}$, whose elements we use to index the correctors constructed in~\eqref{wqplusoneone} and~\eqref{wqplusoneoneh}, and let $\ttl, \ov \ttl \in \{p,c, hp, hc\}$ denote either principal or divergence corrector parts for Euler-Reynolds and helicity correctors, respectively. Then if $(\ov\diamond,(\ov \xi), \ov I) \neq (\diamond,(\xi),I)$, we have that for any $\ttl, \ov \ttl$,
    \begin{equation}\label{eq:dodging:newbies}
        B\left(\supp w_{\pxi}^{(\ttl),I}, \frac14 \la_\qbn^{-1}\Ga_q^2\right) \cap B\left(\supp w_{(\ov \xi)}^{(\ov \ttl), \ov I}, \frac14 \la_\qbn^{-1}\Ga_q^2\right) = \emptyset \, .
    \end{equation}
    \item\label{item:dodging:zero}  $\hat w_{\qbn,R}$ satisfies Hypothesis~\ref{hyp:dodging2} with $q$ replaced by $q+1$.
\end{enumerate}
\end{lemma}

\begin{proof}
The proof is the same as \cite[Lemma~4.5]{GKN23strong}.  The role played by the Euler-Reynolds corrector when $\diamond = R$ is identical in both cases, and the role played by the current corrector when $\diamond=\varphi$ in~\cite{GKN23strong} is analogous to that of the case $\diamond=h$ here.  Of crucial importance is the fact that the helical deformations defined in Proposition~\ref{prop:pipe:helical} are small enough that the helically deformed pipe flow fits in a pipe flow with radius enlarged by only a constant, $q$-independent factor; see~\eqref{helical.in.straight}.  Therefore, the same dodging procedure which works for ``straight pipes,'' i.e. those which are initiated without helical deformations, works for the helical pipe bundles as well. Finally, we note that the proof of the dodging (see~\cite[Lemmas~4.1, 4.2]{GKN23strong}) requires velocity increments at different generations in $q$ to be defined using different vector directions.  This condition is satisfied due to Definition~\ref{def:vec} and the choices of $\xi$ in~\eqref{eq:a:xi:def} and~\eqref{eq:a:xi:h:def}. We omit further details.   
\end{proof}

\begin{remark}[\bf Verifying Hypothesis~\ref{hyp:dodging1}]\label{rem:checking:hyp:dodging:1}
We claim that \eqref{eq:dodging:oldies} and \eqref{eq:dodging:useful:support} imply that Hypothesis~\ref{hyp:dodging1}.  The proof is identical to that in~\cite[Remark~6.3]{GKN23}, and so we omit further details.
\end{remark}

\subsection{Estimates for \texorpdfstring{$w_{q+1}$}{wqn} and \texorpdfstring{$\hat w_{q+\bn}$}{hwqn}}\label{ss:stress:w:estimates}

We now record the needed estimates for the velocity increments.  As the proofs are nearly identical to those in~\cite{GKN23}, we omit most details.
\begin{lemma}[\textbf{Coefficient function estimates}]
\label{lem:a_master_est_p}
Let $A=(\nabla\Phi_{(i,k)})^{-1}$. For $N,N',N'',M$ with $N'',N'\in\{0,1\}$ and $N,M \leq {\sfrac{\Nfin} {3}}$, and $r \in [1, \infty]$ we have the following estimates.
\begin{subequations}\label{e:a_master_est_p}
\begin{align}
&\left\|D^{N-N''} D_{t,q}^M (\xi^\ell A_\ell^h \partial_h)^{N'} D^{N''} a_{\xi,i,j,k,\vec{l}}\right\|_{ r} \notag\\
&\lessg \left| \supp \eta_{i,j,k,\xi,\vecl} \right|^{\sfrac 1r} \delta_{q+\bn}^{\sfrac 12} \Gamma^{j+{7}}_{q}  \left({\Gamma_{q}^{-5}\lambda_{q+1}}\right)^N  \left({\Gamma_q^{  5}\Lambda_{q}}\right)^{N'} \MM{M, \NindSmall, \mu_{q}^{-1}\Gamma_{q}^{i+13}, \Tau_{q}^{-1}}\label{e:a_master_est_p_R} \, , \\
&\left\| D^{N-N''} D_{t,q}^M (\xi^\ell A_\ell^h \partial_h)^{N'} D^{N''} \left( a_{\xi,i,j,k,\vecl} \left( \rhob_{(\xi)} \zetab_\xi^{I} \right) \circ \Phiik \right) \right\|_{ r}\notag\\
&\lessg \left| \supp \left( \eta_{i,j,k,\xi,\vecl} \, \zetab_\xi^{I} \right) \right|^{\sfrac 1r} \delta_{q+\bn}^{\sfrac 12} \Gamma_q^{j {+10}} \left({\lambda_{q+\lfloor \sfrac \bn 2 \rfloor}}\right)^N \left({\Gamma_q^{  5}\Lambda_{q}}\right)^{N'} \MM{M, \NindSmall, \mu_{q}^{-1}\Gamma_{q}^{i+13}, \Tau_{q}^{-1}}\label{e:a_master_est_p_R:zeta}
\, ,\\
&\left\| D^{N-N''} D_{t,q}^M (\xi^\ell A_\ell^h \partial_h)^{N'} D^{N''} a_{\xi, i,k,\vecl}\right\|_{ r} \notag\\
&\lesssim \frac{\lambda_\qbn}{\lambda_\qho^{\sfrac 32}}{\Ga_{q-1}^{-\sfrac12}} \left| \supp \eta_{i,k,\xi,\vecl} \right|^{\sfrac 1r} \left({\Gamma_{q}^{-5}\lambda_{q+1}}\right)^N \left({\Gamma_q^{ {13}}\Lambda_{q}}\right)^{N'} \MM{M, \NindSmall, \mu_{q}^{-1}\Gamma_{q}^{i+ {13}}, \Tau_{q}^{-1} {\Ga_q^8}}\label{e:a_master_est_p_h} \, , \\
&\left\| D^{N-N''} D_{t,q}^M (\xi^\ell A_\ell^h \partial_h)^{N'} D^{N''} \left( a_{\xi,i,k,\vecl} \left( \rhob_{(\xi)} \zetab_\xi^{I} \right) \circ \Phiik \right) \right\|_{ r}\notag\\
&\lessg \frac{\lambda_\qbn}{\lambda_\qho^{\sfrac 32}} {\Ga_{q-1}^{-\sfrac12}}\left| \supp  \eta_{i,k,\xi,\vecl} \, \zetab_\xi^{I}  \right|^{\sfrac 1r} \left({\lambda_{q+\lfloor \sfrac \bn 2 \rfloor}}\right)^N \left({\Gamma_q^{ {13}}\Lambda_{q}}\right)^{N'} \MM{M, \NindSmall, \mu_{q}^{-1}\Gamma_{q}^{i+ {13}}, \Tau_{q}^{-1} {\Ga_q^8}}\label{e:a_master_est_p_h:zeta} \, .
\end{align}
\end{subequations}
In the case that $r=\infty$, the above estimates give that
\begin{subequations}\label{e:a_master_est_p_uniform}
\begin{align}
&\left\| D^{N-N''} D_{t,q}^M (\xi^\ell A_\ell^h \partial_h)^{N'} D^{N''} a_{\xi,i,j,k,\vec{l}}\right\|_{\infty} \notag  \\
&\qquad\lessg {\Gamma_q^{\frac{\badshaq}{2}+ {10}}} \left(\Gamma_{q}^{-5}\lambda_{q+1}\right)^N \left({\Gamma_q^{ {13}}\Lambda_{q}} \right)^{N'} \MM{M,\Nindt, \mu_{q}^{-1}\Gamma_{q}^{i+ {13}},\Tau_q^{-1} {\Ga_q^8}} \label{e:a_master_est_p_uniform_R} \, , \\
&\left\| D^{N-N''} D_{t,q}^M (\xi^\ell A_\ell^h \partial_h)^{N'} D^{N''} a_{\xi, i,k,\vec{l}}\right\|_{ \infty}\notag \\
&\qquad\lessg \lambda_\qbn\lambda_\qho^{-\sfrac 32} {\Ga_{q-1}^{-\sfrac12}}\left(\Gamma_{q}^{-5}\lambda_{q+1}\right)^N \left({\Gamma_q^8\Lambda_{q}} \right)^{N'} \MM{M,\Nindt, \mu_{q}^{-1}\Gamma_{q}^{i+13},\Tau_q^{-1}} \label{e:a_master_est_p_uniform_h} \, ,
\end{align}
\end{subequations}
with analogous estimates (incorporating a loss of $\Ga_q^3$ for $\diamond=R$) holding for the product $a_{(\xi),\diamond}\zetab_{\xi}^{I}\rhob_{(\xi)}$. The constants in the above estimates~\eqref{e:a_master_est_p}-\eqref{e:a_master_est_p_uniform} depend on the maximum values of the first $\Nfin$ derivatives of $e(t)$ and $h(t)$ on the time interval $[-1,T+1]$; crucially, they are independent of $q$.
\end{lemma}
\begin{proof}
The proofs of these statements are nearly identical to those from~\cite[Lemma~6.4]{GKN23}.  The only differences are as follows.  First, there are no factors of $r_q^{-\sfrac 13}$, which are not needed since we aim for $L^2$ regularity rather than $L^3$ regularity.  Second, the amplitude for the helical coefficient functions has no $j$ dependence.  Finally, we estimate the terms when the material derivative falls on $H(u_q)$ in the coefficient of the helical perturbation by using the definition~\eqref{defn:H}, the inductive assumptions~\eqref{hel:lip} and~\eqref{eq:resolved:helicity1a}, and the computation
\begin{align*}
    2 &\left( (h-H)^{\sfrac 12} \right) \left[ \sfrac{d}{dt} \right]^{N+1} \left( (h-H)^{\sfrac 12} \right)\\ &= \left[ \sfrac{d}{dt} \right]^{N+1} \left( h-H \right) + \sum_{0<N'<N+1} c_{N,N'} \left[ \sfrac{d}{dt} \right]^{N'} \left( (h-H)^{\sfrac 12} \right) \left[ \sfrac{d}{dt} \right]^{N+1-N'} \left( (h-H)^{\sfrac 12} \right) \, ,
\end{align*}
which implies by induction that for $N+1 \leq \sfrac{\Nfin}{3}$ (here we crucially use the lower bound~\eqref{eq:resolved:helicity1a})
\begin{align*}
    \left| \left[ \sfrac{d}{dt} \right]^{N+1} \left( (h-H)^{\sfrac 12} \right) \right| \les \Gamma_q^{\sfrac 12} \MM{N+1, \Nindt, \mu_{q-1}^{-1}, \Tau_{q-1}^{-1}} \, . 
\end{align*}
\end{proof}

\begin{corollary}[\textbf{Full velocity increment estimates}]
\label{cor:corrections:Lp}
For $N,M\leq {\sfrac \Nfin 4}$, we have\footnote{We have added the subscript $\diamond=R, h$ to the notation from~\eqref{wqplusoneonep} and~\eqref{w:q+1:H:def} since the estimates differ.}
\begin{subequations}\label{eq:w:oxi:ps}
\begin{align}
\norm{D^ND_{t,q}^M w_{(\xi),R}^{(p),I}}_{L^r} 
&\lessg \left| \supp \left(\eta_{\pxi}\zetab_\xi^{I} \right) \right|^{\sfrac 1r} \delta_{q+\bn}^{\sfrac 12}\Gamma_q^{j+ {10}} r_q^{\frac2r-1} \lambda_{q+\bn}^N  \MM{M, \NindSmall, \mu_{q}^{-1}\Gamma_{q}^{i+ {13}}, \Tau_{q}^{-1} {\Ga_q^8}}\label{eq:w:oxi:est:master}
\\
\norm{D^ND_{t,q}^M w_{(\xi),R}^{(p),I}}_{L^\infty}
&\lessg \Gamma_q^{\frac \badshaq 2 +  {13}}
r_{q}^{-1} \lambda_{q+\bn}^N \MM{M, \NindSmall, \mu_{q}^{-1}\Gamma_{q}^{i+ {13}}, \Tau_{q}^{-1} {\Ga_q^8}} \, ,
\label{eq:w:oxi:unif:master} \\
\norm{D^ND_{t,q}^M w_{(\xi),h}^{(p),I}}_{L^r} 
&\lessg \left| \supp \left(\eta_{\pxi}\zetab_\xi^{I} \right) \right|^{\sfrac 1r} \frac{\lambda_\qbn}{\lambda_\qho^{\sfrac 32}} {\Ga_{q-1}^{-\sfrac12}}r_q^{\frac2r-1} \lambda_{q+\bn}^N  \MM{M, \NindSmall, \mu_{q}^{-1}\Gamma_{q}^{i+ {13}}, \Tau_{q}^{-1} {\Ga_q^8}}\label{eq:w:oxi:h:est:master}
\\
\norm{D^ND_{t,q}^M w_{(\xi),h}^{(p),I}}_{L^\infty}
&\lessg \frac{\lambda_\qbn}{\lambda_\qho^{\sfrac 32}}{\Ga_{q-1}^{-\sfrac12}}
r_{q}^{-1} \lambda_{q+\bn}^N \MM{M, \NindSmall, \mu_{q}^{-1}\Gamma_{q}^{i+ {13}}, \Tau_{q}^{-1} {\Ga_q^8}} \, .
\label{eq:w:oxi:h:unif:master}
\end{align}
\end{subequations}
Also, for $N,M\leq {\sfrac \Nfin 4}$, we have that
\begin{subequations}\label{eq:w:oxi:cs}
\begin{align}
\norm{D^ND_{t,q}^M w_{(\xi),R}^{(c),I}}_{L^r}
&\lessg r_q \left| \supp \left(\eta_{\pxi}\zetab_\xi^{I} \right) \right|^{\sfrac 1r} \delta_{q+\bn}^{\sfrac 12} \Gamma_q^{j+ {10}} r_q^{\frac2r-1} \lambda_{q+\bn}^N  \MM{M, \NindSmall, \mu_{q}^{-1}\Gamma_{q}^{i+ {13}}, \Tau_{q}^{-1} {\Ga_q^8}} \label{eq:w:oxi:est:master:c}
\\
\norm{D^ND_{t,q}^M w_{(\xi),R}^{(c),I}}_{L^\infty}
&\lessg \Gamma_q^{\frac \badshaq 2 +  {13}} \lambda_{q+\bn}^N \MM{M, \NindSmall, \mu_{q}^{-1}\Gamma_{q}^{i+ {13}}, \Tau_{q}^{-1} {\Ga_q^8}} \, ,
\label{eq:w:oxi:unif:master:c}\\
\norm{D^ND_{t,q}^M w_{(\xi),h}^{(c),I}}_{L^r}
&\lessg r_q \left| \supp \left(\eta_{\pxi}\zetab_\xi^{I} \right) \right|^{\sfrac 1r} \frac{\lambda_\qbn}{\lambda_\qho^{\sfrac 32}} {\Ga_{q-1}^{-\sfrac12}}r_q^{\frac2r-1} \lambda_{q+\bn}^N  \MM{M, \NindSmall, \mu_{q}^{-1}\Gamma_{q}^{i+ {13}}, \Tau_{q}^{-1} {\Ga_q^8}} \label{eq:w:oxi:est:master:hc}
\\
\norm{D^ND_{t,q}^M w_{(\xi),h}^{(c),I}}_{L^\infty}
&\lessg \frac{\lambda_\qbn}{\lambda_\qho^{\sfrac 32}} {\Ga_{q-1}^{-\sfrac12}}\lambda_{q+\bn}^N \MM{M, \NindSmall, \mu_{q}^{-1}\Gamma_{q}^{i+ {13}}, \Tau_{q}^{-1} {\Ga_q^8}} \, .
\label{eq:w:oxi:unif:master:hc}
\end{align}
\end{subequations}
\end{corollary}
\begin{proof}[Proof of Corollary~\ref{cor:corrections:Lp}]
The proof is essentially identical to~\cite[Corollary~6.6]{GKN23}; the main ingredients are the estimates from Lemma~\ref{lem:a_master_est_p} and the decoupling result from~\cite[Lemma~A.1]{GKN23}.
\end{proof}

Finally, we estimate the mollified velocity increment given in~Definition~\ref{def:wqbn}. 
\begin{lemma}[\textbf{Estimates on $\hat w_{q+\bn}$}]\label{lem:mollifying:w}
We have that $\hat w_{q+\bn}$ satisfies the following properties.
\begin{enumerate}[(i)]
\item\label{item:moll:vel:1} For all $N+M\leq 2\Nfin$, we have that
\begin{subequations}\label{eq:vellie:upgraded:statement}
\begin{align}
&\norm{ D^N D_{t,q+\bn-1}^M  \hat w_{q+\bn} }_{L^2(\supp \psi_{i,q+\bn-1})}  \notag\\
& \les \left( \Gamma_q^{20} \delta_{q+\bn}^{\sfrac 12} + \frac{\lambda_\qbn}{\lambda_\qho^{\sfrac 32}}{\Ga_{q-1}^{-\sfrac12}} \right) \left(\lambda_{q+\bn}\Gamma_{q+\bn-1}\right)^N \MM{M, \NindRt, \Gamma_{q+\bn-1}^{i-1} \mu_{q+\bn-1}^{-1}, \Tau_{q+\bn-1}^{-1} \Ga_{q+\bn-1} }
\label{eq:vellie:inductive:dtq-1:upgraded:statement} \\
&\norm{ D^N D_{t,q+\bn-1}^M  \hat w_{q+\bn} }_{L^\infty(\supp \psi_{i,q+\bn-1})}  \notag\\
&\les \left( \Gamma_q^{\sfrac{\badshaq}{2}+16} + \frac{\lambda_\qbn}{\lambda_\qho^{\sfrac 32}}{\Ga_{q-1}^{-\sfrac12}} \right) r_q^{-1}\left(\lambda_{q+\bn}\Gamma_{q+\bn-1}\right)^N \MM{M, \NindRt, \Gamma_{q+\bn-1}^{i-1} \mu_{q+\bn-1}^{-1}, \Tau_{q+\bn-1}^{-1} \Ga_{q+\bn-1} } \, . \label{eq:vellie:inductive:dtq-1:uniform:upgraded:statement}
\end{align}
\end{subequations}
\item\label{item:moll:vel:2} For all $N+M\leq \sfrac{\Nfin}{4}$, we have that
\begin{align}
\norm{D^N D_{t,q+\bn-1}^M \left(w_{q+1}- \hat w_{q+\bn}\right)}_\infty &\lec \delta_{q+3\bn}^3 \Tau_\qbn^{25\Nindt} \left(\lambda_{q+\bn}\Gamma_{q+\bn-1}\right)^N  \notag\\
&\qquad \qquad \times \MM{M, \NindRt, \mu_{q+\bn-1}^{-1}, \Tau_{q+\bn-1}^{-1} \Ga_{q+\bn-1} }  \, . \label{eq:diff:moll:vellie:statement}
\end{align}
\end{enumerate}
\end{lemma}
\begin{proof}[Proof of Lemma~\ref{lem:mollifying:w}]
The proof of this Lemma is nearly identical to that of~\cite[Lemma~6.7]{GKN23}.  The main ingredients are the aggregation lemmas, the previous lemmas from this section, a mollification lemma, and parameter inequalities for $\Tau_q$ from~\eqref{v:global:par:ineq} and $N_g$ and $N_c$ (large integers quantifying how many derivatives are needed to run the abstract mollification lemma) from~\eqref{eq:darnit}.
\end{proof}

\section{The helicity and energy increments}
\label{sec:helicity.energy.inc}

\subsection{Helicity increment}

In this subsection, we consider the helicity generated by $w_{q+1}$. 
\begin{lemma}[\textbf{$w^{(p)}_{q+1,h}$ reduces the helicity error}]\label{lem:helicity.h}
$w_{q+1,h}^{(p)}$ as defined in~\eqref{w:q+1:H:def} satisfies 
\begin{align}\label{hel:inc:est}
 \frac{\Ga_{q-1}}{4\Ga_q^2}\leq h(t) - H(u_q)- \int_{\T^3}  w_{q+1,h}^{(p)}(t,x) \cdot \curl w_{q+1, h}^{(p)}(t,x) \, dx
    \leq \frac 34 \Ga_q^{-1} \, .
\end{align}
In addition, for any multiplier $\varphi:\R^3 \rightarrow \overline{B_1(0)} \subset \mathbb{C}$, $w_{q+1,h}^{(p)}$ satisfies
\begin{equation}
\left| \int_{\T^3} \bp_\varphi(w_{q+1,h}^{(p)}) \cdot \curl w_{q+1,h}^{(p)} \right| < \frac{1}{10} \Gamma_{q-4}^{-\frac12} \, . \label{hab}
\end{equation}
Finally, for $M \leq \sfrac{\Nfin}{3}$,
\begin{equation}
\left| \ddt^M H(w_{q+1,h}^{(p)}) \right| \les \MM{M, \Nindt, \mu_q^{-1}\Gamma_q^{20}, \Tau_q^{-1}\Ga_q^8} \, . \label{derivs:1}
\end{equation}
\end{lemma}
\begin{proof}
We break the proof into steps. In \texttt{Step 1}, we prove \eqref{hab} and~\eqref{derivs:1} directly from the amplitude estimates. In \texttt{Step 2}, we decompose the helicity $H(w_{q+1,h}^{(p)})$ into the leading order term and the rest. 
In \texttt{Step 3} and \texttt{Step 4}, 
we compute the leading-order and remaining contributions to the helicity, respectively. In \texttt{Step 5}, we conclude the proof, obtaining \eqref{hel:inc:est}. 


For simplicity we perform all computations in the case $\xi = \vec h_0 = \vec e_2$, which corresponds to the case $q \, \textnormal{mod} \, \bn = 0$; see~\eqref{eq:a:xi:h:def}. Setting $\tilde a_{(\vec e_2)}:= a_{(\vec e_2)}\left(\chib_{(\vec e_2)} \etab_{\vec e_2}^{I} \right)\circ \Phiik$ and $\td\Phi_{(i,k)}  = \Psi_q^+ \circ \Phiik$, recall from \eqref{w:q+1:H:def}, \eqref{defn:mathcal.H}, and~\eqref{eq:hpipes:flowed:1} that the principal part of the helicity perturbation is given by
\begin{align*}
    w_{q+1,h}^{(p)} 
     = \sum_{i,k,\vecl,I}   \td a_{(\vec e_2)}   \, \left(\nabla \td\Phi_{(i,k)}\right)^{-1} \vec e_2 \varrho_{\vec e_2, \lambda_\qbn, r_q}^k\circ \td\Phi_{(i,k)}\, .
\end{align*}

\noindent\texttt{Step 1: Proof of \eqref{hab} and~\eqref{derivs:1}.}
Since the multiplier operator $\bp_{\ph}$ is bounded on $L^2$ with unit operator norm, we use Corollary~\ref{cor:corrections:Lp},~\eqref{eq:desert:cowboy:sum}, ~\eqref{def:thursday}, and \eqref{eq:qho:choice3} to estimate
\begin{align*}
    \left| \int_{\T^3} \bp_\varphi(w_{q+1,h}^{(p)}) \cdot \curl w_{q+1,h}^{(p)} \right|^2 
    &\lec \| w_{q+1,h}^{(p)} \|_{L^2}^2 \| \curl w_{q+1,h}^{(p)} \|_{L^2}^2 \\
    &\lec \left(\sum_{i,k,\xi,\vecl, I} \left| \supp \left( \eta_{i,k,\xi,\vecl} \zetab_\xi^I \right) \right| \right)^2 \left( \frac{\lambda_\qbn}{\lambda_\qho^{\sfrac 32}} \Gamma_{q-1}^{-\sfrac 12} \frac{\lambda_\qbn^2}{\lambda_\qho^{\sfrac 32}} \Gamma_{q-1}^{-\sfrac 12} \right)^2 \\
    &\lec \left(\frac{\la_{q+\bn}}{\la_\qho}\right)^6
    \Ga_{q-1}^{-2}
    \lec \left(\Ga_{q-1}^{-1} \Ga_{q-100}^{\sfrac 3{100}}\right)^2 < \left(\frac 1{10}\Ga_{q-4}^{-\sfrac 12} \right)^2 \,.
\end{align*}
This proves~\eqref{hab}.  In order to prove~\eqref{derivs:1}, we consider $M \geq 1$, since $M=0$ follows from~\eqref{hab}.  When $M\geq 1$, we convert time derivatives outside the integral to material derivatives inside the integral, since the material derivative is taken along a divergence-free vector field, obtaining
\begin{align*}
\left|  \ddt^M \int_{\T^3} w_{q+1,h}^{(p)} \cdot \curl w_{q+1,h}^{(p)} \right| &= \left| \int_{\T^3} \sum_{0 \leq M' \leq M} c_{M,M'} \Dtq^{M'} w_{q+1,h}^{(p)} \cdot \Dtq^{M-M'} \curl w_{q+1,h}^{(p)} \right| \, .
\end{align*}
Then using H\"older's inequality and Corollaries~\ref{cor:corrections:Lp} and~\ref{rem:summing:partition} gives the desired estimate.

\noindent\texttt{Step 2: Helicity decomposition.}  By using the dodging properties between $\{w_{(\vec e_2)}^{(p),I}\}$ from Lemma~\ref{lem:dodging} and \eqref{eq:dodging:newbies}, and denoting $\td A=(\nabla\td {\Phi}_{(i,k)})^{-1}$, we decompose 
\begin{align*}
H\left( w_{q+1,h}^{(p)}  \right)
    &= \int_{\T^3}  w_{q+1,h}^{(p)}  \cdot \curl w_{q+1,h}^{(p)} dx =\sum_{\substack{i,k,\vecl, I}} \int_{\T^3}  w_{(\vec e_2),h}^{(p),I}  \cdot \curl w_{(\vec e_2),h}^{(p),I} dx \\
&= \sum_{\substack{i,k,\vecl, I}} 
\int_{\T^3}  \td a_{(\vec e_2)} \td A \vec e_2 \varrho_{\vec e_2, \la_{q+\bn, r_q}}^k (\td\Phi_{(i,k)})  \cdot 
    \curl\left(\td a_{(\vec e_2)} \td A \vec e_2 \varrho_{\vec e_2, \la_{q+\bn, r_q}}^k(\td\Phi_{(i,k)})\right) dx\, . 
    \end{align*}
Suppressing all indices in the summand for convenience, we decompose the helicity further based on whether curl operator falls on $\td A e_2$ or not:     
\begin{align}
\sum_{\substack{i,k,\vecl, I}} 
    \int_{\T^3} \td a^2 {\varrho} (\td \Phi) \varrho(\td \Phi) \td A \vec e_2 \cdot \curl (\td A \vec e_2) dx + 
 \int_{\T^3} \td a \varrho(\td \Phi)  \td A \vec e_2 \cdot
 \na\left(\td a\varrho(\td \Phi)\right)
 \times
 (\td A \vec e_2) dx =: H_1 + H_2\, .
\label{h1h2}
\end{align}
Then the second term $H_2$ vanishes due to the fact that $\na\left(\td a\varrho(\td \Phi)\right)\times (\td A \vec e_2)$ is orthogonal to $\td A \vec e_2$.  Lastly, from the proof of Proposition~\ref{prop:pipe:helical}, we recall that 
$$\td e_2 \cdot \curl \td e_2 = \frac{\la_\qho^3}{\la_{q+\bn}^2}\,\text{ and } \td A \vec e_2 =(\na \Phi)^{-1} (\td e_2\circ \Phi)\,,$$
and we decompose $H_1$ into 
\begin{align*}
    H_1 
    &= \sum_{\substack{i,k,\vecl, I}} \int_{\T^3} \td a^2  \varrho^2(\td \Phi) \left(\td A \vec e_2 \cdot \curl (\td A \vec e_2) 
    -(\td e_2 \cdot \curl \td e_2)(\Phi)
    \right)
    dx \notag \\
    & \qquad + \frac{\la_\qho^3}{\la_{q+\bn}^2}
    \int_{\T^3} \td a^2((\mathbb{P}_{=0} + \mathbb{P}_{\neq 0}) (\varrho^2\circ\Psi))(\Phi)  
    dx\\
    &=: H_{11} + H_{12} + H_{13}\, . 
\end{align*}
As we will see, the $H_{12}$ will be the leading order term and the rest will be small.

\bigskip
\noindent\texttt{Step 3: Leading-order helicity contribution $H_{12}$.}  Using~\eqref{eq:a:xi:h:def} and the fact that $\zetab^2$ and $\eta^2$ are partitions of unity and $\bp_{= 0}(\rhob^2)=1$, we write
\begin{align}
 &\sum_{\substack{i,k ,\vecl, I}} \frac{\la_\qho^3}{\la_{q+\bn}^2} \int_{\T^3} \td a^2(\mathbb{P}_{=0}(  \varrho^2\circ\Psi))(\Phi)  
    dx \notag\\
 &= \sum_{\substack{i,k, \vecl, I}}  h_{q+1}^2 \int_{\T^3} \eta^2 (\rhob  \zetab)^2 \circ \Phi \,  dx  
 = \sum_{\substack{i,k, \vecl}}  h_{q+1}^2 
\int_{\T^3}  \eta^2 \left( \mathbb{P}_{=0}+\mathbb{P}_{\neq 0} \right)\left(\rhob^2\right)\circ (\Phi)\,  dx\notag \\
& =  (h(t)-H(u_q)) \left(
1-\frac{\Ga_{q-1}}{2\Ga_q}\right)   + \sum_{i,k,\vecl} h_{q+1}^2 
 \int_{\T^3} \eta^2(\Phi^{-1}) \mathbb{P}_{\neq 0}(\rhob^2) \, dx \, .\label{e:h:s:5}
\end{align}
We first show that the second term is small. Note that $\rhob$ is $(\sfrac{\T}{\la_{q+1}\Ga_q^{-4}})^3$-periodic by item \ref{i:bundling:1} in Proposition \ref{prop:bundling} and the spatial derivative cost of $\eta^2(\Phi^{-1})$ is $\la_{q+1}\Ga_q^{-5}$ from~\eqref{def:cumulative:hel},~\eqref{eq:sharp:Dt:psi:i:q:old}, and~\eqref{eq:checkerboard:derivatives}.  Therefore this term can be made small after integrating by parts many times.  Specifically, using the $L^1$ norm of the high-frequency function, the $L^\infty$ norm of the low-frequency object, Lemma~\ref{lem:checkerboard:estimates} to bound the number of different values of $\vecl$, and~\eqref{ineq:dpot:hel}, we bound the second term  by
\begin{align}
 \left|
\sum_{i,k,\vecl}  h_{q+1}^2 
 \int_{\T^3} \eta^2(\Phi^{-1}) \mathbb{P}_{\neq 0}(\rhob^2) 
\right|
&= 
\left|
\sum_{i,k,\vecl} h_{q+1}^2 
 \int_{\T^3}(-\Delta)^{\sfrac \dpot 2} \eta^2(\Phi^{-1}) (-\Delta)^{-\sfrac \dpot 2}\mathbb{P}_{\geq  \la_{q+1}\Ga_q^{-4}}(\rhob^2) 
\right| \notag \\
&\lec \la_\qbn^{10} \Gamma_q^{-\dpot} \Ga_{q+1}^{-1} \leq \lambda_\qbn^{-1} \, . \label{e:h:s:3}
\end{align}

\bigskip

\noindent\texttt{Step 4: Analyzing the remaining contributions ${H_{11}}$ and $H_{13}$.}
We first analyze $H_{11}$ using
\begin{align*}
(\curl (\td A \vec e_2))_l
&= \eps_{lmn} \pa_m(\td A(\vec e_2))_n\\
&=\eps_{lmn} \pa_m((\na \Phi)^{-1}_{no}(\td e_2\circ\Phi)_o)\\
&=\eps_{lmn} (\pa_m (\na \Phi)^{-1}_{no})(\td e_2\circ \Phi)_o
+\eps_{lmn}(\na \Phi)^{-1}_{no} (\na\Phi)_{vm} (\pa_v (\td e_2)_o)(\Phi)\, .
\end{align*}
Combining this with~\eqref{eq:Lagrangian:Jacobian:1}--\eqref{eq:Lagrangian:Jacobian:6}, \eqref{eq:hel.building2}--\eqref{eq:hel.building3}, and~\eqref{eq:qho:choice1}, 
we obtain
\begin{align*}
    &\norm{\td A \vec e_2 \cdot \curl (\td A \vec e_2) - (\td e_2 \cdot \curl \td e_2)(\Phi)}_{L^\infty(\supp(\td a))}\\
    &\lec 
    \norm{(\na \Phi)^{-1}_{lw} (\td e_2\circ \Phi)_w
    \eps_{lmn} (\pa_m (\na \Phi)^{-1}_{no})(\td e_2\circ \Phi)_o}_{L^\infty(\supp(\td a))}\\
    &\quad+\norm{[(\na \Phi)^{-1}_{lw}
    (\na \Phi)^{-1}_{no}(\na\Phi)_{vm}- \de_{lw}\de_{no}\de_{vm}]
    (\td e_2\circ \Phi)_w\eps_{lmn}  (\pa_v (\td e_2)_o)(\Phi)}_{L^\infty(\supp(\td a))}\\
    &\lec 
    (\tau_q^{-1} \mu_q \Ga_q^5)(\la_q \Ga_q)
 +(\tau^{-1}_q \mu_q \Ga_q^5)\cdot \frac{\la_\qho^2}{\la_{q+\bn}} \lec (\tau_q^{-1} \mu_q \Ga_q^5)\frac{\la_\qho^2}{\la_{q+\bn}}\, . 
\end{align*}
From this,~\eqref{eq:a:xi:h:def},~\eqref{e:hpipe:estimates:1},~\eqref{eq:resolved:helicity1a}, 
and~\cite[Lemma~A.1]{GKN23}, we can estimate $H_{11}$ by
\begin{equation}\label{e:h:s:4}
    |H_{11}| \lec \frac{\la_{q+\bn}^2}{\la_\qho^3}\cdot (\tau_q^{-1} \mu_q \Ga_q^5)\frac{\la_\qho^2}{\la_{q+\bn}}\, 
    \lec \frac{\la_{q+\bn}}{\la_\qho}(\tau_q^{-1} \mu_q \Ga_q^5)
    \underset{\eqref{eq:qho:choice2}}{\leq} \la_q\de_{q+\bn}(\tau_q^{-1} \mu_q \Ga_q^5)\underset{\eqref{eq:defn:tau}}{\leq} \Ga_{q+1}^{-1} \, .
\end{equation}
 
Next, we may estimate $H_{13}$ in a similar way as the high frequency part of $H_{12}$ by integrating by parts many times.   Specifically, we use that \blue{$(\varrho^2\circ\Psi)$ is $(\T / \lambda_{q+\half} \Gamma_q)^3$-periodic}, $\td a(\Phi^{-1})$ has derivative cost $\lambda_{q+\half}$ from item~\ref{item:check:3} in Lemma~\ref{lem:finer:checkerboard:estimates}, and the parameter condition~\eqref{ineq:dpot:hel} to obtain that
\begin{align}\label{e:h:s:7}
    |H_{13}|
    &= \left|\sum_{\substack{i,k,\vecl, I}}  \frac{\la_\qho^3}{\la_{q+\bn}^2}
    \int_{\T^3} \td a^2(\mathbb{P}_{\neq 0}) (\varrho^2\circ\Psi))( \Phi)  
    dx\right|
    = \left|\sum_{\substack{i,k,\vecl, I}}  \frac{\la_\qho^3}{\la_{q+\bn}^2}
    \int_{\T^3} \td a^2( \Phi^{-1})
    \mathbb{P}_{\neq 0}(\varrho^2\circ\Psi)  
    dx\right| \leq \la_{q+\bn}^{-1}
\end{align}
\bigskip

\noindent\texttt{Step 5: Conclusion.} Combining~\eqref{e:h:s:4},~\eqref{e:h:s:5},~\eqref{e:h:s:3}, and~\eqref{e:h:s:7}, we obtain
\begin{align*}
    h(t) - H(u_q)- H(w_{q+1,h}^{(p)})
   =h(t) - H(u_q) - H_1
   =(h(t) - H(u_q))\left(\frac{\Ga_{q-1}}{2\Ga_q}\right)
   + O(\Ga_{q+1}^{-1})\,, 
\end{align*}
and using the inductive assumption \eqref{eq:resolved:helicity1a}, we obtain $    \frac{\Ga_{q-1}}{4\Ga_q^2}\leq h(t) - H(u_q)- H(w_{q+1,h}^{(p)})
    \leq \frac 34 \Ga_q^{-1}$.
\end{proof}

\begin{lemma}[\textbf{Computing negligible Fourier modes}]\label{ph1}
There exists a choice of $\Dpot$ from Proposition~\ref{prop:pipeconstruction} and~\ref{prop:pipe:helical} so that the following holds.
\begin{enumerate}[(i)]
\item For any $m \in \mathbb{Z}^3$ satisfying $|m| \leq \lambda_\qbn \Gamma_{q-100}^{-\frac{1}{1000}}$, $
        \left| \mcf(w_{q+1})(m) \right| \leq {\lambda_\qbn^{-23}}$.
\item For any $m \in \mathbb{Z}^3$ satisfying $|m| \geq \lambda_\qbn \Gamma_q$,  $\left| \mcf(w_{q+1})(m) \right| \leq {|m|^{-24}}$.
\end{enumerate}
\end{lemma}
\begin{proof}
We begin with the first claim. We will only estimate $|\mcf(w_{q+1,R}^{(p)})(m)|$. The estimates of $|\mcf( w_{q+1,R}^{(c)})(m)|$ and $|\mcf(w_{q+1,h}^{(p)})(m)|$ then follow in a similar way. Recalling that
\begin{align*}
w_{q+1,R}^{(p)}   = \sum_{i,k,\xi,\vecl,I} \underbrace{a_{(\xi)} (t,y)\left(\chib_{(\xi)} \etab_{\xi}^{I} \right) \circ \Phiik  \nabla\Phi_{(i,k)}^{-1} }_{=: A}
    \mathbb{W}_{(\xi)}^I \circ \Phi_{(i,k)} \, , 
\end{align*}
we write
\begin{align*}
|\mcf(w_{q+1,R}^{(p)}(t))(m)|
&=\left|\sum_{i,k,\xi,\vecl,I}\int_{\R^3} A \mathbb{W}_{(\xi)}^I \circ \Phi_{(i,k)} e^{-2\pi i m\cdot x} \right|\\
&= \left|\sum_{i,k,\xi,\vecl,I}\int_{\R^3}  \xi \lambda_{q+\bn}^{-\Dpot }\div^\Dpot  \left(\vartheta^k_{\xi,\lambda_{q+\bn},r_{q}}\right) A\circ\Phi_{(i,k)}^{-1}e^{-2\pi im\cdot\Phi_{(i,k)}^{-1}} dx\right|\\
&= \lambda_{q+\bn}^{-\Dpot }\left|\sum_{i,k,\xi,\vecl,I}\int_{\R^3}    \left(\vartheta^k_{\xi,\lambda_{q+\bn},r_{q}}\right): \na^\Dpot\left( A\circ\Phi_{(i,k)}^{-1}e^{-2\pi im\cdot\Phi_{(i,k)}^{-1}}\right) dx
\right|\\
&\lec \la_{\qbn}^{-\Dpot} \left(\lambda_\qbn \Gamma_{q-100}^{-\frac{1}{1000}}\right)^{\Dpot} \de_{\qbn}^{1/2} < \frac{1}{\lambda_\qbn^{23}}\, . 
\end{align*}
The proof of the second estimate is even easier: integrating by parts $\Nind$ times, we obtain
\begin{align*}
    \left| \mcf(w_{q+1}(t))(m) \right| &= \left| \int_{\T^3} e^{-2\pi im\cdot x}w_{q+1}(t,x)\,dx \right|
    = \left| \int_{\T^3} \Delta^{-\lceil \Nind/2 \rceil}{e^{-2\pi im\cdot x}}\Delta^{\lceil \Nind/2 \rceil} w_{q+1}(t,x)\,dx \right|\\
    &\leq \frac1{|2\pi m|^{24}(\la_\qbn \Ga_q)^{\Nind-24}} \|D^{\Nind} w_{q+1}(t,x)\|_{L^1} \lesssim \frac1{|m|^{24}(\la_\qbn \Ga_q)^{\Nind-24}} \la_{\qbn}^{\Nind} \de_{\qbn}^{\sfrac{1}{2}} \\
    &\leq |m|^{-24}\Ga_q^{-\Nind+24}\la_{\qbn}^{24}\de_{\qbn}^{\sfrac{1}{2}} \underset{\eqref{eq:Nind:3}}{<} \frac{1}{|m|^{24}}\,.
\end{align*}
\end{proof}

\begin{proposition}[\textbf{Frequency localization}]\label{prop:hyp:freq:loc}
    Hypothesis~\ref{hyp:freq:loc} is true at step $q+1$.
\end{proposition}
\begin{proof}
    Let $\varphi_i$ for $i=1,2$ be as in Hypothesis~\ref{hyp:freq:loc} at step $q+1$. By the definition of $\hat w_\qbn$ in~\eqref{def.w.mollified},
    \begin{align*}
        \|\bp_{\varphi_i} \hat w_\qbn \|_\infty &= \|\bp_{\varphi_i} \mathcal{\tilde P}_{q+\bn,x,t} w_{q+1}\|_\infty 
        \leq \|\mathcal{\tilde P}_{q+\bn,x,t}\|_{L^1\to L^1} \| \bp_{\varphi_i}  w_{q+1}\|_\infty = \| \bp_{\varphi_i}  w_{q+1}\|_\infty\,.
    \end{align*}
    Thus, it suffices to estimate $\| \bp_{\varphi_i}  w_{q+1}\|_\infty$. For $i=1$, we have
    \begin{align*}
        \| \bp_{\varphi_1}  w_{q+1}\|_\infty &= \left\| \sum_{m\in \Z^3} \varphi_1(m) \mathcal{F}(w_{q+1})(m) e^{2\pi i m\cdot x} \right\|_\infty = \left\| \sum_{|m|\leq k_1} \varphi_1(m) \mathcal{F}(w_{q+1})(m) e^{2\pi i m\cdot x} \right\|_\infty\\
        &\leq \sum_{|m|\leq k_1} |\varphi_1(m)| |\mathcal{F}(w_{q+1})(m)| \leq \norm{\varphi_1}_{\infty}\sum_{|m|\leq k_1} \frac1{\la_\qbn^{23}} \leq \frac{k_1}{\la_\qbn^{20}}\,.
    \end{align*}
    Similarly, for $i=2$, we have
    \begin{align*}
        \| \bp_{\varphi_2}  w_{q+1}\|_\infty 
        &\leq \sum_{|m|\geq k_2} |\varphi_2(m)| |\mathcal{F}(w_{q+1})(m)| \leq\norm{\varphi_2}_{\infty} \sum_{|m|\geq k_2} \frac1{|m|^{24}} \leq \frac{\la_\qbn}{k_2^{20}}\,.
    \end{align*}
\end{proof}

Before estimating the helicity from $w_{q+1,h}^{(c)}$ and $w_{q+1,R}$, we introduce the following lemma.
\begin{lemma}[\textbf{Commutators and dodging}]\label{lem.multcomm}
Let $p\in [1,\infty]$ and $\la$, $\La$, and $\Ga$ be positive constants. Let $G:\T^3\to\R$ and $H:\T^3\to\R$ satisfy $\|G\|_{C^1} \leq \const_G\lambda$ and $\|H\|_{L^p}\leq \const_H$.
Then, for any multiplier $\varphi:\R^3 \rightarrow B$ such that
\begin{equation}\label{schwartz:decay:redux:redux:commlem}
    \| \check \varphi \|_{L^1(\R^3)}
    \leq 1\,, \qquad
    \| \mathbf{1}_{F} \check \varphi \|_{L^1(\R^3)} \leq \La^{-100} \, , \qquad \textnormal{ where } F = \left( B_{{(\Lambda \Gamma^{-2})^{-1}}} \right)^c \, ,
\end{equation}
we have the following.
    \begin{enumerate}[(i)]
        \item The commutator $[\bp_\varphi, G] H:=\bp_\varphi(GH) - G \bp_\varphi H $ satisfies
\begin{equation}\label{comm.est}
        \|[\bp_\varphi, G] H\|_{L^p} \lesssim \la\Ga^2\La^{-1} \const_G \const_H\,. 
    \end{equation}
    \item 
    The multiplier operator can be decomposed into  $\bp_\varphi = \bp_\varphi^{\text{loc}}+\bp_\varphi^{\text{nl}}$ such that the local part satisfies
\begin{equation}\label{comm.loc.supp}
\supp(\bp_\varphi^{\text{loc}}(GH)) \subset B(\supp(GH),\La^{-1}\Ga^2) \,
    \end{equation}
and the nonlocal part satisfies   \begin{equation}\label{comm.nl.est}
\|\bp_\varphi^{\text{nl}} (G H)\|_{L^p} \lesssim \la\La^{-100} \const_G \const_H\,.
    \end{equation}
    
    \end{enumerate}
    
\end{lemma}
\begin{proof}
Denote by $\varphi^{loc}$ the multiplier with kernel $\check \varphi^{loc} = \mathbf{1}_{B_{\La^{-1}\Ga^2}} \check \varphi$ and by $\varphi^{nl}$ the multiplier with kernel $\check \varphi^{nl} = \mathbf{1}_{F}\check\varphi$. Then $\varphi = \varphi^{loc} + \varphi^{nl}$. Now we define $\bp_\varphi^{loc} := \bp_{\varphi^{loc}}$ and $\bp_\varphi^{nl} := \bp_{\varphi^{nl}}$. Note that $[\bp_\varphi, G] = [\bp_{\varphi^{loc}}, G] + [\bp_{\varphi^{nl}}, G]$. Since $\supp(\check \varphi^{loc})\subset B_{\La^{-1}\Ga^2}$, we have~\eqref{comm.loc.supp}. We estimate the local part of the commutator:
    \begin{align*}
        \|[\bp_{\varphi^{loc}}, G]H\|_{L^p} &= \left\| \int_{\R^3} \check \varphi^{loc}(y)\left(G(x-y) - G(x) \right) H(x-y)\,dy \right\|_{L^p(\T^3)}\\
        &\leq \left\| \int_{\R^3} \left|\check \varphi^{loc}(y)\left(G(x-y) - G(x) \right) H(x-y)\right|\,dy \right\|_{L^p(\T^3)}\\
        &\leq \|\na G\|_{L^\infty} \left\| \int_{B_{\La^{-1}\Ga^2}} |\check \varphi(y)||y| |H(x-y)|\,dy \right\|_{L^p} \\
        &\leq \|\na G\|_{L^\infty}\La^{-1}\Ga^2 \left\| \int |\check \varphi(y)| |H(x-y)|\,dy \right\|_{L^p}\\
        &\lesssim \la \Ga^2{\La}^{-1}\const_G \const_H\,.
    \end{align*} 
The estimate of the nonlocal part \eqref{comm.nl.est}, on the other hand, can be obtained as follows:
\begin{align*}
\|\bp_{\varphi^{nl}}(GH)\|_{L^p} &= \left\|\int_{\R^3} \check\varphi^{nl}(y)(GH)(x-y)\,dy\right\|_{L^p}\lesssim \|\check\varphi^{nl}\|_{L^1}\|GH\|_{L^p} \leq \la\La^{-100}\const_G\const_H\,.
\end{align*}
    Similarly,
    \begin{align*}
        \|G\bp_{\varphi^{nl}}H\|_{L^p} &\leq \|G\|_{L^\infty} \|\bp_{\varphi^{nl}}H\|_{L^p} \lesssim \const_G \|\check\varphi^{nl}\|_{L^1}\|H\|_{L^p} \leq \La^{-100}\const_G\la\|H\|_{L^p}\,.
    \end{align*}
    Putting these estimates together yields~\eqref{comm.est}.
\end{proof}

\begin{lemma}[\textbf{Computing negligible contributions to helicity}]\label{lem:helicity.pc}
Let $\varphi:\R^3 \rightarrow B$ be any multiplier such that\footnote{In a slight abuse of notation, we allow $\varphi \equiv 1$, so that $\check \varphi = \delta_0$ is the Dirac delta.}
\begin{equation}\label{schwartz:decay:redux:redux}
    \| \check \varphi \|_{L^1(\R^3)} \leq 1\, , 
    \qquad
    \| \mathbf{1}_{F_{q+1}} \check \varphi \|_{L^1(\R^3)} \leq \lambda_\qbn^{-100} \, , \qquad \textnormal{ where } F_{q+1} = \left( B_{(\lambda_{q+\bn} \Gamma_{q+\bn}^{-2})^{-1}} \right)^c \, .
\end{equation}
Then the new perturbation $w_{q+1} = w_{q+1}^{(p)} + w_{q+1}^{(c)}+ w_{q+1,h}^{(p)} + w^{(c)}_{q+1,h}$ satisfies
\begin{equation}\label{hab1}
\left| \int_{\T^3} \bp_\varphi(w_{q+1}) \cdot \curl w_{q+1} -  \int_{\T^3} \bp_\varphi(w_{q+1,h}^{(p)}) \cdot \curl w_{q+1,h}^{(p)} \right| \lec \Ga_{q+1}^{-2} \, .
\end{equation}
In addition, for all $M \leq \sfrac{\Nfin}{3}$,
\begin{equation}\label{hab1:derivs}
\left| \ddt^M \left( \int_{\T^3} (w_{q+1}) \cdot \curl w_{q+1} -  \int_{\T^3} (w_{q+1,h}^{(p)}) \cdot \curl w_{q+1,h}^{(p)} \right) \right| \les \MM{M, \Nindt, \mu_q^{-1}\Gamma_q^{20}, \Tau_q^{-1}\Ga_q^8} \, .
\end{equation}
\end{lemma}

\begin{proof}\, 
Using~\eqref{defn:w} and
the dodging properties from Lemma~\ref{lem:dodging}, specifically~\eqref{eq:dodging:newbies}, we have that
\begin{align}
    &H(\bp_{\varphi}w_{q+1}, w_{q+1})- H(\bp_{\varphi}w_{q+1,h}^{(p)},w_{q+1,h}^{(p)})\\
    &= H(\bp_{\varphi}w_{q+1, R}^{(p)},w_{q+1, R}^{(p)})
    +2H(\bp_{\varphi}w_{q+1,R}^{(p)},w_{q+1,R}^{(c)})
    +H(\bp_{\varphi}w_{q+1,R}^{(c)},w_{q+1,R}^{(c)}) \notag \\
    &\quad
+2H(\bp_{\varphi}w_{q+1,h}^{(p)},w_{q+1,h}^{(c)})
    +H(\bp_{\varphi}w_{q+1,h}^{(c)},w_{q+1,h}^{(c)}) \, . \label{eq:helicity.decomp}
\end{align}
We therefore estimate each term on the right-hand side of \eqref{eq:helicity.decomp}. 

\bigskip

\noindent\texttt{Case 1: Euler-Reynolds principal-principal part $H(\bp_{\varphi}w_{q+1,R}^{(p)},w_{q+1,R}^{(p)})$.}

\noindent\texttt{Step 1a: Commutator and off-diagonal terms.} By using the dodging properties between $\{w_{(\xi)}^{(p),I}\}$ from~\ref{eq:dodging:newbies}, and denoting $ A=((\nabla {\Phi}_{(i,k)})^{-1})$, we compute the mollified helicity of $w_{q+1,R}^{(p)}$
\begin{align*}
&H\left( \bp_{\varphi} \left( w_{q+1,R}^{(p)} \right) , w_{q+1,R}^{(p)}  \right)\\
&= \int_{\T^3} \bp_{\varphi} \left( w_{q+1,R}^{(p)} \right) \cdot \curl w_{q+1,R}^{(p)} dx \\
&=\underbrace{\sum_{\substack{i,j,k,\xi,\vecl,I}} \int_{\T^3} \bp_{\varphi} \left( w_{(\xi)}^{(p),I} \right) \cdot \curl w_{(\xi)}^{(p),I} dx}_{\text{(diagonal)}}+ \underbrace{\sum_{\substack{i,j,k,\xi,\vecl,I}} \int_{\T^3} \bp_{\varphi} \left(w_{q+1,R}^{(p)} - w_{(\xi)}^{(p),I} \right) \cdot \curl w_{(\xi)}^{(p),I} dx}_{\text{(off-diagonal)}}\,.
\end{align*}
We first deal with the off-diagonal term and show that it is negligibly small. Indeed, by the dodging property~\eqref{eq:dodging:newbies}, we have that 
$$ \supp \left(w_{q+1,R}^{(p)} - w_{(\xi)}^{(p),I} \right) \cap B\left(\supp (w_{(\xi)}^{(p),I}), \frac14 \la_\qbn^{-1}\Ga_q^2\right) = \emptyset\,.  $$
Then, Lemma~\ref{lem.multcomm} implies that 
\begin{align*}
    \|\bp_{\varphi} \left(w_{q+1,R}^{(p)} - w_{(\xi)}^{(p),I} \right)&\|_{L^\infty(\supp w_{(\xi)}^{(p),I})} \leq \|\bp_{\varphi^{nl}} \left(w_{q+1,R}^{(p)} - w_{(\xi)}^{(p),I} \right)\|_{L^\infty}\\
    &\leq \|\mathbf{1}_{F_{q+1}} \check \varphi\|_{L^1}\|w_{q+1,R}^{(p)} - w_{(\xi)}^{(p),I}\|_{L^\infty} \lesssim \la_\qbn^{-100} \Gamma_q^{\frac \badshaq 2 +   {16}} r_{q}^{-1} \, .
\end{align*}
Thus, we obtain
\begin{align*}
    |\text{(off-diagonal)}| \leq \|\bp_{\varphi} \left(w_{q+1,R}^{(p)} - w_{(\xi)}^{(p),I} \right)\|_{L^\infty(\supp w_{(\xi)}^{(p),I})} \|\curl w_{(\xi)}^{(p),I}\|_1 &\lec \la_\qbn^{-99} \Gamma_q^{\frac \badshaq 2 +   {20}}  \underset{\eqref{def:thursday}}{\leq}  \Ga_{q+1}^{-2}
\end{align*}
which is negligibly small. We move on to estimate the diagonal term:
\begin{align*}
\text{(diagonal)} &= \sum_{\substack{i,j,k,\xi,\vecl,I}} 
\int_{\T^3} \bp_{\varphi} \left( a_{(\xi)} \left(\chib_{(\xi)} \etab_{\xi}^{I} \right) \circ \Phiik  (\nabla {\Phi}_{(i,k)})^{-1} \xi \varrho_{(\xi)}^I \circ \Phi_{(i,k)}  \right) \cdot 
    \curl w_{(\xi)}^{(p),I} dx\\
    &= \underbrace{\sum_{\substack{i,j,k,\xi,\vecl,I}} 
\int_{\T^3} A \xi \bp_{\varphi} \left( a_{(\xi)} \left(\chib_{(\xi)} \etab_{\xi}^{I} \right) \circ \Phiik  \varrho_{(\xi)}^I \circ \Phi_{(i,k)} \right) \cdot 
    \curl w_{(\xi)}^{(p),I} dx}_{\text{(main)}}\\
    &\qquad + \underbrace{\sum_{\substack{i,j,k,\xi,\vecl,I}} 
\int_{\T^3} [\bp_{\varphi}, A \xi] \left( a_{(\xi)} \left(\chib_{(\xi)} \etab_{\xi}^{I} \right) \circ \Phiik  \varrho_{(\xi)}^I \circ \Phi_{(i,k)} \right) \cdot 
    \curl w_{(\xi)}^{(p),I} dx}_{\text{(commutator)}}
    \end{align*}
We now estimate the commutator term above using Lemma~\ref{lem.multcomm} with the following choices:
\begin{align*}
    p=2,\,&\varphi = \varphi,\,\La=\la_\qbn,\,F=F_{q+1},\,G= (\nabla {\Phi}_{(i,k)})^{-1} \xi,\,\const_G = \tau_q^{-1}\mu_q,\, \la=\la_q,\, \Ga= \Ga_{q+\bn}\\
    & H = a_{(\xi)} \left(\chib_{(\xi)} \etab_{\xi}^{I} \right) \circ \Phiik  \varrho_{(\xi)}^I \circ \Phi_{(i,k)},\,
    \const_H = \left| \supp \left(\eta_{i,j,k,\xi,\vecl}\zetab_\xi^{I} \right) \right|^{\sfrac 12} \delta_{q+\bn}^{\sfrac 12}\Gamma_q^{j+ {10}}\,.
\end{align*}
We then obtain from~\eqref{comm.est} that
\begin{align*}
    \left\|[\bp_{\varphi}, A \xi] \left( a_{(\xi)} \left(\chib_{(\xi)} \etab_{\xi}^{I} \right) \circ \Phiik  \varrho_{(\xi)}^I \circ \Phi_{(i,k)} \right)\right\|_{L^2} \leq \frac{\tau_q^{-1}\mu_q\la_q}{\la_\qbn} \left| \supp \left(\eta_{\pxi}\zetab_\xi^{I} \right) \right|^{\sfrac 12} \delta_{q+\bn}^{\sfrac 12}\Gamma_q^{j+ {10}}
\end{align*}
Also, from~\eqref{eq:w:oxi:est:master}, we get that $\left\|\curl w_{(\xi)}^{(p),I}\right\|_{L^2} \lec \left| \supp \left(\eta_{\pxi}\zetab_\xi^{I} \right) \right|^{\sfrac 12} \delta_{q+\bn}^{\sfrac 12}\Gamma_q^{j+ {10}} \la_{\qbn}$. Note that $\curl w_{(\xi)}^{(p),I}$ is supported on $\supp \left( \eta_{\pxi} \, \zetab_{\xi}^{I }\circ\Phiik \right)$ and so, by Corollary~\ref{rem:summing:partition} with $p=1, \theta_1=0, \theta_2=2$, 
\begin{align*}
    |\text{(commutator)}| &\leq \sum_{\substack{i,j,k,\xi,\vecl,I}} 
\left\|[\bp_{\varphi}, A \xi] \left( a_{(\xi)} \left(\chib_{(\xi)} \etab_{\xi}^{I} \right) \circ \Phiik  \varrho_{(\xi)}^I \circ \Phi_{(i,k)} \right)\right\|_{L^2} 
    \left\|\curl w_{(\xi)}^{(p),I}\right\|_{L^2} \notag\\
    &\leq \frac{\tau_q^{-1} \mu_q \la_q}{\lambda_\qbn} \delta_\qbn \lambda_{\qbn}\Ga_q^{26} \underset{\eqref{eq:defn:tau}}{\leq} \Gamma_{q+1}^{-4}
\end{align*}
\medskip

\noindent\texttt{Step 1b: Main term and time derivatives.} Now we estimate the main term with mollifiers but no time derivatives. Writing $B^I_{(\xi)} = a_{(\xi)} \left(\chib_{(\xi)} \etab_{\xi}^{I} \right) \circ \Phiik  \varrho_{(\xi)}^I \circ \Phi_{(i,k)}$ and suppressing all indices in the summand, we decompose based on whether curl operator falls on $A \xi$ or not:     
\begin{align}
\text{(main)}  &= 
{\sum_{\substack{i,j,k,\xi,\vecl,I}} 
    \int_{\T^3} A\xi (\bp_\varphi B) B  \cdot \curl (A \xi) dx} + {\sum_{\substack{i,j,k,\xi,\vecl,I}} 
 \int_{\T^3}  A \xi (\bp_\varphi B)  \cdot\left((\na B)\times
 (A \xi) \right) dx}\, .
\label{h1h2:er}
\end{align}
The second term vanishes due to the fact that $(\na B) \times (A\xi)$ is orthogonal to $A\xi$. 

So it remains to estimate the first term. 
Since $\xi$ is constant, $\xi \cdot \curl \xi=0$ and \eqref{eq:Lagrangian:Jacobian:2} gives
$$ \left| A\xi \cdot \curl(A\xi) \right| = \left| A\xi \cdot \curl (A\xi) - \xi \cdot \curl \xi \right| \les \tau_q^{-1}\Gamma_q^{-5}\mu_q \, . $$
Combining this with Corollary~\ref{rem:summing:partition} with $\theta_1=0$ and $\theta_2=2$, we estimate the first term by
\begin{align}
\left|\sum_{\substack{i,j,k,\vecl,\xi, I}} \int_{\T^3} (\bp_\varphi B) B  \left( A \xi \cdot \curl (A \xi) - \xi \cdot \curl \xi \right) \,  dx \right| 
    &\lec
\tau_q^{-1}\mu_q\Ga_q^5 (\la_q \Ga_q)\cdot \de_{q+\bn}\underset{\eqref{eq:defn:tau}}{\leq} \Ga_{q+1}^{-2} \label{514:1}
\end{align}

Finally, we estimate time derivatives of $H(w_{q+1,R}^{(p)})$.  Since there is no mollifier present, we need only consider the analogue of~\eqref{h1h2:er}, but with time derivatives and no mollifiers:
$$ \left| \ddt^M {\sum_{\substack{i,j,k,\xi,\vecl,I}} 
    \int_{\T^3} A\xi  B^2 \cdot \curl (A \xi) dx} + \ddt^M {\sum_{\substack{i,j,k,\xi,\vecl,I}} 
 \int_{\T^3}  A \xi B \cdot\left((\na B)\times
 (A \xi) \right) dx}\, \right| . $$
The second term vanishes again, and we add and subtract $\xi\cdot\curl\xi$ in the first term again.  Then converting time derivatives outside the integral to material derivatives inside the integral and using Corollary~\ref{cor:deformation}, Lemma~\ref{lem:a_master_est_p}, and Corollary~\ref{rem:summing:partition} as above, we obtain an estimate analogous to~\eqref{514:1}, but multiplied by the cost of the time derivatives, consistent with~\eqref{hab1:derivs}. 
\bigskip

\noindent\texttt{Case 2: Helical corrector parts $2H(\bp_{\varphi}w_{q+1,h}^{(p)}, w_{q+1,h}^{(c)}) + H(\bp_{\varphi}w_{q+1,h}^{(c)})$.}
The estimates are made in the same fashion as \texttt{Step 1} of the proof of Lemma~\ref{lem:helicity.h}, in which we proved~\eqref{hab} and~\eqref{derivs:1}; notice these estimates required no orthogonality properties and were made based purely on amplitude.  However, the estimates in this case are even better, since the helicity corrector $w_{q+1,h}^{(c)}$ obeys a better bound than $w_{q+1,h}^{(p)}$ from Corollary~\ref{cor:corrections:Lp}.  We omit further details.

\bigskip
\noindent\texttt{Case 3: Euler-Reynolds principal-corrector part $2H(\bp_{\varphi}w_{q+1,R}^{(p)}, w_{q+1,R}^{(c)})$.} As in \texttt{Case 1, Step 1a}, we first decompose into a diagonal and off-diagonal part:
\begin{align*}
&\int_{\T^3}  \bp_\varphi w_{q+1,R}^{(p)} \cdot \curl  w_{q+1,R}^{(c)} dx \\
&= \underbrace{\sum_{i,j,k,\xi,\vecl,I}  \int_{\T^3}  \bp_\varphi w_{q+1,(\xi)}^{(p),I} \cdot \curl w_{q+1,(\xi)}^{(c),I}\,dx}_{\text{(diagonal)}}+ \underbrace{\sum_{i,j,k,\xi,\vecl,I}  \int_{\T^3}  \bp_\varphi \left(w_{q+1,R}^{(p)} - w_{q+1,(\xi)}^{(p),I} \right)\cdot \curl w_{q+1,(\xi)}^{(c),I}\,dx}_{\text{(off-diagonal)}}\,.
\end{align*}
The off-diagonal term is negligibly small by the same argument as before. We move on to the diagonal term. We set $B=B_{(\xi)}^I = a_{(\xi)} \left(\chib_{(\xi)} \etab_{\xi}^{I} \right) \circ \Phiik$, drop indices for concision, and write
\begin{align}
    \text{(diagonal)}
    &=  \sum_{i,j,k,\xi,\vecl,I} 
    \int_{\T^3} 
    \bp_\varphi w_{q+1,(\xi)}^{(p),I}  \eps_{mln} \eps_{nop}
    \pa_l\left(\pa_o B    \pa_p \Phi^r {\mathbb{U}}_r( \Phi)
    \right) dx \notag \\
   &= \underbrace{\sum_{i,j,k,\xi,\vecl,I} 
    \int_{\T^3} 
    \bp_\varphi w_{q+1,(\xi)}^{(p),I} \eps_{mln}\eps_{nop} \pa_l\left(\pa_o B     \pa_p \Phi^r \right) {\mathbb{U}}_r( \Phi)
     dx}_{I_1} \notag \\
    &\quad+
\underbrace{\sum_{i,j,k,\xi,\vecl,I} 
\int_{\T^3} 
    \bp_\varphi w_{q+1,(\xi)}^{(p),I} \eps_{mln}\eps_{nop}
    \pa_o B \pa_p \Phi^r \pa_l\left({\mathbb{U}}_r( \Phi)\right) 
     dx}_{I_2}  \, .  \label{i1i2}
\end{align}
We estimate $I_1$ directly. Using aggregation lemmas as usual,~\eqref{e:a_master_est_p_R:zeta}, and~\eqref{e:pipe:estimates:2},
\begin{align}
    |I_1| &\leq \sum_{i,j,k,\xi,\vecl,I} 
    \int_{\T^3} 
    \left| \bp_\varphi w_{q+1,(\xi)}^{(p),I} \eps_{mln}\eps_{nop} \pa_l\left(\pa_o B \pa_p \Phi^r \right) {\mathbb{U}}_r( \Phi)\right|
     dx \notag \\
     &\leq \sum_{i,j,k,\xi,\vecl,I} \left\|\bp_\varphi w_{q+1,(\xi)}^{(p),I}\right\|_{L^2} \|\eps_{mln}\eps_{nop} \pa_l\left(\pa_o B  \pa_p\Phi^r \right) {\mathbb{U}}_r( \Phi)\|_{L^2} \notag \\
     &\lesssim \sum_{i,j,k,\xi,\vecl,I} \left\| w_{q+1,(\xi)}^{(p),I}\right\|_{L^2} \|D\left(D B    (\na\Phi)^T\right)\|_2 \| {\mathbb{U}}( \Phi)\|_{2} \lec \frac{\la_{q+\bn/2}^2}{\la_{\qbn}}\de_{\qbn} \Gamma_\qbn^{500} \underset{\eqref{par:helcc}}{\lec} \Gamma_{\qbn}^{-100} \, . \label{514:2}
\end{align}
We move now to $I_2$; recalling that $A=(\nabla\Phi)^{-1}$, we write this term as
\begin{align*}
    I_2 &= \underbrace{\sum_{i,j,k,\xi,\vecl,I} 
\int_{\T^3} 
    A\xi \, \bp_\varphi \left( B \, \varrho(\Phi) \right) \eps_{mln}\eps_{nop}
    \pa_o B    \pa_p\Phi^r\pa_l\left({\mathbb{U}}_r( \Phi)\right) 
     dx}_{\text{(main)}}\\
     &\qquad + \underbrace{\sum_{i,j,k,\xi,\vecl,I} 
\int_{\T^3} 
    \left[\bp_\varphi, A\xi \right]\left( A   \,  \varrho(\Phi) \right) \eps_{mln}\eps_{nop}
    \pa_o B    \pa_p\Phi^r\pa_l\left({\mathbb{U}}_r( \Phi)\right) 
     dx}_{\text{(commutator)}}
\end{align*}
The commutator term can be estimated similarly as in \texttt{Case 1, Step 1a}, yielding
$$ |\text{(commutator)}| \lesssim \frac{\mu_q\tau_q^{-1}\lambda_q}{\lambda_\qbn} \delta_\qbn \lambda_{q+\half}\Ga_q^{26} \leq \Gamma_\qbn^{-4} \, . $$
Next, using $\eps_{mln}\eps_{nop} = \de_{mo}\de_{lp}-\de_{mp}\de_{lo}$, the main term from $I_2$ can be rewritten as
\begin{align}
    |\text{(main)}| &= \sum_{i,j,k,\xi,\vecl,I} \int 
    \left(A \xi \right)_m \bp_\varphi \left( B   \, \varrho(\Phi) \right) \left( \pa_m B  
       \pa_l\Phi^r - \pa_l B \pa_m \Phi^r \right)\pa_l \Phi^s (\pa_s {\mathbb{U}}_r)( \Phi)
     dx  \,. \label{i21i22}
\end{align}
We note that the term with $\pa_m\Phi^r$ above vanishes due to orthogonality; indeed as ${\mathbb{U}} \cdot \vec \xi =0$, $\pa_s {\mathbb{U}}$ remains orthogonal to $\vec e_2$ for any $s$, and so using $A_j^m \pa_m\Phi^r=\delta_{jr}$,
$$ (A\xi)_m \pa_m \Phi^r (\pa_s \UU_r)(\Phi)= A_j^m \xi^j \pa_m \Phi^r (\pa_s {\mathbb{U}}_r)(\Phi) 
= \xi_r (\pa_s {\mathbb{U}}_r)(\Phi) =0 \, . $$
So we must estimate the term with $\pa_m B\pa_l\Phi^r$. We recall from~\eqref{e:a_master_est_p_R:zeta} that
$$ \left\|\left(A \xi \right)_m \pa_m B \right\|_{2} \lesssim \left| \supp \eta_{i,j,k,\xi,\vecl} \right|^{\sfrac 12} \delta_{q+\bn}^{\sfrac 12} \Gamma^{j+10}_{q} {\Gamma_q^{  5}\Lambda_{q}} \, . $$
Next, we have from Corollary~\ref{cor:deformation} and~\eqref{e:pipe:estimates:2} that 
$$ \|D_{t,q}\left(\pa_{ {l}}\Phi^r \pa_l \Phi^s (\pa_s {\mathbb{U}}_r)( \Phi)\right)\|_2 = \|D_{t,q}\left(\pa_{ {l}}\Phi^r \pa_l \Phi^s\right) (\pa_s {\mathbb{U}}_r)( \Phi)\|_2 \lec \tau_q^{-1} \Ga_q^{-5}$$
and that when $\Phi \equiv \Id$ at $t=k\mu_q\Ga_q^{-i-2}$, $\pa_{ {l}}\Phi^r \pa_l \Phi_s (\pa_s {\mathbb{U}}_r)( \Phi) = \pa_l \UU_l \equiv 0$. Thus, by transport estimates on the $\mu_q$ timescale, we have that $\|\pa_l \Phi^r \pa_l \Phi^s (\pa_s {\mathbb{U}}_r)( \Phi) \|_2 \lec \tau_q^{-1} \Ga_q^{-5} \mu_q$. Using the above bounds and aggregation estimates as usual, we conclude that
\begin{align*}
    |I_2| &\leq \sum_{i,j,k,\xi,\vecl,I} \left\| 
    \left(A \xi \right)_m\pa_m B \pa_l \Phi^r \pa_l \Phi_s (\pa_s {\mathbb{U}}_r)(\Phi) \right\|_2 \left\|\bp_\varphi \left(B \, \varrho(\Phi) \right)  \right\|_2\\
    &\lec \sum_{i,j,k,\xi,\vecl,I} \left\| 
    \left(A \xi \right)_m\pa_m B \right\|_2 \left\|\pa_l\Phi^r \pa_l \Phi_s (\pa_s {\mathbb{U}}_r)(\Phi) \right\|_2 \left\|B \, \varrho(\Phi) \right\|_2\\
    &\lec \Ga_q^3 \de_\qbn^{\sfrac 12} \Ga_q^{15}\La_q \tau_q^{-1}\Ga_q^{-5}\mu_q \de_{\qbn}^{\sfrac 12} \underset{\eqref{eq:defn:tau}}{\lec} \Ga_{q+1}^{-2}\,.
\end{align*}

Finally, we must estimate time derivatives of $H(w_{q+1,R}^{(p)},w_{q+1,R}^{(c)})$.  As there is no mollifier present, we need only take time derivatives of the diagonal terms $I_1$ and $I_2$ where $\varphi \equiv 1$.  Estimating $I_1$ but incorporating time derivatives (which are converted to material derivatives inside the integral as usual), we get an estimate analogous to~\eqref{514:2}, but multiplied the cost of material derivatives, consistent with~\eqref{hab1:derivs}.  Next, $I_2$ has no commutator term, and the second term from~\eqref{i21i22} still vanishes by orthogonality.  Converting the time derivatives outside the integral in the first to to material derivatives inside the integral, but retracing the above estimates with extra material derivatives, we obtain a bound consistent with~\eqref{hab1:derivs}.

\bigskip
\noindent\texttt{Case 4: Euler-Reynolds corrector-corrector part $H(\bp_{\varphi}w_{q+1,R}^{(c)},w_{q+1,R}^{(c)})$.}
A heuristic estimate gives
\begin{align*}
\left|\int_{\T^3} \curl w_{q+1,R}^{(c)}\cdot w_{q+1,R}^{(c)} dx \right|
\lec \de_{q+\bn} \frac{\la_{q+\frac{\bn}2}^2}{\la_{q+\bn}^2} \la_{q+\bn} \Ga_q^{10}
\underset{\eqref{par:helcc}}{\leq} \Ga_{q+1}^{-2} \, .
\end{align*}
Using aggregation lemmas and taking time derivatives as in the previous steps, the above estimate can be made rigorous and is consistent with~\eqref{hab1} and~\eqref{hab1:derivs}. 
\end{proof}

\begin{proposition}[\textbf{Verifying Hypothesis~\ref{hyp:helinashell}}]   
\label{prop:hel:shell}
Hypothesis~\ref{hyp:helinashell} is true at step $q+1$.
\end{proposition}
\begin{proof}
Let $\varphi$ be as in the hypothesis. Note that from~\eqref{hab} and~\eqref{hab1}, we immediately have that
    $$ \left|\int_{\T^3} \bp_\varphi (w_{q+1}) \cdot\curl w_{q+1}\right| < \frac15 \Ga_{q-4}^{-1/2}\,. $$
Since $w_{q+1}-\hat w_\qbn$ is negligibly small from~\eqref{eq:diff:moll:vellie:statement},  Hypothesis~\ref{hyp:helinashell} then follows from direct computation; we omit further details.
\end{proof}

\begin{lemma}[\textbf{Helicity interaction between $\hat w_{q+\bn}$ and $u_q$}]\label{lem:helicity.u} For $M \leq \sfrac{\Nfin}{3}$,
\begin{align}\label{eq:helicity.u}
    \left| \ddt^{M} H(u_q, \hat w_{q+\bn}) \right| \leq \Ga_{q+1}^{-10} \MM{M, \Nindt, \mu_q^{-1}, \Tau_q^{-1}\Ga_q^8}  \,. 
\end{align}    
\end{lemma}
\begin{proof}
We first claim that we have $ |\ddt^{M} H(u_q, w_{q+1})| \leq \Ga_{q+1}^{-15}\MM{M, \Nindt, \mu_q^{-1}, \Tau_q^{-1}\Ga_q^8} $. To prove this, we only estimate $H(u_q, w_{q+1,R}^{(p)})$. The estimates of $H(u_q, w_{q+1,R}^{(c)})$ and $H(u_q, w_{q+1,h})$ then follow similarly, because $w_{q+1,R}^{(c)}$ and $w_{q+1,h}$ have the same structure as $w_{q+1,R}^{(p)}$. Recalling that
\begin{align*}
w_{q+1,R}^{(p)}   = \sum_{i,k,\xi,\vecl,I} \underbrace{a_{(\xi)} (t,y)\left(\chib_{(\xi)} \etab_{\xi}^{I} \right) \circ \Phiik  \nabla\Phi_{(i,k)}^{-1} }_{=:\tilde{a}}
    \mathbb{W}_{(\xi)}^I \circ \Phi_{(i,k)} \, , 
\end{align*}
and using that $\hat u_q$ is divergence-free (so that $\int_{\T^3} d/dt (\cdot) = \int_{\T^3} \Dtq (\cdot)$), we write
\begin{align*}
&\left| \ddt^M \int_{\T^3}  w_{q+1}^{(p)}\cdot \curl u_q\, dx\right| \\
&=\left| \sum_{0\leq M' \leq M}  c_{M,M'} \sum_{i,k,\xi,\vecl,I}\int_{\T^3} \left(\Dtq^{M'}  \tilde{a} \right) \mathbb{W}_{(\xi)}^I \circ \Phi_{(i,k)}  \cdot \Dtq^{M-M'} \left( \curl u_q \right)\, dx\right|\\
&= \left| \sum_{0 \leq M' \leq M} c_{M,M'} \sum_{i,k,\xi,\vecl,I}\int_{\T^3}  \xi \lambda_{q+\bn}^{-\dpot }\div^\dpot  \left(\vartheta^k_{\xi,\lambda_{q+\bn},r_{q}}\right) \cdot \left(\Dtq^{M'} \tilde{a} \, \Dtq^{M-M'} \curl u_q\right)\circ\Phi_{(i,k)}^{-1} dx\right|\\
&= \lambda_{q+\bn}^{-\dpot }\left| \sum_{0 \leq M' \leq M} c_{M,M'} \sum_{i,k,\xi,\vecl,I}\int_{\R^3}    \left(\vartheta^k_{\xi,\lambda_{q+\bn},r_{q}}\right): \na^\dpot\left( \left(\Dtq^{M'}\tilde{a} \xi \cdot \Dtq^{M-M'} \curl u_q\right)\circ\Phi_{(i,k)}^{-1}\right) dx
\right|\\
&\lec \la_{q+\bn-1}\left(\frac{\la_{q+\bn-1}}{\la_{q+\bn}}\right)^{\dpot} \de_{q+\bn} \MM{M, \Nindt, \mu_q^{-1}\Gamma_q^{20}, \Tau_q^{-1}\Ga_q^8}\\
&\leq \Ga_{q+1}^{-20} \MM{M, \Nindt, \mu_q^{-1}\Gamma_q^{20}, \Tau_q^{-1}\Ga_q^8} \, . 
\end{align*}
Note in the last inequality, we have used a large choice of $\dpot$ in~\eqref{i:par:10}. Thus we obtain the claim.

Finally, the estimate for $\ddt^M H(u_q, \hat w_\qbn)$ follows from~\eqref{eq:diff:moll:vellie:statement}.
\end{proof}

\begin{proposition}[\textbf{Verifying Hypothesis~\ref{hyp:helicity.prescription}}]
\label{prop:hel,prescription}
Hypothesis~\ref{hyp:helicity.prescription} is true at step $q+1$.
\end{proposition}
\begin{proof}
    Let $\varphi$ be as in the hypothesis at step $q+1$. Note that it suffices to prove that
    \begin{align}\label{eq:new.hel}
        2\Ga_{q+1}^{-1} \leq h(t) - H(u_{q+1}) &\leq \frac{15}{16} \Ga_q^{-1} \, , \\
\label{eq:hel.moll.diff}
        \left| H(u_{q+1}) - \int_{\T^3} \bp_\varphi (u_{q+1})\cdot \curl u_{q+1} \right| &< \Ga_q^{-2}\,.
    \end{align}
To prove~\eqref{eq:new.hel}, we see from~\eqref{hel:inc:est},~\eqref{hab1} with $\varphi \equiv 1$, and~\eqref{eq:diff:moll:vellie:statement}  that 
$$ \frac{\Ga_{q-1}}{2\Ga_q^2} \leq h(t) - H(u_q) - H(\hat w_{q+\bn}) \leq \frac78 \Ga_q^{-1}\,. $$
Since $H(u_{q+1}) = H(u_q) + H(\hat w_{q+\bn}) + 2H(u_q, \hat w_{q+\bn})$,~\eqref{eq:new.hel} follows from this and Lemma~\ref{lem:helicity.u}.

It remains to prove~\eqref{eq:hel.moll.diff}. Notice that
$$ H(u_{q+1}) - \int_{\T^3} \bp_\varphi (u_{q+1})\cdot \curl u_{q+1} =  \int_{\T^3} \bp_{1-\varphi} (u_{q+1})\cdot \curl u_{q+1} $$
and that the multiplier $\td\varphi := (1-\varphi)/2$ satisfies $\td\varphi : \Z^3 \to B$ and $\td\varphi(m) = 0$ for all $|m|\leq \la_\qbn\Ga_\qbn$. Therefore, by Hypothesis~\ref{hyp:freq:loc} proved at step $q+1$ in Proposition~\ref{prop:hyp:freq:loc}, we have that 
$$ \|\bp_{\td\varphi} u_{q+1}\|_\infty = \sum_{i\leq q+1} \left\|\bp_{\td\varphi} \hat w_{i+\bn-1}\right\|_\infty \leq (q+\bn) \frac{\la_\qbn}{(\la_\qbn\Ga_\qbn)^{20}} < \frac12 \la_\qbn^{-19}\,. $$
Thus we have
\begin{align*}
    \left| \int_{\T^3} \bp_{1-\varphi} (u_{q+1})\cdot \curl u_{q+1}\right| \leq \|\bp_{1-\varphi} (u_{q+1})\|_\infty \|\curl u_{q+1}\|_1 \leq \la_\qbn^{-10}< \Ga_q^{-2}\,.
\end{align*}

\end{proof}

\begin{lemma}[\textbf{Verifying Hypothesis~\ref{hyp:temp}}]\label{lem:hel:derivs}

$H(u_{q+1})$ satisfies
\begin{equation}
    \left\| \left[ \sfrac{d}{dt} \right]^M H(u_{q+1}) \right\|_{L^{\infty}([0,T])} \les \MM{M, \Nindt, \mu_{q}^{-1}\Gamma_q^{20}, \Tau_{q}^{-1}\Ga_q^8} \, .
\end{equation}
\end{lemma}
\begin{proof}
We split $H(u_{q+1}) = H(u_q)+2H(u_q, \hat w_\qbn)+H(w_{q+1})+H( \hat w_\qbn) - H(w_{q+1})$.  Applying~\eqref{hel:lip} to the first term,~\eqref{eq:helicity.u} to the second term,~\eqref{derivs:1} and~\eqref{hab1:derivs} to the third term, and~\eqref{eq:diff:moll:vellie:statement} to the final term yields the result.
\end{proof}

\subsection{Energy increment}\label{sec:energy}
In this subsection, we verify \eqref{eq:resolved:energy} for $u_{q+1} = u_q + \hat {w}_{q+\bn}$. 
We first write
\begin{align*}
    e(t) - \frac 12 \| u_{q+1}(t,\cdot) \|_2^2
    = e(t) -\frac 12 \norm{u_q}_2^2 
    -\frac 12 \|w_{q+1,R}^{(p)}\|_2^2
    - \frac 12 \int |\hat w_{q+\bn} |^2 - |w_{q+1,R}^{(p)}|^2
    - \int u_q \cdot \hat w_{q+\bn}.
\end{align*}
By the construction of $w_{q+1,R}^{(p)}$ (see \eqref{eq:a:xi:def}), recalling that $R_\ell$ is trace-free (see \eqref{def:mollified:stuff}), and following the computation as in \cite[Section 8.2]{GKN23}
we have
\begin{align*}
    \int |w_{q+1, R}^{(p)}|^2 
    &= \sum_{i,j,k,\xi, \vecl, I}
    \int a_{(\xi)}^2  
\left(\chib_{(\xi)} \etab_{\xi}^{I} \right)^2 \circ \Phiik  |\nabla\Phi_{(i,k)}^{-1}\xi|^2
\left(\varrho_{(\xi),R}^I\right)^2 \circ \Phi_{(i,k)}\\
&=
2e(t) - \|u_q\|_2^2 -\de_{q+\bn+1}
+
\sum_{i,j,k,\xi, \vecl}
    \int a_{(\xi)}^2  
\left(\bp_{\neq 0}\chib_{(\xi)}^2 \right) \circ \Phiik  |\nabla\Phi_{(i,k)}^{-1}\xi|^2
\\
&\quad+\sum_{i,j,k,\xi, \vecl, I}
    \int a_{(\xi)}^2  
\left(\chib_{(\xi)} \etab_{\xi}^{I} \right)^2 \circ \Phiik  |\nabla\Phi_{(i,k)}^{-1}\xi|^2
\left(\bp_{\neq 0}\left(\varrho_{(\xi),R}^I\right)^2 \right)\circ \Phi_{(i,k)}\\
&=2e(t) - \|u_q\|_2^2 -\de_{q+\bn+1} + I_1 + I_2\, .
\end{align*}
To estimate $I_1$ and $I_2$, we integrate by parts many times to bound these terms by $\de_{q+\bn+1}/50$.  Thus
\begin{align*}
  \frac 12 \de_{q+\bn+1} - \frac 1{50} \de_{q+\bn+1}   \leq e(t) -\frac 12 \norm{u_q}_2^2 
    -\frac 12 \|w_{q+1,R}^{(p)}\|_2^2
    \leq  \frac 12 \de_{q+\bn+1} + \frac 1{50} \de_{q+\bn+1} 
\end{align*}
Then, using the estimates in Corollary~\ref{cor:corrections:Lp} and Lemma~\ref{lem:mollifying:w} together with Corollary~\ref{rem:summing:partition}, we have
\begin{align*}
 \frac 12 \left|\int |\hat w_{q+\bn} |^2 - |w_{q+1,R}^{(p)}|^2\right|
&\leq \frac 12 \left(
\|\hat w_{q+\bn} - w_{q+1}\|_2 
+\|w_{q+1,R}^{(c)}\|_2 +\norm{w_{q+1,h}}_2\right)
\left(
\|\hat w_{q+\bn}\|_2 + \|w_{q+1,R}^{(p)}\|_2
\right)\\
&\lec \de_{q+\bn}^\frac12 \Ga_q^{50}\left(\frac{\la_{q+\bn}}{\la_\qho^{\frac32}} + \de_{q+\bn}^{\frac12} r_q\right)
\leq \frac 1{50}\de_{q+\bn+1}\, . 
\end{align*}
Arguing as in Lemma \ref{lem:helicity.u} and using \eqref{eq:diff:moll:vellie:statement}, 
$\left| \int u_q \cdot \hat w_{q+\bn}     \right|\leq \frac 1{50} \de_{q+\bn+1}$. Combining gives \eqref{eq:resolved:energy}.

\section{Error estimates}\label{sec:err}

\subsection{Identification of error terms}\label{sec:error.ER}

We define the new Reynolds stress $S_{q+1}$ by adding $\hat w_{q+\bn}$ to the Euler-Reynolds system for $(u_q, p_q, R_q)$ in Lemma~\ref{lem:upgrading}.  We use~\eqref{eq:dodging:newbies} in the string of equalities below, which prevents $w_{q+1, R}$ and $w_{q+1, h}$ from interacting nonlinearly to produce an oscillation error.  
\begin{align}
\div(S_{q+1}) 
&= \pa_t \hat w_{q+\bn} + (u_{q}\cdot \na )\hat w_{q+\bn} + (\hat w_{q+\bn}\cdot \na) u_q 
\notag \\
&\quad
+ \div(\hat w_{q+\bn}\otimes \hat w_{q+\bn}- P_{q+1}\Id + R_\ell) + \div \left( R_q^q -\frac13(\tr R_q^q) \Id- R_\ell \right)  \notag \\
&= \underbrace{(\pa_t + \hat u_q \cdot \na) w_{q+1} + w_{q+1}\cdot\nabla \hat u_q}_{=:\, \div S_{TN}}
+ \underbrace{ \div \left( w_{q+1,h} \otimes w_{q+1, h} \right) }_{=: \div S_{O,h}} \notag \\ 
&\quad + \underbrace{\div \left(\wp_{q+1, R}\otimes \wp_{q+1, R}  - P_{q+1}\Id+ R_\ell \right)}_{=:\, \div S_{O, R}} \notag \\
&\quad + \underbrace{\div \left(\wp_{q+1, R}\otimes_s \wc_{q+1, R}+\wc_{q+1, R}\otimes \wc_{q+1, R}\right)
}_{=:\, \div S_{C}} 
+ \underbrace{\div \left( R_q^q - \frac13 (\tr R_q^q)\Id - R_\ell \right) }_{=:\div S_{M1}} \notag \\
&\quad + 
\underbrace{(\pa_t + \hat u_{q}\cdot \na )(\hat w_{q+\bn}-w_{q+1}) + \div((\hat w_{q+\bn}-w_{q+1} )\otimes  \hat u_q + \hat w_{q+\bn}\otimes \hat w_{q+\bn} - w_{q+1}\otimes w_{q+1} )}_{=:\div S_{M2}}
\, .  \label{new.error}
\end{align}
Here $P_{q+1}$ will be specified later in section \ref{sec:osc.err}. Note that we used \eqref{eq:dodging:oldies} to replace $u_q$ with $\hat u_q$ in the second equality.  Also, we have that $\pa_t w_{q+1} + (\hat u_q\cdot\na)w_{q+1} + w_{q+1}\cdot \na \hat u_q$ has zero mean, and so it can be written in divergence form $\div S_{TN}$.  The same statement holds for  $\div S_{M2}$.  We set
\begin{equation}\label{Rq+1:def}
    R_{q+1} = R_q - R_q^q + S_{q+1} \,, \quad
    p_{q+1} = p_q - \frac13 (\tr R_q^q) + P_{q+1}\, . 
\end{equation}
We then have that $(u_{q+1}, p_{q+1}, R_{q+1})$ solves the Euler-Reynolds system 
\begin{align}
    \partial_t u_{q+1} 
    + \div \left( u_{q+1} \otimes u_{q+1} \right) + \nabla p_{q+1}    
    = \div ( R_{q+1} ) \, , \qquad \div \, u_{q+1} = 0 \, . \label{ER:new:equation}
\end{align}
We will prove that the new stress error $S_{q+1}$ can be decomposed into components ${S}_{q+1}^k$ as $S_{q+1} =\sum_{k=q+1}^{q+\bn} {S}_{q+1}^k$, each of which satisfies estimates consistent with the inductive assumptions.

\subsection{Oscillation error}\label{sec:osc.err}

\begin{lemma}[\textbf{Oscillation error from $w_{q+1, h}$}]
The symmetric stress $S_{O,h} := w_{q+1,h} \otimes w_{q+1,h}$ satisfies the following estimates for $N,M \leq \frac{\Nfin}{10}$:
\begin{subequations}\label{ho}
\begin{align}
    \left\| \psi_{i,q} D^N \Dtq^M S_{O,h} \right\|_{1} &\les \delta_{q+2\bn} \lambda_\qbn^N \MM{M, \Nindt, \mu_q^{-1}\Gamma_q^{i+13}, \Tau_q^{-1}\Gamma_q^8} \, , \label{ho1}  \\
    \left\| \psi_{i,q} D^N \Dtq^M S_{O,h} \right\|_{\infty} &\les \lambda_\qbn^N \MM{M, \Nindt, \mu_q^{-1}\Gamma_q^{i+13}, \Tau_q^{-1}\Gamma_q^8}  \, . \label{ho2}
\end{align}
\end{subequations}

\end{lemma}
\begin{proof}
We appeal to Corollary~\ref{cor:corrections:Lp}; for the $L^1$ estimate, we use~\eqref{eq:w:oxi:h:est:master} and~\eqref{eq:w:oxi:est:master:hc} with $r=2$, the former of which gives the weakest estimate.  Fixing values of $i,k, \vecl, I$ and recalling the notation from~\eqref{wqplusoneoneh}, we have that for $N, M \leq \frac{\Nfin}{10}$, 
\begin{align}
    \left\| D^N \Dtq^M \left( w_{\pxi}^{(h),I} \otimes w_{\pxi}^{(h),I} \right) \right\|_1 \les \left| \supp \left( \eta_{i,k,\xi,\vecl} \zetab_\xi^I \right) \right| \frac{\lambda_\qbn^2}{\lambda_\qho^3} \lambda_\qbn^N \MM{M, \Nindt, \mu_q^{-1}\Gamma_q^{i+13}, \Tau_q^{-1}\Gamma_q^8} \, .
\end{align}
Upon applying the usual aggregation lemmas and parameter inequality~\eqref{eq:qho:choice3}, we obtain~\eqref{ho1}.  In order to obtain~\eqref{ho2}, we apply~\eqref{eq:w:oxi:h:unif:master}, which essentially asserts that the $L^\infty$ norm of $w_\pxi^{(h), I} \otimes w_\pxi^{(h), I}$ is $\lambda_\qbn^2\lambda_\qho^{-3} r_q^{-2}$.  This quantity is much less than $1$ since there is an extra $\lambda_\qho$ in the denominator, and the remainder $\lambda_\qbn^2 \lambda_\qho^{-2} r_q^{-2}$ is of order $b-1$.  Derivative estimates follow from Corollary~\ref{cor:corrections:Lp}, and we omit further details.  
\end{proof}

We now consider the oscillation error $S_{O,R}$  from $w_{q+1, R}$.
In order to define it, we first consider
\begin{align}\label{eqn:wpwp}
\div\left(\wp_{q+1,R}\otimes \wp_{q+1,R}\right)^\bullet &= \sum_{\xi,i,j,k,\vecl } \partial_\alpha \left(  a_{(\xi) } (\nabla\Phi_{(i,k)}^{-1})_{\theta}^\alpha \BB_{(\xi) }^\theta(\Phi_{(i,k)}) \,  a_{(\xi) } (\nabla\Phi_{(i,k)}^{-1})_\gamma^\bullet \BB_{(\xi) }^\gamma(\Phi_{(i,k)}) \right) \, ,
\end{align}
where $\bullet$ denotes the unspecified components of a vector field and we have used \eqref{eq:dodging:newbies} from Lemma~\ref{lem:dodging} to eliminate all cross terms. Recalling from \eqref{eq:W:xi:q+1:nn:def} that
$\BB_{(\xi)}=\chib_{(\xi)}\sum_I \etab_\xi^{I} \WW_{(\xi)}^I$
and using \eqref{eq:sa:summability}, 
item \eqref{i:bundling:3} of Proposition \ref{prop:bundling}, and
item \eqref{item:pipe:4} of Proposition \ref{prop:pipeconstruction}, 
we decompose
\begin{align}\label{eq:BB.decomp}
(\BB\otimes\BB)_{(\xi) }
&= \chib_{(\xi)}^2 \sum_{I} (\etab_{\xi}^{I })^{2} \mathbb{P}_{\neq 0}(\WW_{(\xi) }^I\otimes  \WW_{(\xi)}^I)
+ \mathbb{P}_{\neq 0} \chib_{(\xi)}^2 
\xi\otimes \xi
+ 
\xi\otimes \xi.
\end{align}
Since the vector field used to define the simple symmetric tensors in \eqref{eq:BB.decomp} does not vary in the $\xi$-direction, the symmetric tensor satisfies $\xi\cdot\nabla (\BB\otimes \BB)_{\pxi}=0$. Then using the identity $\partial_\alpha ((\nabla\Phiik^{-1})_\theta^\alpha (\BB\otimes \BB)_{\pxi}\circ \Phiik \xi^\theta) = \xi^\theta (\partial_\theta (\BB\otimes \BB)_{\pxi})\circ \Phiik =0$, \eqref{eqn:wpwp} can be expanded as
\begin{subequations}
\label{eqn:osc.expand}
\begin{align}
 \div\left(\wp_{q+1,R}\otimes \wp_{q+1,R}\right)^\bullet
&=
\sum_{\xi,i,j,k,\vecl} \partial_\alpha \left(  a_{(\xi)}^2 
(\nabla\Phi_{(i,k)}^{-1})_{\theta}^\alpha 
(\nabla\Phi_{(i,k)}^{-1})_\gamma^\bullet
(\xi^\theta \xi^\gamma) \right)
\label{eqn:osc.exp1}\\
&\quad+
\sum_{\xi,i,j,k,\vecl} B_{(\xi)}^{\bullet}
\left(\mathbb{P}_{\neq 0} \chib_{(\xi)}^2 \right) \circ \Phi_{(i,k)}
\label{eqn:osc.exp2}\\
&\quad+
\sum_{\xi,i,j,k,\vecl} B_{(\xi)}^{\bullet}
\left( \chib_{(\xi)}^2 \sum_{I} (\etab_{\xi}^{I})^{2} \mathbb{P}_{\neq 0}(\varrho_{(\xi),R}^I)^2\right) \circ \Phi_{(i,k)}
\label{eqn:osc.exp3}
\end{align}
\end{subequations}
where 
\begin{equation}\label{eq:rhoxidiamond:def}
B_{(\xi),\diamond}^{\bullet} := \xi^\theta \xi^\gamma\partial_\alpha \left(  a_{(\xi)}^2 
(\nabla\Phi_{(i,k)}^{-1})_{\theta}^\alpha 
(\nabla\Phi_{(i,k)}^{-1})_\gamma^\bullet\right)\, , \qquad \varrho_{(\xi),R}^{I}=\xi \cdot \WW_{(\xi)}^I \, .
\end{equation}

The first term \eqref{eqn:osc.exp1} cancels out $\div R_\ell$ from \eqref{new.error} up to the addition of a potential $P_{q+1}$:
\begin{align}
    &\sum_{\xi,i,j,k,\vecl} a_{(\xi)}^2 \nabla\Phi_{(i,k)}^{-1} \left(\xi\otimes \xi\right) \nabla\Phi_{(i,k)}^{-\top}
\underset{\eqref{eq:a:xi:def}}{=} \sum_{\xi,i,j,k,\vecl} e_{q+1}\Gamma_{q}^{2j} \eta_{i,j,k,\xi,\vecl}^2 \gamma_{\xi}^2\left(\frac{R_{q,i,j,k}}{e_{q+1}\Gamma_{q}^{2j}}\right)\nabla\Phi_{(i,k)}^{-1}\left(\xi\otimes\xi\right)\nabla\Phi_{(i,k)}^{-\top} \notag\\
    &\quad \underset{
\eqref{def:cumulative:current}, \eqref{eq:rqnpj}, \eqref{eq:summy:summ:1}}{=}  \sum_{i,j,k} 
    e_{q+1}
    \psi_{i,q}^2 \omega_{j,q}^2 \chi_{i,k,q}^2 \Ga_q^{2j} \Id - R_\ell
    =: P_{q+1} \Id - R_\ell\, . 
    \label{defn:P}
\end{align}
The inverse divergence of the remaining terms \eqref{eqn:osc.exp2}-\eqref{eqn:osc.exp3} will therefore form the oscillation errors. 

\begin{lemma}[\textbf{Oscillation error from $w_{q+1, R}$}]
\label{lem:oscillation:general:estimate}
There exist symmetric stresses $S_{O,R}^m$ for $m=1,\dots,q+\bn$ such that the following hold.\index{$S_O^m$}
\begin{enumerate}[(i)]
    \item\label{item:oscee:1} 
    For the potential $P_{q+1}$ defined as in \eqref{defn:P}, we have $\div \left( w_{q+1}^{(p)} \otimes w_{q+1}^{(p)} - P_{q+1} \Id+ R_\ell  \right) = \sum_{m=q+1}^{q+\bn} \div S_{O,R}^m$.
    \item For $m=q+1, \dots, q+\bn$ and $N+M\leq 2\Nind$, the errors ${S}^{m}_{O,R}$ satisfy
\begin{subequations}
\label{eq:Onpnp:est}
\begin{align}
    \left\| \psi_{i,m-1} D^N D_{t, m-1}^M {S}^{m}_{O,R} \right\|_1
    &\lesssim   \Ga_m^{-1}\delta_{m+\bn} \Lambda_{m}^N \MM{M, \Nindt, \mu_{m-1}^{-1}\Gamma_{m-1}^{i+20}, \Tau_{m-1}^{-1}\Ga_{m-1}^{10}}\label{eq:Onpnp:estimate:2} \\
\left\|\psi_{i,m-1} D^N D_{t, m-1}^M {S}^{m}_{O,R} \right\|_{\infty} 
    &\lesssim \Gamma_{m}^{\badshaq-1} \Lambda_{m}^N
    \MM{M, \Nindt, \mu_{m-1}^{-1}\Gamma_{m-1}^{i+20}, \Tau_{m-1}^{-1}\Ga_{m-1}^{10}} \, . \label{eq:Onpnp:estimate:2:new}
\end{align}   \end{subequations}
\end{enumerate}
\end{lemma}
\begin{proof}
The proof essentially follows from that of \cite[Lemma 8.1]{GKN23}. 
Applying~\cite[Proposition A.13]{GKN23}, the inverse divergence of \eqref{eqn:osc.exp2} will define $S_{O,R}^{q+1}$. Also, we decompose \eqref{eqn:osc.exp3} using the synthetic Littlewood-Paley decomposition (see~\cite[Section 4.3]{GKN23}), then the inverse divergence of each piece will generate $S_{O,R}^m$ for $m= q+\sfrac{\bn}2 +1, \cdots, q+\bn$. The undefined $S_{O,R}^m$ are set to be zero. By \eqref{eqn:osc.expand} and \eqref{defn:P}, item \eqref{item:oscee:1} holds.

To estimate $S^m_{q+1}$, we apply~\cite[Proposition A.13]{GKN23} and Corollary \ref{rem:summing:partition}, and use Lemma \ref{lem:a_master_est_p}. In our applications of these estimates, the parameter choices are essentially the same as in \cite[Lemma 8.1]{GKN23} except that here we estimate in $L^1$ instead of $L^{\sfrac32}$. With suitable modifications, we obtain \eqref{eq:Onpnp:est} for $m=q+1$:
\begin{align*}
   \left\| \psi_{i,q} D^N D_{t, q}^M {S}^{q+1}_{O,R} \right\|_1
   &\lec \Lambda_q\de_{q+\bn}\Ga_q^{56} \la_{q+1}^{-1+\alpha+N}\MM{M, \Nindt, \mu_{q}^{-1}\Gamma_{q}^{i+14}, \Tau_{q}^{-1}\Ga_q^9}\\
   &\lesssim \Gamma_{q+1}^{-9}  \delta_{q+1+\bn} \lambda_{q+1}^N \MM{M, \Nindt, \mu_{q}^{-1}\Gamma_{q}^{i+14}, \Tau_{q}^{-1}\Ga_q^9}
  \\
  \left\| \psi_{i,q} D^N D_{t, q}^M {S}^{q+1}_{O,R} \right\|_\infty
  &\lec \Ga_{q+1}^{\badshaq -9}  \la_{q+1}^{N}\MM{M, \Nindt, \mu_{q}^{-1}\Gamma_{q}^{i+14}, \Tau_{q}^{-1}\Ga_q^9}
\end{align*}
for $N, M\leq 2\Nind$,
where $\alpha$ is a constant arbitrarily close to $0$ (see the choice of $\alpha$ in item~\ref{item:choice:of:alpha}). 

For $m=q+\sfrac{\bn}2 +1, \cdots, q+\bn$, $S_{O,R}^{m}$ have a decomposition $S_{O,R}^{m} = S_{O,R}^{m,l} + S_{O,R}^{m,*}$ such that the following holds. In the case of $m= q+\sfrac{\bn}2+1$, we have  
\begin{align*}
   \left\| \psi_{i,q} D^N D_{t, q}^M {S}^{q+\sfrac{\bn}2+1,l}_{O,R} \right\|_1
   &\lec \Lambda_q \de_{q+\bn}\Ga_q^{55} 
\la_{q+\sfrac{\bn}2 +1}^{-1 + N}\MM{M, \Nindt, \mu_{q}^{-1}\Gamma_{q}^{i+14}, \Tau_{q}^{-1}\Ga_q^9}
  \\
  \left\| \psi_{i,q} D^N D_{t, q}^M {S}^{q+\sfrac{\bn}2+1,l}_{O,R} \right\|_\infty
  &\lec \Ga_{q+\sfrac{\bn}2+1}^{\badshaq -9}  \la_{q+\sfrac{\bn}2+1}^{N}\MM{M, \Nindt, \mu_{q}^{-1}\Gamma_{q}^{i+14}, \Tau_{q}^{-1}\Ga_q^9}\, ,
\end{align*}
and otherwise 
\begin{align*}
    \left\| \psi_{i,q} D^N D_{t, q}^M {S}^{m,l}_{O,R} \right\|_1
    &\lec \Lambda_q \de_{q+\bn}\Ga_q^{56} 
     \la_{m-1}^{-2} \la_m^{N+1}\MM{M, \Nindt, \mu_{q}^{-1}\Gamma_{q}^{i+14}, \Tau_{q}^{-1}\Ga_q^9}\\
     \left\| \psi_{i,q} D^N D_{t, q}^M {S}^{m,l}_{O,R} \right\|_\infty
     &\lec \Ga_m^{\badshaq-9}\la_m^N\MM{M, \Nindt, \mu_{q}^{-1}\Gamma_{q}^{i+14}, \Tau_{q}^{-1}\Ga_q^9}\,  
\end{align*}
for $N, M\leq 2\Nind$. Here the $L^\infty$-norm estimates follows from \eqref{eq:par:div:2}. Furthermore, $S_{O,R}^{m,l}$ and $S_{O,R}^{m,*}$ satisfy that for  $m= q+\sfrac{\bn}2 +1, \cdots, q+\bn$, $q+1\leq q'\leq m-1$, and $N,M \leq 2\Nind$,
\begin{align*}
    &B(\supp\hat w_{q'}, \la_{q'}^{-1}\Ga_{q'+1}) \cap \supp S_{O}^{m,l} = \emptyset \, , \qquad \left\| D^N D_{t, q}^M {S}^{m,*}_{O,R} \right\|_\infty
    \leq \de_{q+3\bn} \Tau_{q+\bn}^{4\Nindt}\la_m^N \mu_q^{-M} \, , 
    \end{align*}
where the latter inequality follows from item~\ref{i:par:9.5}. 
Combining these outcomes with Lemma~\ref{lem:upgrading.material.derivative} and using \eqref{par:osc0} and \eqref{par:osc}, we obtain the desired estimates \eqref{eq:Onpnp:est} for the remaining cases of $m$.
\end{proof}

\subsection{Transport and Nash errors}\label{ss:ER:TN}

\begin{lemma}[\bf Transport/Nash error]\label{l:transport:error}
There exist symmetric stresses $S_{TN}$ such that\index{$S_{TN}$}
\begin{enumerate}[(i)]
\item\label{item1:tran.err}
$\div S_{TN}=
(\pa_t + \hat u_q \cdot \na) w_{q+1} + w_{q+1}\cdot\nabla \hat u_q$, 
\item\label{item2:tran.err} and for all $N+M\leq 2\Nind$, the error $S_{TN}$ satisfies
\begin{equation}\label{eq:trans:est}
\begin{aligned}
\left\| \psi_{i,q} D^N D_{t,q+\bn-1}^M S_{TN} \right\|_{1} &\les 
 \Ga_{q+\bn}^{-1}\delta_{q+2\bn} \Lambda_{q+\bn}^N \MM{M,\Nindt, \mu_{q+\bn-1}^{-1}\Gamma_{q+\bn-1}^{i+20},\Tau_{q+\bn-1}^{-1}\Ga_{q+\bn-1}^{10}} \\
\left\| \psi_{i,q} D^N D_{t,q+\bn-1}^M S_{TN} \right\|_{\infty}  &\les \Gamma_{q+\bn}^{\badshaq-1} \Lambda_{q+\bn}^N \MM{M,\Nindt, \mu_{q+\bn-1}^{-1}\Gamma_{q+\bn-1}^{i+20},\Tau_{q+\bn-1}^{-1}\Ga_{q+\bn-1}^{10}} \, .
\end{aligned}
\end{equation}
\end{enumerate}
\end{lemma}

\begin{proof}
The proof follows from the same argument as in \cite[Lemma 8.6]{GKN23}, but incorporates the helicity corrector instead of the current corrector. Since $D_{t,q} w_{q+1}$ and $w_{q+1}\cdot \na \hat u_q$ are mean-zero, we apply~\cite[Proposition A.13]{GKN23} to define $S_{TN}$, verifying \eqref{item1:tran.err}. Furthermore, by \cite[Proposition A.13]{GKN23}, Corollary \ref{rem:summing:partition}, Lemma \ref{lem:a_master_est_p}, \eqref{par:trans}, \eqref{eq:par:div:2} and 
item~\ref{i:par:9.5}, $S_{TN}= S_{TN}^l + S_{TN}^*$, where $S_{TN}^l$ and $S_{TN}^*$ satisfy that for any $N, M \leq 2\Nind$, 
\begin{align*}
    \left\| \psi_{i,q} D^N \Dtq^M S_{TN}^l \right\|_{1}
&\lec \Ga_q^{\CLebesgue+53} \de_{q+\bn}^{\frac12}r_q \mu_q^{-1}\la_{q+\bn}^{-1+N}\MM{M,\Nindt, \mu_q^{-1}\Gamma_{q}^{i+15},\Tau_q^{-1}\Ga_q^9}\\
&\les \Gamma_{q+\bn}^{-100} \delta_{q+2\bn} \lambda_{q+\bn}^N \MM{M,\Nindt, \mu_q^{-1}\Gamma_{q}^{i + 15},\Tau_q^{-1}\Ga_q^9}\\
\left\| \psi_{i,q} D^N \Dtq^M S_{TN}^l \right\|_{\infty}
&\lec \Ga_{q+\bn}^{\badshaq-200}\la_{q+\bn}^N\MM{M,\Nindt, \mu_q^{-1}\Gamma_{q}^{i+15},\Tau_q^{-1}\Ga_q^9}\\ 
\left\|  D^N \Dtq^M S_{TN}^* \right\|_{\infty}
&\lec \de_{q+3\bn}^2 \Tau_{q+\bn}^{4\Nindt}
\la_{q+\bn}^{N}\mu_q^{-M}\, , 
\end{align*}
and for any $q+1\leq q'\leq q+\bn-1$, $
B(\supp \hat w_{q'},  \la_{q'}^{-1}\Ga_{q'+1}) \cap \supp S_{TN}^l = \emptyset$. Here, the error from helicity corrector does not change the upper bound, since $\la_{q+\bn}\la_\qho^{-3/2}\leq \de_{q+\bn}^{1/2}$ (see  \eqref{eq:qho:choice3}). The estimates in \eqref{item2:tran.err} then follow from Lemma~\ref{lem:upgrading.material.derivative}.
\end{proof}

\subsection{Divergence corrector error}
Recalling \eqref{new.error}, we write $S_C = S_{C1} + S_{C2}$,\index{$S_C$} where
\begin{equation}
\div S_{C1} = \div \left( w_{q+1,R}^{(p)}\otimes_s w_{q+1,R}^{(c)}\right) \, , \qquad S_{C2} = w_{q+1,R}^{(c)} \otimes w_{q+1,R}^{(c)} \, . \label{eq:div:cor:expand}
\end{equation}

\begin{lemma}[\bf Divergence corrector error from $w_{q+1,R}$]\label{l:divergence:corrector:error}
There exist symmetric stresses $S_C^{m}$ for $m= q+\sfrac{\bn}{2}+1, \cdots , q+\bn$ such that the following hold.
\begin{enumerate}[(i)]
    \item $\div \left( w_{q+1,R}^{(p)}\otimes_s w_{q+1,R}^{(c)} +  w_{q+1}^{(c)} \otimes w_{q+1}^{(c)} \right) = \sum_{m=q+\lfloor\sfrac{\bn}{2}\rfloor+1}^{q+\bn} \div S_C^{m}$. 
\item For the same range of $m$ and for all $N+M\leq 2\Nind$, the errors satisfy
\begin{subequations}\label{eq:div:corrector}
\begin{align}
&\left\| \psi_{i,m-1} D^N D_{t,m-1}^M S_C^{m}
\right\|_{1} \les \Gamma_m^{-1} \delta_{m+\bn}\Lambda_{m}^N \MM{M,\Nindt, \mu_{m-1}^{-1}\Gamma_{m-1}^{i + 20},\Tau_{m-1}^{-1}\Ga_{m-1}^{10}}
\label{eq:div:corrector:L1}
\\
&\left\| \psi_{i,m-1} D^N D_{t,m-1}^M S_C^{m} \right\|_{\infty} \les \Gamma_m^{-1}  \Lambda_{m}^N \MM{M,\Nindt, \mu_{m-1}^{-1}\Gamma_{m-1}^{i + 20},\Tau_{m-1}^{-1}\Ga_{m-1}^{10}} \, . \label{eq:div:corrector:Linfty}
\end{align}
\end{subequations}
\end{enumerate}
\end{lemma}
\begin{proof}
We first define $S_C^m$ as in \texttt{Step 2} of the proof of \cite[Lemma 8.10]{GKN23}; For $S_{C1}$, we expand $\div \left( w_{q+1,R}^{(p)}\otimes_s w_{q+1,R}^{(c)}\right)$,
decompose them using the synthetic Littlewood-Paley decomposition, and take the inverse divergence applying~\cite[Proposition A.13]{GKN23} to define $S_{C1}^m$. For $S_{C2}$, on the other hand, the term $w_{q+1,R}^{(c)}\otimes w_{q+1,R}^{(c)}$ directly goes into $S_{C2}^{q+\bn}$. Then, we set $S_{C}^m := S_{C1}^m + S_{C2}^m$, where all undefined $S_{C1}^m$ and $S_{C2}^m$ are identically zero. Then, the estimates for $S_C^m$ are obtained essentially by the same argument as in the proof of \cite[Lemma 8.10]{GKN23} and Lemma~\ref{lem:upgrading.material.derivative} with suitable modifications to replace $L^{\sfrac32}$-norm estimate by $L^1$-norm estimate. 
\end{proof}

\subsection{Mollification errors}\label{ss:essemm}

Recalling from subsection \ref{sec:error.ER} that $\div S_{M2}$ has mean-zero, we use~\cite[Proposition A.13]{GKN23},~\cite[Remank A.15]{GKN23} to first define the mollification error $S_M = S_{M1} + S_{M2}$ by
\begin{align}
    S_{M1}&:=R_q^q - \sfrac 13 (\tr R_q^q)\Id - R_\ell
   \label{ER:new:error:moll1} \\
    S_{M2}&:= \divR  \left[(\pa_t + \hat u_{q}\cdot \na )(\hat w_{q+\bn}-w_{q+1})\right]
    + (\hat w_{q+\bn}-w_{q+1})\otimes \hat u_q  + \hat w_{q+\bn}\otimes \hat w_{q+\bn} - w_{q+1}\otimes w_{q+1}
    \, . \notag
\end{align}

\begin{lemma}[\bf Mollification error $S_M$]\label{lem:moll:ER} For all $N+ M\leq 2\Nind$,
$S_{M}$ satisfies \index{$S_M$}
\begin{subequations}\label{est.stress.moll}
    \begin{align}
        \norm{D^N D_{t,q}^M S_{M1}}_\infty
        &\leq \Ga_{q+1}^9 \delta_{q+3\bn} \Tau_{q+1}^{2\Nindt} \left(\lambda_{q+1}\Gamma_{q+1}\right)^N \MM{M, \NindRt, \mu_{q}^{-1}, \Tau_{q}^{-1} } \, ,
    \label{est:stress.mollification1}\\
        \norm{ D^N D_{t,q+\bn-1}^M S_{M2}}_\infty
        &\leq \Ga_\qbn^9 \delta_{q+3\bn}\Tau_\qbn^{2\Nindt} \left(\lambda_{q+\bn}\Gamma_{q+\bn}\right)^N \MM{M, \NindRt, \tau_{q+\bn-1}^{-1}, \Tau_{q+\bn-1}^{-1} } \, . 
    \label{est:stress.mollification}
    \end{align}
\end{subequations}
\end{lemma}
\begin{proof}[Proof of Lemma~\ref{lem:moll:ER}]
From item~\ref{item:moll:four} from Lemma~\ref{lem:upgrading}, we have that for all $N+M\leq 2\Nind$,
\begin{align*}
     \left\| D^N \Dtq^M S_{M} \right\|_{\infty} 
    \les \Gamma_{q+1} \Tau_{q+1}^{2\Nindt} \delta_{q+3\bn}^2 \lambda_{q+1}^N \MM{M,\Nindt, \mu_{q}^{-1},\Gamma_{q}^{-1}\Tau_{q}^{-1}} \, ,
\end{align*}
which immediately leads to \eqref{est:stress.mollification1}. To deal with $S_{M2}$, we recall from \eqref{eq:diff:moll:vellie:statement} that
    \begin{align*}
    \norm{D^N D_{t,q+\bn-1}^M \left(w_{q+1}- \hat w_{q+\bn}\right)}_\infty \lec \delta_{q+3\bn}^3 \Tau_\qbn^{25\NindRt} \left(\lambda_{q+\bn}\Gamma_{q+\bn-1}\right)^N \MM{M, \NindRt, \tau_{q+\bn-1}^{-1}, \Tau_{q+\bn-1}^{-1} }  \, .
\end{align*}
for all $N+M\leq \sfrac{\Nfin}{4}$. Using Lemma \ref{lem:dodging}, we note that $D_{t,q-\bn-1}w_{q+1} = D_{t,q}w_{q+1}$ and $D_{t,q-\bn-1}\hat w_{q+\bn} = D_{t,q}\hat w_{q+\bn}$. Then, writing $\hat w_{q+\bn}\otimes \hat w_{q+\bn} - w_{q+1}\otimes w_{q+1}
=(\hat w_{q+\bn}-w_{q+1})\otimes \hat w_{q+\bn} + w_{q+1}\otimes (\hat w_{q+\bn}-w_{q+1}) $ and using \eqref{eq:vellie:upgraded:statement} and \eqref{eq:diff:moll:vellie:statement},
we have that for all $N+M\leq 2\Nind$,
\begin{align}
    &\norm{\psi_{i, q+\bn-1}
    D^N D_{t,q+\bn-1}^M
    [\hat w_{q+\bn}\otimes \hat w_{q+\bn} - w_{q+1}\otimes w_{q+1}]
    }_\infty\notag\\
    &\qquad\leq \delta_{q+3\bn} \Tau_\qbn^{2\NindRt} \left(\lambda_{q+\bn}\Gamma_{q+\bn}\right)^N \MM{M, \NindRt, \tau_{q+\bn-1}^{-1}, \Tau_{q+\bn-1}^{-1} }\, .
    \label{est:SM21.proof}
\end{align}

As for the remaining term, we first upgrade the material derivative in the estimate for $\hat u_q$.
Applying Lemma \ref{lem:upgrading.material.derivative} to $F^l = 0$, $F^* = \hat u_q$, $k=q+\bn$, $N_\star = \sfrac{3\Nfin}{4}$ with \eqref{eq:bob:Dq':old}, we get
\begin{align*}
    \norm{ D^N D_{t,q+\bn-1}^M \hat u_q}_\infty
    \lec \Tau_{q}^{-1} \lambda_{q+\bn}^N \Tau_{q+\bn-1}^{-M}
\end{align*}
Here, we used \eqref{v:global:par:ineq}. Then, we use~\cite[Remank A.15]{GKN23} with \eqref{v:global:par:ineq}, setting 
\begin{align*}
    G &=D_{t,q+\bn-1} (\hat w_{q+\bn} - w_{q+1})
    \quad (\text{or } G= (\hat w_{q+\bn} - w_{q+1}) \otimes \hat u_q)
    , \quad v = \hat u_{q+\bn-1}
\\
    \const_{G, \infty}&= \de_{q+\bn}^3 \Tau_{q+\bn}^{20\Nindt}, \quad
    \la=\la' = \la_{q+\bn}\Ga_{q+\bn-1}, \quad 
    M_t = \Nindt,\quad
    \nu =\nu' = \Tau_{q+\bn}^{-1}, \quad \const_v = \La_{q+\bn-1}^{\sfrac12} \\
    &\hspace{2cm}N_* = \sfrac{\Nfin}{9}, \quad
M_* = \sfrac{\Nfin}{10}, \quad
N_\circ = M_\circ = 2\Nind\,  .
\end{align*}
As a result, with a suitable choice of positive integer $K_\circ$ so that $\de_{q+\bn}^3 \Tau_{q+\bn}^{20\Nindt }\la_{q+\bn}^{5} 2^{2\Nind}
    \leq \la_{q+\bn}^{-K_\circ}
    \leq \de_{q+3\bn} \Tau_\qbn^{10\NindRt}$, we deduce that
\begin{align}\label{est:SM22.proof}
    \norm{D^N D_{t,q+\bn-1}^M\divR ( D_{t,q} (\hat w_{q+\bn} - w_{q+1})) }_\infty
    &=
    \norm{D^N D_{t,q+\bn-1}^M\divR ( D_{t,q+\bn-1} (\hat w_{q+\bn} - w_{q+1})) }_\infty  \\
    &\lec \de_{q+3\bn} \Tau_\qbn^{10\Nindt} (\la_{q+\bn}\Ga_\qbn)^N \Tau_{q+\bn}^{-M} \notag \\
    &\leq \de_{q+3\bn}\Tau_\qbn^{2\Nindt}(\la_{q+\bn}\Ga_\qbn)^N \MM{M, \Nindt,\tau_{q+\bn}^{-1} ,\Tau_{q+\bn}^{-1}}\, ,   \notag 
\end{align}
for all $N+M\leq 2\Nind$. This completes the proof of \eqref{est:stress.mollification}.
\end{proof}

\section{Parameters}\label{sec:pam}

In this section, we specify parameters and list some resulting inequalities, none of which depend on $q$.  We note that we choose $a$ last, for which $\la_q$ and $\de_q$ are defined as in \eqref{eq:def:la:de}; as a consequence, $a^{(b^q)} \leq \la_q \leq 2a^{(b^q)}$ and $\frac{1}{3} \la_q^b \leq \la_{q+1} \leq 2\la_q^b$.
\begin{enumerate}[(i)]
\item\label{i:par:0} Choose $\ovb$ as in item~\eqref{ip1}.   

\item\label{i:par:1} Choose $\beta$ as in item~\eqref{ip2}. The upper bound $\be<\sfrac 1{2\ovb}$ is needed in the oscillation error estimates and guarantees
\begin{align}\label{par:osc:pre3}
    \la_q \de_{q+\bn} (\la_m \de_{m+\bn})^{-1}\ll 1 
\end{align}
for any integer $m>q$. The lower bound $\beta > \frac{2+\ovb-2\ovb^2 }{1+3\ovb^2 - 2\ovb^4 }$ comes from the transport error estimate, where we need 
\begin{align}\label{par:trans:pre}
    \la_q^2 \de_q^{\frac12} \de_{q+\bn}^{\frac32} {\la_{q+\frac{\bn}2}} {\la_{q+\bn}^{-2}} \ll \de_{q+2 \bn}\, . 
\end{align}
Lastly, we imposed the other lower bound $\be> \frac{2-\ovb}{2\ovb}$ to guarantee 
\begin{align}
    \delta_\qbn \lambda_{q+\half}^2\lambda_\qbn^{-1} \ll 1\,.
    \label{lb1}
\end{align}
This is required for the estimate of the helicity generated by Euler-Reynolds principal-corrector.
\item\label{i:par:1.5} Choose an even integer $\bn\gg 1$ as in \eqref{ip3}, so that $\ovb^2<2$ and \eqref{par:osc:pre} implies that
\begin{align}\label{par:osc:pre1}
    \la_q \de_{q+\bn}\la_{m-1}^{-2}\la_m \ll \delta_{m+\bn}
\end{align}
for $m=q + \sfrac{\bn}2 +1, \cdots, q+\bn$, which is required for the oscillation error estimates.

\item\label{i:par:2.5} 
Consequently, for $q'\geq q-\sfrac \bn 2+1$, $\la_{q'+\half}\la_{q+\half}\left(\la_q\la_{q'+\bn}\right)^{-1} < 1$.
\item\label{i:par:3} Choose $\CLebesgue=\frac{6+b}{b-1}$.
\item\label{i:par:4} Define $\Ga_q$, $r_q$, $\tau_q$, $\mu_q$, $\La_q$ and $\la_\qho$ as in~\eqref{eq:deffy:of:gamma}-\eqref{def:thursday},
and $0<\varepsilon_\Gamma\ll (b-1)^2<1$ so that
\begin{subequations}
    \begin{align}
\de_{q+\bn}^{\sfrac12}r_q \mu_q^{-1}\la_{q+\bn}^{-1}\Gamma_{q+\bn}^{1000} &\underset{\eqref{par:trans:pre}} < \de_{q+2\bn}
\label{par:trans}
\\   
    \la_q \de_{q+\bn} (\la_m \de_{m+\bn})^{-1}\Ga_{q+\bn}^{1000}&\underset{\eqref{par:osc:pre3}}{<} 1 \quad \text{for } m>q
    \label{par:osc0}
\\
    \la_q \de_{q+\bn}\la_{m-1}^{-2}\la_m  \Ga_{q+\bn}^{1000}&\underset{\eqref{par:osc:pre1}}{<}
  \delta_{m+\bn} \quad \text{for } m= q+\sfrac{\bn}2 +1, \cdots, q+\bn
\label{par:osc} \\
    \la_q\Ga_q &\leq \frac{\la_\qho^2}{\la_{q+\bn}}
\label{eq:qho:choice1}\\
  \frac{\la_{q+\bn}}{\la_\qho}&\leq \la_q\de_{q+\bn}  \, . \label{eq:qho:choice2}\\
    \frac{\la_{q+\bn}^3}{\la_\qho^3} = \Ga_{q-100}^{\frac{3}{100}} &\leq
    \Ga_{q-1}^{\frac12}\leq
    \frac{\la_{q+\bn}\de_{q+2\bn}}{\Gamma_q^{29}} \, . \label{eq:qho:choice3}\\
\frac{\lambda_{q+\half}^2}{\lambda_q\lambda_\qbn} \Gamma_\qbn^{5000} &< 1 \label{super} \\
    \Gamma_\qbn \delta_{\qbn-1}^{-\sfrac 12} r_{q-1}^{-1} &\leq \delta_{\qbn}^{-\sfrac 12} r_q^{-1} \label{ineq.eps.gamma.imax} \\
    \Gamma_q^{25}\lambda_q \delta_q^{\sfrac 12} \leq \mu_q^{-1}\, ,   
    &\qquad \Gamma_\qbn^{300} \mu_{\qbn-1}^{-1} \leq \mu_\qbn^{-1} \label{ineq:tau:q} \\
    \Gamma_{q+1}^{500} &\leq \lambda_q^{1-2\beta b^{\bn}} \label{eq:eg:timescale} \\
\frac{\la_{q+\frac{\bn}2}^2\de_{q+\bn}}{\la_{q+\bn}} \Ga_q^{10}
&\underset{\eqref{lb1}}{<} \Ga_{q+1}^{-1000} 
\label{par:helcc}
\end{align}
   \end{subequations}

      \item\label{i:par:5} Choose $\badshaq\geq 10$ so that for all $\sfrac \bn 2 \leq k \leq \bn$ (see \cite[Section 11]{GKN23} for the details)
\begin{align}
\Ga_q^{\badshaq} \la_q^2 \la_{q+k}^4 \la_{q+\half}^{-4} \la_{q+k}^{2} \la_{q+k-1}^{-4} &< \Ga_{q+\half}^{\badshaq}\notag \\
\label{eq:par:div:2}
\Gamma_q^{\badshaq} \leq \Gamma_{q+\half}^{\badshaq} \Ga_\qbn^{-2000} \, , \qquad \Gamma_q^{\badshaq+500} \Lambda_q &\left( \frac{\la_{q+k}}{\la_{q+\half}} \right)^2 \lambda_{q+k-1}^{-2}\la_{q+k} \leq \Gamma_{q+\half}^{\badshaq} \Ga_\qbn^{-200}  \, . 
    \end{align}
     \item\label{item:choice:of:alpha} Choose $\alpha\in (0,1)$ such that $\lambda_{q+\bn}^\alpha \leq \Gamma_q^{\sfrac 1{10}}$.
\item\label{i:par:tau} Choose $\Tau_q$ to be {no larger than}
    \begin{align}
    \frac12 \Tau_{q-1}^{-1} \geq  
\mu_q^{-1}\Gamma_q^{\badshaq+100}\delta_q^{-\sfrac 12}r_q^{-1} \Lambda_q^3 \, , \label{v:global:par:ineq}
    \end{align}
    and so that all mollification estimates which require a small choice of $\Tau_q$ hold.
\item\label{i:par:6} Choose $\NcutSmall
    $ and $\NcutLarge$ such that
    \begin{subequations}\label{condi.Ncut0}
    \begin{align}
     \NcutSmall &\leq \NcutLarge \, , \label{condi.Ncut0.1} \\
     \la_\qbn^{200}\left(\frac{\Ga_{q-1}}{\Ga_q}\right)^{\frac{\NcutSmall}5}
    &\leq \min\left(\la_{q+\bn}^{-4} \de_{q+3\bn}^2, \Ga_{q+\bn}^{-100\badshaq-1700-100\CLebesgue} \delta_{q+3\bn}^{100}r_q^{100} \right) \, , \label{condi.Ncut0.2} \\
    \delta_{q+\bn}^{-\sfrac 12} r_q^{-10} \Ga_{q+\bn}^{\sfrac{\badshaq}{2}+16+\CLebesgue} \left( \frac{\Ga_{q+\bn-1}}{\Ga_{q+\bn}}\right)^{\NcutLarge} &\leq \Ga_{q+\bn}^{-1} \lambda_\qbn^{-10} \, . \label{condi.Ncut0.3}
    \end{align}
    \end{subequations}
    \item\label{i:par:7} Choose $\Nindt$ such that
    \begin{align}\label{condi.Nindt}
        \Nindt \geq \NcutSmall, \quad       
        \Ga_q^{-\Nindt} (\tau_q^{-1}\Ga_q^{i+40})^{-\NcutSmall-1} 
        (\Tau_q^{-1}\Ga_q)^{\NcutSmall+1}\leq 1 \, .
    \end{align}
    \item\label{i:par:12} Choose $N_{\rm g}, N_{\rm c}$ so that all mollification estimates hold, and
\begin{subequations}\label{eq:darnit}
\begin{align}
    \Gamma_{q-1}^{-N_{\rm g}} \Gamma_q^2 &\leq \Gamma_{q+1} \Tau_{q+1}^{50\Nindt} \delta_{q+3\bn}^3 \, , \label{eq:darnit:2} \\
    2(\Tau_{q+\bn-1}^{-1}{\Ga_{q+\bn-1}^{10}})^{5\Nindt} \Gamma_{q+\bn}^{2\badshaq+\CLebesgue+100} r_q^{-2} \Gamma_{q-1}^{-\sfrac{N_{\rm c}}{2}} &\leq \Gamma_{q+\bn}^{-N_{\rm g}}
    \delta_{q+3\bn}^3\tau_{q+\bn-1}^{50\Nindt} \, , \label{eq:darnit:3} \\
    N_{\rm g} &\leq N_{\rm c} \leq \frac{\Nind}{40} \, .\label{eq:darnit:1}
\end{align}
\end{subequations}
\item\label{i:par:8} Choose $\Nind$ such that \eqref{eq:darnit:1} is satisfied and\index{$\Nind$}
    \begin{subequations}
    \begin{align}
    \Nindt &\leq \Nind \, , \qquad 
    \left(\Ga_{q-1}^{\Nind} \Ga_q^{-\Nind}\right)^{\sfrac{1}{10}} \leq \delta_{q+5\bn}^3 \Gamma_q^{-2\badshaq-3} r_q \, , \qquad 
    \Ga_q^{-\Nind+24}\la_{q+\bn}^{30}\leq 1  \, . \label{eq:Nind:3}
    \end{align}
    \end{subequations}
    \item\label{i:par:9} Choose $\Ndec$ such that\index{$\Ndec$} $(\la_{q+\bn+2}\Ga_q)^4 \leq \left(  \frac{\Ga_q^{\sfrac{1}{10}}}{{4\pi}}
        \right)^\Ndec$ and $\Nind \leq \Ndec$.
\item\label{i:par:9.5} Choose $K_\circ$ large enough so that\index{$K_\circ$} $\lambda_q^{-K_\circ} \leq \delta_{q+3\bn}^{3} \Tau_{q+\bn}^{{5\Nind}} \lambda_{q+\bn+2}^{-100}$.
\item\label{i:par:10} Choose $\dpot$ and $N_{**}$ such that\index{$\dpot$}\index{$N_{**}$}
\begin{subequations}
    \begin{align}
    \lambda_\qbn^{100}  \Gamma_q^{-\dpot} & \leq \lambda_\qbn^{-1} \label{ineq:dpot:hel}\\
     2\dpot + 3 &\leq N_{**} \, , \label{ineq:Nstarstar:dpot} \\
     \lambda_{q+\bn}^{100} \Gamma_q^{-\sfrac{\dpot}{200}} \Lambda_{q+\bn+2}^{5+K_\circ} \left( 1 + \frac{\max(\la_\qbn^2\Tau_q^{-1},\Lambda_q^{\sfrac 12}\Lambda_{q+\bn})}{\tau_q^{-1}} \right)^{20\Nind} &\leq \Tau_\qbn^{200\Nindt} \, , \label{ineq:dpot:1} \\
     \lambda_{q+\bn}^{100} \Gamma_q^{-\sfrac{N_{**}}{20}} \Lambda_{q+\bn+2}^{5+K_\circ} \left( 1 + \frac{\max(\la_\qbn^2\Tau_q^{-1},\Lambda_q^{\sfrac 12}\Lambda_{q+\bn})}{\tau_q^{-1}} \right)^{20\Nind} &\leq \Tau_\qbn^{20\Nindt} \, . \label{ineq:Nstarz:1}
\end{align}
\end{subequations}
\item\label{i:par:11} Choose $\Nfin$ such that\index{$\Nfin$}
\begin{subequations}
    \begin{align}
        2\Ndec + 4 + 10\Nind &\leq \sfrac{\Nfin}{40000} - \dpot^2 - 10\NcutLarge - 10\NcutSmall - N_{**} - 300 \, . \label{condi.Nfin0}
        \end{align}
    \end{subequations}
\item\label{i:choice:of:a} Having chosen all the parameters mentioned in items~\eqref{i:par:1}--\eqref{i:par:11} except for $a$, there exists a sufficiently large parameter $a_*$ such that $a_*^{(b-1)\varepsilon_\Gamma b^{-2\bn}}$ is at least fives times larger than \emph{all} the implicit constants throughout the paper, as well as those which have been suppressed in the computations in this section. Choose $a$ to be any natural number larger than $a_*$.\index{$a_*$}
\end{enumerate}

\section{Appendix}\label{sec:app}
In this section, we collect a number of useful tools from \cite{BMNV21} and \cite{GKN23}.

\subsection{Upgrading material derivatives}
We now recall a version of \cite[Lemma~A.23]{GKN23}, which upgrades material derivatives for functions decomposed into a ``localized'' piece with support assumptions and a ``nonlocal'' remainder. The only difference is the substitution of $\tau$'a for $\mu$'s.
\begin{lemma}[\bf Upgrading material derivatives]\label{lem:upgrading.material.derivative}
Fix $p\in [1, \infty]$ and a positive integer $N_\star\leq \sfrac{3\Nfin}4$. Assume that a tensor $F$ is given with a decomposition $F = F^l + F^*$ such that
\begin{subequations}
\begin{align} 
\left\| \psi_{i,q'} D^N D_{t,q'}^M F^l \right\|_{p}
&\lesssim 
\const_{p,F}  \lambda_{F}^N  \MM{M,\Nindt,\Gamma_{q'}^{i+c}  {\mu_{q'}^{-1}}, \Gamma_{q'}^{-1}\Tau_{q'}^{-1} }
\label{eq:before.upgraded}\\
\left\| D^N D_{t,q'}^M F^* \right\|_{\infty} 
&\lesssim \const_{*,F}\Tau_{q+\bn}^{\Nindt}\lambda_{F}^N {\mu_{q'}}^{-M} 
\label{eq:before.upgraded.uniform}
\end{align}
\end{subequations}
for all $M+N\leq N_{\star}$, an absolute constant {$c\leq 20$}, and constants $\const_{p,F}$ and $\const_{*,F}$. Assume furthermore that there exists $k$ such that $q'+1<k\leq q'+\bn$ and
\begin{align}\label{supp.F}
    \supp(\hat w_{q''}, \la_{q''}^{-1}{\Ga_{q''}}) \cap \supp(F^l) =\emptyset \quad \forall q'+1\leq q'' < k \, .
\end{align}
Finally, assume that $\la_F \Ga_{q'+\bn}^{\imax+2}\de_{q'+\bn}^{\sfrac12} \leq \Tau_{q'+\bn}^{-1}$. Then $F$ obeys the following estimate with an upgraded material derivative for all $M+N\leq N_{\star}$;
\begin{align}
&\left\| \psi_{i,k-1} D^N D_{t,k-1}^M F \right\|_{p} \lesssim \
(\const_{p,F} + \const_{*,F}) \max(\lambda_{F}, \La_{k-1})^N
  \MM{M,\Nindt,\Gamma_{k-1}^{i}  {\mu}_{k-1}^{-1}, \Gamma_{k-1}^{-1}\Tau_{k-1}^{-1} } \, . \notag 
\end{align}
In particular, the nonlocal part $F^*$ obeys for $N+M\leq N_\star$ the better estimate
\begin{align}\label{est:nonlocal:upgrade}
\norm{D^N D_{t,k-1}^M F^*}_\infty
    \lec \const_{*,F}\max(\lambda_{F}, \la_{k-1}\Ga_{k-1})^N 
    \MM{M, \Nindt,  {\mu}_{k-1}^{-1},  \Tau_{k-1}^{-1}\Gamma_{k-1}^{-1}} \, .
\end{align}
\end{lemma}

\subsection{Velocity cutoffs}\label{sec:inductive.cutoffs}

We must define the velocity cutoffs in order to verify the assumption in subsection~\ref{sec:cutoff:inductive} and~\ref{sec:inductive:secondary:velocity}. The proofs are quite similar to those in~\cite{GKN23, BMNV21, NV22}, and so many details are omitted.  The main differences are cosmetic; most of the $\tau$'s have been replaced with $\mu$'s.  

\subsubsection{Mollified velocity increment and velocity cutoff functions}
\label{sec:cutoff:velocity:properties}

We recall the definition of $\hat w_{q+\bn}$ from \eqref{def.w.mollified}; $\hat w_{q+\bn} = \mathcal{\tilde P}_{q+\bn,x,t} w_{q+1}$. Next, we shall need the following.

\begin{definition}[\bf Translating $\Gamma$'s between $q'-1$ and $q'$]\label{def:istar:j}
Given $i,j,q' \geq 0$, we define $i_* = i_*(j,q') = i_*(j) = \min\{ i \geq 0 \colon \Gamma_{q'}^{i} \geq \Gamma_{q'-1}^{j} \}$ and $j_*(i,q') = \max\{ j: i_*(j) \leq i \}$. A consequence of these definition is the inequality $\Gamma_{q'}^{i-1} < \Gamma_{q'-1}^{j_*(i,q)} \leq \Gamma_{q'}^i$. We also note that for $j=0$, $i_*(j)=0$.  Finally, the upper bound for $i_*(j)$ depends on $j$ but not $q$.
\end{definition}

We now define the velocity cutoff functions using the cutoff functions from~\cite[Lemma~5.5]{GKN23} and~\cite[Lemma~6.2]{BMNV21}, analogously to~\cite[Definition~9.2]{GKN23}.

\begin{definition}[\bf Intermediate cutoffs]\label{def:intermediate:cutoffs}
For $\qbn \geq 1$, $m\leq\NcutSmall$, and $j_m\geq 0$, we define
\begin{align}
    h_{m,j_m,q+\bn}^2(x,t) &= \Gamma_{q+\bn}^{-2i_*(j_m)} \delta_{q+\bn}^{-1}\left( {\mu_{q+\bn-1}^{-1}}\Gamma_{q+\bn}^{i_*(j_m)+2}\right)^{-2m} \sum_{N=0}^{ \NcutLarge} \left( \lambda_{q+\bn} \Gamma_{q+\bn} \right)^{-2N} \left| D^N D_{t,q+\bn-1}^m \hat w_{q+\bn} \right|^2 
    \, . 
\label{eq:h:j:q:def}
\end{align}
\noindent We then define $\psi_{m,i_m,j_m,q+\bn}$ by
\begin{subequations}
\begin{align}
 \psi_{m,i_m,j_m,q+\bn}(x,t) 
 &= \gamma_{m,q+\bn}^3 \left( \Gamma_{q+\bn}^{-2(i_m-i_*(j_m))(m+1)} h_{m,j_m,q+\bn}^2  (x,t) \right) \, \qquad i_m> i_*(j_m) \, , 
\label{eq:psi:i:j:def} \\
 \psi_{m,i_*(j_m),j_m,q+\bn}(x,t) 
 &= \tilde \gamma_{m,q+\bn}^3 \left( h_{m,j_m,q+\bn}^2(x,t) \right) \qquad i_m=i_*(j_m) \, .
\label{eq:psi:i:i:def}
\end{align}
\end{subequations}
The intermediate cutoff functions $\psi_{m,i_m,j_m,q+\bn}$ are equal to zero for $ i_m < i_*(j_m)$.  
\end{definition}
\noindent Similar to~\cite[equations~(9.7),~(9.8)]{GKN23}, it follows that $\sum_{i_m\geq0} \psi_{m,i_m,j_m,q+\bn}^{2} = \sum_{i_m\geq i_*(j_m)} \psi_{m,i_m,j_m,q+\bn}^{2} = \sum_{\{ i_m \colon \Gamma_{q+\bn}^{i_m} \geq \Gamma_{q+\bn-1}^{j_m} \}} \psi_{m,i_m,j_m,q+\bn}^{2} \equiv 1$
for any $m$, and for $|i_m-i'_m|\geq 2$, $\psi_{m,i_m,j_m,q+\bn}\psi_{m,i_m',j_m,q+\bn}=0$.

\begin{definition}[\bf $m^{\textnormal{th}}$ Velocity Cutoff Function]\label{def:psi:m:im:q:def}
At stage $q+1$ and for $i_m\geq 0$, we define
\begin{equation}\label{eq:psi:m:im:q:def}
\psi_{m,i_m,q+\bn}^{2} = \sum\limits_{\{j_m\colon i_m\geq i_*(j_m)\}} \psi_{j_m,q+\bn-1}^{2} \psi_{m,i_m,j_m,q+\bn}^{2} \, .
\end{equation}
\end{definition}

We shall employ the notation $\Vec{i} =  \{i_m\}_{m=0}^{\NcutSmall} = \left( i_0,...,i_{\NcutSmall} \right) \in \mathbb{N}_0^{\NcutSmall+1}$ to signify a tuple of non-negative integers of length $\NcutSmall+1$.

\begin{definition}[\bf Velocity cutoff function]
\label{def:psi:i:q:def}
At stage $q+1$ and for $0 \leq i \leq i_{\rm max}$,\index{$\psi_{i,q}$}\index{$i$} we define
\begin{align}
    \psi_{i,q+\bn}^{2} = \sum_{\mathcal{I}} \prod_{m=0}^{\NcutSmall} \psi_{m,i_m,q+\bn}^{2} \, , \qquad \mathcal{I} = \left\{\Vec{i}\colon\max\limits_{0\leq m\leq\NcutSmall} i_m =i\right\} \, .
  \label{eq:psi:i:q:recursive}
\end{align}
\end{definition}
\noindent For $\Vec{i}$ as in the sum of \eqref{eq:psi:i:q:recursive}, we shall denote 
\begin{align}
\supp\left( \prod\limits_{m=0}^{\NcutSmall} \psi_{m,i_m,q+\bn}  \right) = \bigcap_{m=0}^{\NcutSmall} \supp(\psi_{m,i_m,q+\bn}) =: \supp (\psi_{\Vec{i},q+\bn} )\,,
\label{eq:new:supp:notation}
\end{align}
so that $(x,t) \in \supp(\psi_{i,q+\bn})$ if and only if there exists $\Vec{i}\in \N_0^{\NcutSmall+1}$ such that $\max_{0\leq m\leq\NcutSmall} i_m =i$, and $(x,t) \in \supp(\psi_{\Vec{i},q+\bn})$.

\subsubsection{Partition of unity, dodging, simple bounds on increments, and spatial derivatives}

\begin{lemma}[\bf Partition of unity]
\label{lem:partition:of:unity:psi}
For all $m$, $\sum_{i_m\geq 0} \psi_{m,i_m,q+\bn}^{2}\equiv 1$ and $\psi_{m,i_m,q+\bn}\psi_{m,i'_m,q+\bn}=0$ for $|i_m-i'_m|\geq 2$.  In addition, we have that $\sum_{i\geq 0} \psi_{i,q+\bn}^{2}\equiv 1$ and $\psi_{i,q+\bn}\psi_{i',q+\bn} \equiv 0$ for $|i-i'|\geq 2$.
\end{lemma}

\begin{proof}[Proof of Lemma~\ref{lem:partition:of:unity:psi}]
The proof of the claims for the intermediate cutoffs are nearly identical to that of~\cite[Lemma~9.5]{GKN23}, and we omit further details. The proof of the first claim for $\psi_{i,\qbn}$ is nearly identical to that of~\cite[Lemma~9.6]{GKN23}, and so we omit further details.  The proof of the second claim for $\psi_{i,\qbn}$ is again essentially identical.
\end{proof}

\begin{lemma}[\bf Lower order derivative bounds on $\hat w_{q+\bn}$]
\label{lem:h:j:q:size}
If $(x,t)\in \supp( \psi_{m,i_m,j_m,q+\bn})$ then
\begin{align}\label{eq:h:psi:supp}
h_{m,j_m,q+\bn}\leq  \Gamma_{q+\bn}^{(m+1)\left(i_m+1-i_*(j_m)\right)}.
\end{align}
Moreover, if $i_m>i_*(j_m)$ we have $h_{m,j_m,q+\bn} \geq  (\sfrac 12) \Gamma_{q+\bn}^{(m+1)(i_m-i_*(j_m))}$ on the support of $\psi_{m,i_m,j_m,q+\bn}$. 
As a consequence, we have that for all $0 \leq m, M \leq \NcutSmall$ and $0 \leq N \leq \NcutLarge$,
\begin{subequations}
\begin{align}
\norm{D^N D_{t,q+\bn-1}^{m} \hat w_{q+\bn}}_{L^\infty( \supp \psi_{m,i_m,q+\bn})}
&\leq \delta_{q+\bn}^{\sfrac 12} \Gamma_{q+\bn}^{i_m+1} (\lambda_{q+\bn} \Gamma_{q+\bn})^N ( {\mu_{q+\bn-1}^{-1}}\Gamma_{q+\bn}^{i_m+3})^{m} \label{eq:derivatives:psi:i:m:q} \\
\norm{D^N D_{t,q+\bn-1}^M \hat w_{q+\bn}}_{L^\infty( \supp \psi_{i,q+\bn})}
&\leq \delta_{q+\bn}^{\sfrac 12} \Gamma_{q+\bn}^{i+1} (\lambda_{q+\bn} \Gamma_{q+\bn})^N ( {\mu_{q+\bn-1}^{-1}} \Gamma_{q+\bn}^{i+3})^{M}  \, .
\label{eq:derivatives:psi:i:q}
\end{align}
\end{subequations}
\end{lemma}
\begin{proof}[Proof of Lemma~\ref{lem:h:j:q:size}]
The proof of the first two estimates uses the definitions of cutoff functions in~\cite[Lemma~5.5]{GKN23} and $h_{m,j_m,q+\bn}$ in \eqref{eq:h:j:q:def}. \eqref{eq:derivatives:psi:i:m:q} follows from~\eqref{eq:psi:m:im:q:def} and~\eqref{eq:h:psi:supp}. The proof of \eqref{eq:derivatives:psi:i:q} follows since we have employed the \textit{maximum} over $m$ of $i_m$ to define $\psi_{i,q+\bn}$ in \eqref{def:psi:i:q:def}.
\end{proof}

\begin{corollary}[\bf Higher order bounds on $\hat w_{q+\bn}$]
\label{cor:D:Dt:wq:psi:i:q}
For $N+M \leq 2\Nfin$ and $i \geq 0$, we have
\begin{align}
&\norm{D^N D_{t,q+\bn-1}^M \hat w_{q+\bn}}_{L^\infty(\supp \psi_{i,q+\bn})} \notag\\
&\qquad \qquad \leq \Gamma_{q+\bn}^{i+1} \delta_{q+\bn}^{\sfrac 12}  (\la_{q+\bn}\Ga_{q+\bn})^N \MM{M,\Nindt,\Gamma_{q+\bn}^{i+3}  {\mu_{q+\bn-1}^{-1}},\Tau_{q+\bn-1}^{-1}\Gamma_{q+\bn-1}} 
 \,.
\label{eq:D:Dt:wq:psi:i:q}
\end{align}
\end{corollary}
\begin{proof}[Proof of Corollary~\ref{cor:D:Dt:wq:psi:i:q}]
When $0 \leq N \leq \NcutLarge, 0\leq M \leq \NcutSmall \leq \Nindt$, \eqref{eq:derivatives:psi:i:q} gives the result. For the remaining cases, either $N>\NcutLarge$ or $M> \NcutSmall$, if $0\leq m \leq \NcutSmall$ and $(x,t) \in \supp \psi_{m,i_m,q+\bn}$, and so there exists $j_m\geq 0$ with $i_*(j_m) \leq i_m$ such that $(x,t) \in \supp \psi_{j_m,q+\bn-1}$. Thus we may appeal to  {Lemma~\ref{lem:mollifying:w}. Using that $i_*(j_m) \leq i_m$ implies $\Gamma_{q+\bn-1}^{j_m} \leq \Gamma_{q+\bn}^{i_m}$, we deduce~\eqref{eq:D:Dt:wq:psi:i:q}
after using that either $N>\NcutLarge$ or $M> \NcutSmall$, \eqref{condi.Ncut0.2}, a large choice of $a$ to absorb the implicit constant in the spare factor of $\Ga_{q+\bn}$, and taking the maximum over $m$ from Definition~\ref{def:psi:i:q:def}.}
\end{proof}

\begin{lemma}[\bf Spatial derivatives for the cutoffs]
\label{lem:sharp:D:psi:i:q}
Fix $q+\bn\geq 1$,  $0 \leq m \leq \NcutSmall$, and $i_m\geq0$. For all $j_m \geq 0$ such that $i_m \geq i_*(j_m)$, all $i\geq 0$, and all  $N \leq \Nfin$, we have 
\begin{align}
{\bf 1}_{\supp( \psi_{j_m,q+\bn-1})} \frac{|D^N \psi_{m,i_m,j_m,q+\bn}|}{\psi_{m,i_m,j_m,q+\bn}^{1 -  N/\Nfin}} 
\les (\la_{q+\bn}\Ga_{q+\bn})^N
\,, \qquad
\frac{|D^N \psi_{i,q+\bn}|}{\psi_{i,q+\bn}^{1-N/\Nfin}} \les (\la_{q+\bn}\Ga_{q+\bn})^N \, .
\label{eq:sharp:D:psi:i:q}
\end{align}
\end{lemma}
\begin{proof}[Proof of Lemma~\ref{lem:sharp:D:psi:i:q}]
The proof of Lemma~\ref{lem:sharp:D:psi:i:q} is nearly identical to the proof of~\cite[Lemma~9.9]{GKN23}.  The only parameter inequality which plays a role is~\eqref{condi.Ncut0.3}, which we note is the analogue of~\cite[(11.14c)]{GKN23}, which is used to absorb a lossy estimate for a high number of derivatives by the extra factors of $\Gamma_\qbn$ present in~\eqref{eq:sharp:D:psi:i:q}; we refer to~\cite[Lemma~9.9]{GKN23} for further details.
\end{proof}

\subsubsection{Maximal index appearing in the cutoff}

\begin{lemma}[\bf Maximal $i$ index in the definition of $\psi_{i,q+\bn}$]
\label{lem:maximal:i}
There exists $\imax = \imax(q+\bn) \geq 0$ such that if $\lambda_0$ is sufficiently large, then\index{$\imax$}
\begin{align}
\psi_{i,q+\bn} &\equiv 0 \quad \mbox{for all} \quad i > i_{\rm max} \, ,\qquad \Gamma_{q+\bn}^{i_{\rm max}} &\leq  \Ga_{q}^{\sfrac{\badshaq}{2}+18} \delta_{q+\bn}^{-\sfrac 12} r_{q}^{-1} \, ,
\qquad 
\imax(q)  &\leq \frac{\badshaq+12}{(b-1)\varepsilon_\Gamma}
\,.
\label{eq:imax:upper:bound:uniform}
\end{align}
\end{lemma}

\begin{proof}[Proof of Lemma~\ref{lem:maximal:i}]
The structure of this proof is identical to that of~\cite[Lemma~9.10]{GKN23}.  Indeed the proof does not change as a result of the substitution of $\mu$ for $\tau$.  It only depends on an inductive assumption for $\imax(q-1)$, and an inequality in which factors of $\mu$ or $\tau$ are balanced on both sides and therefore cancel; see~\cite[(9.37)]{GKN23}.  We therefore omit further details.
\end{proof}

\subsubsection{Mixed derivative estimates}
\newcommand{\Dqp}{D_{q+\bn}}
We set $D_{q+\bn} = \hat w_{q+\bn} \cdot \nabla$, so that $D_{t,q+\bn} = D_{t,q+\bn-1} + D_{q+\bn}$. Then we recall from \cite[equations~(6.54)-(6.55)]{BMNV21} that
\begin{align}
D_{q+\bn}^K = \sum_{j=1}^{K} f_{j,K} D^j \, ,
\qquad 
\textnormal{ where } \quad f_{j,K} = \sum_{\{ \gamma \in \N^K \colon |\gamma|=K-j \} } c_{j,K,\gamma} \prod_{\ell=1}^K D^{\gamma_{\ell}} \hat w_{q+\bn} \, .
\label{eq:D:q:K:ii}
\end{align}
The  $c_{j,K,\gamma}$'s are computable and depend only on $K,j$, and $\gamma$.

\begin{lemma}[\bf Bounds for $D_{q+\bn}^K$]
\label{lem:D:a:fj}
For $q+\bn\geq 1$, $1 \leq K \leq  2\Nfin$, any $a \leq 2 \Nfin-K+j$, and any $0\leq i\leq \imax(q+\bn)$, the functions $\{ f_{j,K}\}_{j=1}^{K}$ defined in \eqref{eq:D:q:K:ii} obey the estimate
\begin{align}
\norm{D^a f_{j,K}}_{L^\infty (\supp \psi_{i,q+\bn})} &\les  (\Gamma_{q+\bn}^{i+1} \delta_{q+\bn}^{\sfrac 12})^K (\la_{q+\bn}\Ga_{q+\bn})^{a+K-j} \, .
\label{eq:D:a:fj} 
\end{align}
\end{lemma}
\begin{proof}[Proof of Lemma~\ref{lem:D:a:fj}]
The proof mirrors that of~\cite[Lemma~9.11]{GKN23} and follows from Corollary~\ref{cor:D:Dt:wq:psi:i:q} with $M=0$ and Lemma~\ref{lem:mollifying:w}; we omit further details.
\end{proof}

\begin{lemma}[\bf Mixed derivatives for $\hat w_\qbn$]
\label{lem:Dt:Dt:wq:psi:i:q}
For $q+\bn\geq 1$ and $0 \leq i \leq \imax$, we have that 
\begin{align*}
&\norm{D^N D_\qbn^K D_{t,\qbn-1}^M \hat w_\qbn}_{L^\infty(\supp \psi_{i,\qbn})} \notag\\
&\qquad \qquad \les  
(\Gamma_{\qbn}^{i+1} \delta_\qbn^{\sfrac 12})^{K+1} 
(\la_{q+\bn}\Ga_\qbn)^{N+K}
\MM{M,\NindSmall,\Gamma_{\qbn}^{i+3}   {\mu_{\qbn-1}^{-1}},\Tau_{\qbn-1}^{-1}\Ga_{\qbn-1}}\notag\\
&\qquad \qquad \les  
(\Gamma_{\qbn}^{i+1} \delta_\qbn^{\sfrac 12}) 
(\la_\qbn\Ga_\qbn)^N (\Gamma_\qbn^{i-5}  {{\mu_{\qbn}^{-1}} })^{K}  \MM{M,\NindSmall,\Gamma_\qbn^{i+3}    {\mu_{\qbn-1}^{-1}},\Tau_{\qbn-1}^{-1}\Ga_{\qbn-1}}
\end{align*}
holds for $0 \leq K + N  + M \leq 2\Nfin$.
\end{lemma}

\begin{proof}[Proof of Lemma~\ref{lem:Dt:Dt:wq:psi:i:q}]
The second estimate follows from~\eqref{ineq:tau:q}. In order to prove the first estimate, we use~\eqref{eq:D:Dt:wq:psi:i:q},~\eqref{eq:D:q:K:ii}, and~\eqref{eq:D:a:fj}.  The proof mirrors that of~\cite[Lemma~9.12]{GKN23}, the first difference being that no factors of $r_q^{-\sfrac 13}$ appear, and the second difference being that we save one factor of $\delta_\qbn^{\sfrac 12} \Gamma_\qbn^{i+1}$ out in front while converting all other such factors to $\mu_\qbn^{-1}\Gamma_\qbn^{i-5}$ instead of $\tau_\qbn^{-1}\Gamma_\qbn^{i-5}$. We omit further details.
\end{proof}

\begin{lemma}[\bf More mixed derivatives for $\hat w_\qbn$ and derivatives for $\hat u_\qbn$]
\label{lem:Dt:Dt:wq:psi:i:q:multi}
For $q+\bn\geq 1$, $k\geq 1$, $\alpha,\beta \in {\mathbb N}^k$ with $|\alpha| = K$, $|\beta| = M$, and $K+M\leq \sfrac{3\Nfin}{2}+1$, we have
\begin{align}
&\norm{ \Big( \prod_{i=1}^k D^{\alpha_i} D_{t,\qbn-1}^{\beta_i} \Big) \hat w_\qbn }_{L^\infty(\supp \psi_{i,\qbn})} \notag\\
& \qquad  \lesssim \Gamma_\qbn^{i+1} \delta_\qbn^{\sfrac 12} (\la_\qbn\Ga_\qbn)^K \MM{M,\NindSmall,\Gamma_\qbn^{i+3}   {\mu_{\qbn-1}^{-1}}, \Gamma_{\qbn-1} \Tau_{\qbn-1}^{-1}}   \, . \label{eq:nasty:D:wq}
\end{align}
Next, we have that
\begin{subequations}
\begin{align}
&\norm{D^N \Big( \prod_{i=1}^k D_\qbn^{\alpha_i} D_{t,\qbn-1}^{\beta_i} \Big) \hat w_\qbn }_{L^\infty(\supp \psi_{i,\qbn})}\notag\\
&\qquad\lesssim   (\Gamma_\qbn^{i+1} \delta_\qbn^{\sfrac 12})^{K+1} (\la_\qbn\Ga_\qbn)^{N+K} \MM{M,\NindSmall,\Gamma_\qbn^{i+3}    {\mu_{\qbn-1}^{-1}},\Gamma_{\qbn-1} \Tau_{\qbn-1}^{-1}}  \label{eq:nasty:Dt:wq} \\
&\qquad  
\lesssim  \Gamma_\qbn^{i+1} \delta_\qbn^{\sfrac 12} (\la_\qbn\Ga_\qbn)^N (\Gamma_\qbn^{i-5}   {\mu_\qbn^{-1}})^{K} \MM{M,\NindSmall,\Gamma_\qbn^{i+3}    {\mu_{\qbn-1}^{-1}},\Gamma_{\qbn-1} \Tau_{\qbn-1}^{-1}} 
\label{eq:nasty:Dt:wq:WEAK}
\end{align}
\end{subequations}
holds for all  $0 \leq K + M + N \leq \sfrac{3\Nfin}{2}+1$. Lastly, we have the estimate 
\begin{align}
\norm{ \Big( \prod_{i=1}^k D^{\alpha_i} D_{t,\qbn}^{\beta_i} \Big) D \hat u_\qbn }_{L^\infty(\supp \psi_{i,\qbn})}&\lesssim \Ga_\qbn^{i+2} \delta_\qbn^{\sfrac 12} \lambda_\qbn  (\la_{\qbn}\Ga_\qbn)^K\notag
\\
&\qquad \MM{M,\NindSmall,\Gamma_\qbn^{i-5}    {\mu_{\qbn}^{-1}},\Gamma_{\qbn-1} \Tau_{\qbn-1}^{-1}} 
\label{eq:nasty:D:vq}
\end{align}
for all  $ K + M \leq \sfrac{3\Nfin}{2}$, the estimate 
\begin{align}
& \norm{ \Big( \prod_{i=1}^k D^{\alpha_i} D_{t,\qbn}^{\beta_i} \Big)  \hat u_\qbn }_{L^\infty(\supp \psi_{i,\qbn})} \notag\\
&\qquad \lesssim  \Gamma_\qbn^{i+2} \delta_\qbn^{\sfrac 12} \lambda_\qbn^{2} (\la_\qbn\Ga_\qbn)^K \MM{M,\NindSmall,\Gamma_\qbn^{i-5}    {\mu_{\qbn}^{-1}},\Gamma_{\qbn-1} \Tau_{\qbn-1}^{-1}} 
\label{eq:nasty:no:D:vq}
\end{align}
for all  $ K + M \leq \sfrac{3\Nfin}{2}+1$, and the estimate
\begin{align}
    \left\| D^K \partial_t^M \hat u_\qbn \right\|_\infty \leq \la_\qbn^{\sfrac 12} (\la_\qbn\Ga_\qbn)^K \Tau_{\qbn}^{-M} \label{eq:bobby:new}
\end{align}
for all $K+M\leq 2\Nfin$.
\end{lemma}

\begin{proof}[Proof of Lemma~\ref{lem:Dt:Dt:wq:psi:i:q:multi}]

We note that \eqref{eq:nasty:Dt:wq:WEAK} follows directly from \eqref{eq:nasty:Dt:wq} by appealing to \eqref{ineq:tau:q}. The proof then proceeds by proving~\eqref{eq:nasty:D:wq}, then~\eqref{eq:nasty:Dt:wq}, and finally~\eqref{eq:nasty:D:vq}--\eqref{eq:bobby:new}.
\smallskip

\noindent{\bf Proof of \eqref{eq:nasty:D:wq}.\,}
The statement is proven by induction on $k$ and mirrors the proof of~\cite[(9.42)]{GKN23}, the first difference being that no factors of $r_q^{-\sfrac 13}$ appear, and the second difference being the substitution of $\mu$ for $\tau$. We omit further details.
\smallskip

\noindent{\bf Proof of \eqref{eq:nasty:Dt:wq}.\,} 
To prove this estimate, we use~\cite[Lemma~A.6]{GKN23}; since the proof is nearly identical to~\cite[(9.43a)]{GKN23}, we omit further details.
\smallskip

\noindent{\bf Proofs of \eqref{eq:nasty:D:vq}--\eqref{eq:nasty:no:D:vq}.\,}
The proofs of these estimates use the parameter inequality~\eqref{ineq:tau:q},~\cite[Lemma~A.6]{GKN23}, and induction on $k$.  Since the method is identical to~\cite[(9.44)--(9.45)]{GKN23}, we omit further details.
\smallskip

\noindent\textbf{Proof of \eqref{eq:bobby:new}.\,} The proof of this bound is immediate from Lemma~\ref{lem:mollifying:w}, the definition of $\hat w_\qbn$ in Definition~\ref{def:wqbn}, the inductive assumption \eqref{eq:bobby:old}, and the triangle inequality.
\end{proof}

\subsubsection{Material derivatives}
\begin{lemma}[\bf Mixed spatial and material derivatives for velocity cutoffs]
\label{lem:sharp:Dt:psi:i:q}
Let $q+\bn\geq 1$, $0 \leq i \leq i_{\rm max}(q+\bn)$, $N,K,M,k \geq 0$, and let $\alpha,\beta \in {\mathbb N}^k$ be such that $|\alpha|=K$ and $|\beta|=M$.  Then
\begin{align}
&\frac{1}{\psi_{i,q+\bn}^{1- (K+M)/\Nfin}} \left|\left(\prod_{l=1}^k D^{\alpha_l} D_{t,q+\bn-1}^{\beta_l}\right) \psi_{i,q+\bn}\right| \notag\\
&\qquad \les (\la_\qbn\Ga_\qbn)^K
\MM{M,\Nindt-\NcutSmall,\Gamma_{{q+\bn}}^{i+3}   {\mu_{q+\bn-1}^{-1}}, \Gamma_{\qbn+1} \Tau_{\qbn-1}^{-1}}
\label{eq:sharp:Dt:psi:i:q} \\
&\frac{1}{\psi_{i,\qbn}^{1- (N+K+M)/\Nfin}} \left| D^N \left( \prod_{l=1}^k D_\qbn^{\alpha_l} D_{t,\qbn-1}^{\beta_l}\right)  \psi_{i,\qbn} \right| \notag\\
&\qquad \les  (\la_\qbn\Ga_\qbn)^N
(\Gamma_\qbn^{i-5} \tau_\qbn^{-1})^K 
\MM{M,\Nindt-\NcutSmall,\Gamma_{{q+\bn}}^{i+3}   {\mu_{q+\bn-1}^{-1}}, \Gamma_{\qbn+1} \Tau_{\qbn-1}^{-1}}
\label{eq:sharp:Dt:psi:i:q:mixed}
\end{align}
where the first bound holds for $K + M \leq \Nfin$, and the second for $N+ K+ M \leq \Nfin$.
\end{lemma}

\begin{proof}[Proof of Lemma~\ref{lem:sharp:Dt:psi:i:q}]
The method of proof is identical to~\cite[Lemma~9.16]{GKN23} and only requires substituting squared cutoff functions for cutoff functions raised to the sixth power, and replacing $\tau$ with $\mu$; we omit further details.
\end{proof}

\subsubsection{\texorpdfstring{$L^r$}{Lr} size of the velocity cutoffs}
We now show that~\eqref{eq:psi:i:q:support:old} holds with $q'=\qbn$. 
\begin{lemma}[\bf Support estimate]
\label{lem:psi:i:q:support}
For all $0 \leq i \leq \imax(\qbn)$ and $1\leq r \leq \infty$, we have that $\|\psi_{i,\qbn}\|_r  \lesssim  \Gamma_\qbn^{\frac{-2i+\CLebesgue}{r}}$ where $ \CLebesgue$ is defined in \eqref{eq:psi:i:q:support:old} and thus depends only on $b$. 
\end{lemma}

\begin{proof}[Proof of Lemma~\ref{lem:psi:i:q:support}]
The proof follows exactly the method of~\cite[Lemma~9.17]{GKN23}.  The main difference is that squares are substituted for cubes, and the definition of $\CLebesgue$ from~\eqref{eq:psi:i:q:support:old} has been adjusted accordingly.  We therefore omit further details and refer the reader to~\cite[Lemma~9.17]{GKN23}.
\end{proof}

\subsubsection{Verifying Eqn.~\texorpdfstring{\eqref{eq:inductive:timescales}}{eqinductive}}\label{ss:dodging:verification}
\begin{lemma}[\bf Overlapping and timescales]\label{lem:overlap:timescales}
Let $q''\in\{q+1,\dots,q+\bn-1\}$. Assume that $\psi_{i,q+\bn}\psi_{i'',q''}\not\equiv 0$. Then it must be the case that  {$\mu_{q+\bn}\Gamma_{q+\bn}^{-i}\leq \mu_{q''}\Gamma_{q''}^{-i''-{25}}$}.
\end{lemma}
\begin{proof}[Proof of Lemma~\ref{lem:overlap:timescales}]
The proof of this Lemma is identical to~\cite[Lemma~9.18]{GKN23}, save for substituting $\mu$ for $\tau$, and so we omit further details.
\end{proof}

\medskip

\noindent\textsc{Department of Mathematics, University of Z\"urich, Z\"urich, Switzerland.}
\vspace{.03in}
\newline\noindent\textit{Email address}: \href{mailto:vikram.giri@math.uzh.ch}{vikram.giri@math.uzh.ch}.
\smallskip

\noindent\textsc{Department of Mathematics, ETH Z\"urich, Z\"urich, Switzerland.}
\vspace{.03in}
\newline\noindent\textit{Email address}: \href{hyunju.kwon@math.ethz.ch}{hyunju.kwon@math.ethz.ch}.
\smallskip

\noindent\textsc{Department of Mathematics, Purdue University, West Lafayette, IN, USA.}
\vspace{.03in}
\newline\noindent\textit{Email address}: \href{mailto:mdnovack@purdue.edu}{mdnovack@purdue.edu}.

\end{document}